\documentclass[oneside]{amsart}
\usepackage[utf8]{inputenc}
\usepackage[percent]{overpic}
\usepackage{amsmath}
\usepackage{amssymb}
\usepackage{tikz}
\usepackage{tikz-cd}
\usepackage{mathtools}
\usepackage{stmaryrd}
\usepackage{amsthm}
\usepackage{array} 
\usepackage[english]{babel}
\usepackage{enumitem}
\usepackage{lmodern}
\usepackage{xfrac}  
\usepackage{comment}

\usepackage{csquotes}
\usepackage[style=alphabetic, maxalphanames=99, maxbibnames=99]{biblatex}
\usepackage[paper=a4paper,  margin=3cm]{geometry}
\definecolor{cobalt}{rgb}{0.0, 0.28, 0.67}
\usepackage[hidelinks]{hyperref}
\hypersetup{
    pdftitle={Combinatorial K-Theory},    
    pdfauthor={Dupraz, M.; Gross, A.; Monin, L.},     
    pdfkeywords= {Poincare duality; Frobenius rings; Shift operators; Differential operators},
    pdfsubject= {},
    colorlinks=true,       
    linkcolor=cobalt,          
    citecolor=cobalt,        
}

\usepackage{cleveref}

\theoremstyle{plain}
\newtheorem{theorem}{Theorem}[section]
\newtheorem{lemma}[theorem]{Lemma}
\newtheorem{proposition}[theorem]{Proposition}
\newtheorem{corollary}[theorem]{Corollary}

\theoremstyle{definition}
\newtheorem{definition}[theorem]{Definition}
\newtheorem{remark}[theorem]{Remark}

\newtheorem{example}[theorem]{Example}
\newtheorem{question}[theorem]{Question}
\newtheorem*{warning}{Warning}

\newcommand{\rleft}{\mathopen{}\mathclose\bgroup\left}
\newcommand{\rright}{\aftergroup\egroup\right}

\newcommand{\C}{{\mathbb{C}}}

\newcommand{\R}{{\mathbb{R}}}
\newcommand{\Z}{{\mathbb{Z}}}
\newcommand{\N}{{\mathbb{N}}}
\newcommand{\p}{{\mathbb{P}}}
\newcommand{\Q}{{\mathbb{Q}}}

\newcommand{\Om}{{\mathcal{O}}}

\newcommand{\cL}{{\mathcal{L}}}
\newcommand{\cO}{{\mathcal{O}}}
\newcommand{\cA}{{\mathcal{A}}}

\newcommand{\cP}{{\mathcal{P}}}
\newcommand{\cV}{{\mathcal{V}}}
\newcommand{\cS}{{\mathcal{S}}}
\newcommand{\cQ}{{\mathcal{Q}}}

\newcommand{\Ann}{{\mathrm{Ann}}}
\newcommand{\Fun}{{\mathrm{Fun}}}
\newcommand{\rk}{{\mathrm{rk}}}

\newcommand{\Diff}{{\mathrm{Diff}}}

\newcommand{\SL}{{\mathrm{SL}}}

\newcommand{\Gr}{\operatorname{Gr}}
\newcommand{\gr}{\operatorname{gr}}
\newcommand{\Hom}{{\mathrm{Hom}}}

\newcommand{\Pic}{{\mathrm{Pic}}}

\newcommand{\im}{{\mathrm{im}}}
\newcommand{\Sym}{{\mathrm{Sym}}}
\newcommand{\Vol}{{\mathrm{Vol}}}

\newcommand{\Ehr}{{\mathrm{Ehr}}}
\newcommand{\Snap}{{\mathrm{Snap}}}
\newcommand{\Sh}{{\mathrm{Sh}}}
\newcommand{\ch}{{\mathrm{ch}}}
\newcommand{\ex}{{\mathrm{ex}}}
\newcommand{\td}{{\mathrm{td}}}
\newcommand{\Td}{{\mathrm{Td}}}

\newcommand{\Poch}{{\mathrm{Poch}}}
\newcommand{\PL}{{\mathrm{PL}}}
\newcommand{\PE}{{\mathrm{PE}}}
\newcommand{\mc}[1]{{\mathcal{#1}}}
\newcommand{\num}{{\mathrm{num}}}
\newcommand{\MW}{{\mathrm{MW}}}

\newcommand{\psibar}{\overline{\psi}}

\renewcommand{\vec}[1]{\ensuremath{\mathbf{#1}}}

\usepackage{pgf,tikz}
\usetikzlibrary{arrows}

\title{A Riemann-Roch theorem for Frobenius quotients}
\author{Matthew Dupraz
    \and
    Andreas Gross
    \and
    Leonid Monin
}

\newcommand{\Addresses}{{
  \bigskip
  \footnotesize

  Matthew~Dupraz, \textsc{Institut für Mathematik, Freie Universität Berlin}\par\nopagebreak
  \textit{E-mail address}: \texttt{matthew.dupraz@fu-berlin.de}

  \medskip

  Andreas Gross, \textsc{Institut für Mathematik, Johann Wolfgang Goethe-Universität}\par\nopagebreak
  \textit{E-mail address}: \texttt{gross@math.uni-frankfurt.de}
  
  \medskip

  Leonid~Monin, \textsc{Institut de Math\'ematiques, Ecole Polytechnique F\'ed\'erale de Lausanne}\par\nopagebreak
  \textit{E-mail address}: \texttt{leonid.monin@epfl.ch}

}}


\begin{document}

\begin{abstract}
We construct two families of commutative Frobenius rings: \emph{discrete} and \emph{continuous} Frobenius quotients $K_f$, resp.\ $A_g$, which are defined as the quotients of shift, resp.\ differential operator algebras by the annihilator of a polynomial. These constructions model the numerical $K$-rings and Chow rings of smooth complete varieties. We show that there is a naturally defined isomorphism $\Q K_f \cong \Q A_g$ playing the role of the Chern character, precisely when the polynomials satisfy a combinatorial analogue of the Hirzebruch--Riemann--Roch theorem, which states that there exists an invertible element $\td \in \Q A_g^\times$, called the \emph{Todd class}, such that $f = \td \cdot g$. In this case we show that $K_f$ carries all the structure needed to make it a suitable model for $K$-rings of complete smooth varieties: $K_f$ is a $\lambda$-ring, has well-defined Chern classes, determinants, and satisfies a combinatorial analogue of Serre duality.
We further show that the construction is functorial and obtain a combinatorial analogue of the Grothendieck-Riemann-Roch theorem.
Lastly, we investigate the existence of larger families of isomorphisms between Frobenius quotients defined by the action of a power series on a generating set, such as the truncated Chern character which yields an integral isomorphism $K_f \cong A_g$.

We provide numerous examples and applications: we give a new formula for computing Snapper polynomials of matroids, show that the dualizing class of $K_f$ coincides with dualizing classes of matroids and linear families of polytopes, realize $K$-rings of toric variety bundles as Frobenius quotients, and study $K$-rings of Ehrhart fans as well as exceptional isomorphisms in this setting. 
\end{abstract}

\maketitle


\section{Introduction}
\label{sec:intro}

Since the early days of intersection theory, polynomials played a central role in encoding algebraic information. Van der Waerden \cite{VanDerWaerden} made the connection between Hilbert polynomials and B\'ezout's theorem. Later, this was generalized by Snapper \cite{Snapper} and Kleiman \cite{KleimanSnapper} to define intersection numbers using Snapper polynomials, which capture the holomorphic Euler characteristics of various products of line bundles. With the introduction of $K$-theory, the holomorphic Euler characteristic can be viewed as a natural functional on the $K$-ring. The connection back to intersection numbers can then be expressed via the celebrated Hirzebruch-Riemann-Roch theorem (HRR) \cite{hirzebruch1956neue}, which states that the holomorphic Euler characteristic $\chi: K_\num(X) \to \Z$ on the numerical $K$-ring $K_\num(X)$ of a smooth complete variety $X$ is related to the degree $\deg: A_\num(X) \to \Z$ on the numerical Chow ring $A_\num(X)$ by the equality
\[
    \chi(\xi)= \deg(\td(X)\cdot\ch(\xi))
\]
for all $\xi \in \Q K_\num(X)$, where $\ch: \Q K_\num(X) \to \Q A_\num(X)$ is the Chern character and $\td(X) \in \Q A_\num(X)$ is the Todd class of $X$. Let $\Lambda$ be a lattice and let $\Lambda \to \Pic(X), u \mapsto \cL_u$ be a homomorphism. Then, using HRR,  Snapper and Kleiman's theorem on the polynomiality of the Snapper function 
\[
\Snap\colon \Lambda\to \Z, \; u \mapsto \chi(\cL_u)
\]
follows directly, as does the fact that its top homogeneous part is of degree $d=\dim(X)$ and equal to the volume polynomial
\[
\Vol\colon \Lambda\to \Q,\; u \mapsto \frac{1}{d!}\deg(c_1(\cL_u)^d).
\]

If $X$ is a complete toric variety, the Euler characteristic of an ample line bundle can be expressed as a lattice point count of a lattice polytope. This observation equates the Snapper polynomials of line bundles on $X$ with the Ehrhart polynomials from combinatorics. A special case of this is when $X$ is the permutahedral variety, which can be associated to a uniform matroid of maximal rank. In the seminal work of Adiprasito, Huh, and Katz \cite{CombinatorialHodge}, it was shown that the Chow rings of the (usually incomplete) toric varieties associated to arbitrary matroids enjoy many of the appealing properties of Chow rings of complete varieties. The polynomials that arise from this perspective have likewise received much attention \cite{NormalComplexes, LLPP24, EhrhartFans}.

If $K_\num(X)$ and $A_\num(X)$ are the numerical $K$-ring and Chow ring of a smooth complete variety
$X$ (or of a matroid), then both $K_\num(X)$ and $A_\num(X)$ are commutative Frobenius rings: the Euler characteristic $\chi: K_\num(X) \to \Z$ and the degree map $\deg: A_\num(X) \to \Z$ each induce non-degenerate pairings on these rings. 
Using variants of Macaulay's inverse systems \cite{InverseSystem}, which we call \emph{Frobenius quotients}, it follows that both $A_\num(X)$ and $K_\num(X)$ are in fact fully determined by the Snapper and volume polynomial, respectively.

The \emph{discrete Frobenius quotient} $K_f$ is the quotient of the ring of shift operators on an affine semigroup $\Sigma$ by the ideal annihilating a function $f\colon \Sigma\to \Q$:
\[K_f \coloneqq \Sh(\Sigma)/\Ann_{\Sh}(f).\]
This ring is a commutative Frobenius ring with non-degenerate pairing induced by
$\chi_f([E]) = (E \cdot f)(0)$ for all $[E] \in K_f$. Moreover, all commutative Frobenius rings are of this form \cite{MoninSmirnov23}. In particular, if we choose for $\Sigma\subseteq K_\num(X)$ a multiplicative affine sub-semigroup generating the numerical $K$-ring of $X$, then $K_\num(X) \cong K_\Snap$ for the Snapper function $\Snap: \Sigma \to \Z,\; \xi \mapsto \chi(\xi)$. 

Similarly, if $\Gamma$ is a graded lattice with a weighted homogeneous polynomial $g: \Gamma \to \Q$,
the \emph{continuous Frobenius quotient} $A_g$ is the quotient of the ring of differential operators with constant coefficients by the ideal annihilating $g$:
\[A_g \coloneqq \Diff(\Gamma)/\Ann_{\Diff}(g).\]
This is a graded Frobenius ring (i.e.\ a Poincaré duality algebra when working over a field) with non-degenerate pairing induced by
$\deg_g([D]) = (D \cdot g)(0)$ for all $[D] \in A_g$.
Moreover, all graded Frobenius rings are of this form \cite{Kaveh11, khovanskii2021gorenstein}. In particular, if we choose $\Gamma\subset A_\num(X)$ which generates $A_\num(X)$ as a ring, then $A_\num(X)=A_\Vol$ for the volume polynomial $\Vol\colon \Gamma\to \Q,\;\alpha\mapsto\deg(\exp(\alpha))$.

In this paper we study pairs $(K_f, A_g)$ of such rings. We provide criteria for when the rings are related by an isomorphism mimicking the Chern character from algebraic geometry, and show that in this case many structures we would expect from algebraic geometry transfer naturally to this setting.

The richest situation arises when $\Sigma = \Gamma = \Lambda$ is a lattice and $f$ and $g$ are polynomials on
$\Lambda$. We call this the \emph{standard graded setting}, because in this case $A_g$ is a standard graded ring.
In this setting, there is a natural candidate for the Chern character. Indeed, if $T^u$ is the shift by $u$, that is $(T^u h)(x)=h(x+u)$ for a function $h$ on $\Lambda$, and $\partial_u$ is the directional derivative in the direction of $u$, then 
\[
T^u h =\exp(\partial_u) \cdot h
\]
for every polynomial $h$ (see Lemma \ref{lem:exponential and shifts}). Therefore, the assignment
\[T^u \mapsto \exp(\partial_u)\]
not only mimics the classical definition of the Chern character in topology and geometry, but is also compatible with our interpretation of $K$- and Chow-rings via operators.
Theorem~\ref{intro_thm:Chern character descends} provides criteria for this map to be a well-defined isomorphism, and Theorems \ref{intro_thm:Operators}-\ref{intro_thm:Serre-duality} show that when we have a Chern character, we also recover many structures one has in algebraic geometry.

\begin{theorem}
    \label{intro_thm:Chern character descends}
    Let $\Lambda$ be a lattice and let $f$ and $g$ be polynomials on $\Lambda$ with $g$ homogeneous.
    Then the assignment $T^u \mapsto \exp(\partial_u)$ defines an isomorphism
    \[
        \ch \colon \Q K_f\to \Q A_g,
    \]
    if and only if there exists a unit $\td \in \Q A_g^\times$, called the \emph{Todd class} of $(K_f, A_g)$, such that $\td \cdot g=f$, or equivalently such that
    \[
        \chi_f(\xi)= \deg_g(\td\cdot\ch(\xi))
    \]
    for all $\xi\in \Q K_f$.
\end{theorem}
We show a more general statement in \Cref{thm:map-descends-iff-hrr}; we refer to \Cref{sec:rings-with-duality} for details. When any of the equivalent conditions of \Cref{intro_thm:Chern character descends} is satisfied, we say $f$ and $g$ are \emph{differentially equivalent}, which we denote by $f \sim g$.
If that is the case, we call the pair $(K_f, A_g)$ a \emph{Riemann-Roch pair}. As the next results show,
such pairs carry much of the structure one expects from algebraic geometry.

\begin{theorem}
\label{intro_thm:Operators}
Let $(K_f, A_g)$ be a Riemann-Roch pair, then there exist:
\begin{enumerate}
    \item \emph{Chern classes}, that is, a group homomorphism
    \[
        c\colon K_f \to A_g^\times
    \]
    with $c(T^u)=1+\partial_u$ for all $u\in \Lambda$,
    
    \item a \emph{$\lambda$-ring structure} on $K_f$, that is, a group homomorphism
    \[
        \lambda \colon K_f\to K_f\llbracket t\rrbracket^\times
    \]
    with $\lambda(T^u)=1+T^u t$ for all $u\in \Lambda$,

    \item a \emph{determinant map}, that is, a group  homomorphism
    \[
        \det \colon K_f\to K_f^1 \ ,
    \]
    where $K_f^1$ is the image of $\Lambda$ in $K_f^\times$ and $\det(T^u)=T^u$ for all $u\in \Lambda$,

    \item a \emph{Todd operator}, that is, a group homomorphism
    \[
        \Td\colon K_f\to A_g^\times
    \]
    satisfying
    \[
        \Td(T^u)= \frac{\partial_u}{1-\exp(-\partial_u)}
    \]
    for all $u\in \Lambda$.
\end{enumerate}
\end{theorem}

Theorem~\ref{intro_thm:Operators} relies on a general construction of multiplicative maps defined on $K_f$ given in  \Cref{thm:mult-map-A_g}. The Chern classes and Todd operator are defined in \Cref{def:chern-class-todd}, and determinants and $\lambda$-operations are defined in \Cref{def:det-lambda}. Finally, the fact that $\lambda$-operations give $K_f$ a $\lambda$-ring structure is proven in \Cref{thm:lambda-ring}.

We further show the existence of the dualizing morphism and use it to give a combinatorial version of \emph{Serre duality} in our setting.
\begin{theorem}[{\Cref{prop:dualizing-morphism} and \Cref{prop:serre duality}}]\label{intro_thm:Serre-duality}
    Let $(K_f, A_g)$ be a Riemann-Roch pair. Then there is a \emph{dualizing morphism}, that is, an involution
    \[
        (-)^\vee \colon K_f\to K_f
    \]
    with $(T^u)^\vee=T^{-u}$ for all $u\in \Lambda$.
    Furthermore, there is an element $\omega \in \Q K_f$ called the \emph{dualizing class}, uniquely determined by the relation
    \[
        \chi_f(\xi) = (-1)^d\chi_f(\omega \cdot \xi^\vee)
    \]
    for all $\xi\in \Q K_f$.
    If there exists a \emph{cotangent class}, that is, an element $\Omega\in K_f$ with $\td=\Td(\Omega^\vee)$,
    then we have 
    \[
    \omega=\det(\Omega).
    \]
\end{theorem}
Besides agreeing with the canonical line bundle when $K_f$ is the numerical $K$-ring of a smooth variety, this also provides a generalization of the notion of dualizing class for a linear family of polytopes (see \Cref{subsec:linear-families}) and of matroids, in the sense of \cite{Cheng2025,cheng2026tangent}.

We also define a filtration on $K_f$ via the kernel of the rank map, which is the unique ring homomorphism sending $T^u \mapsto 1$, and show that the associated graded ring $\Q \gr K_f$ is isomorphic to $\Q A_g$
(\Cref{thm:ass-graded}). This generalizes the classical fact that over the rational numbers, the associated graded ring of the topological filtration on the $K$-ring of a smooth variety is isomorphic to its Chow ring, and shows that $\Q A_g$
can be recovered from $\Q K_f$ by a purely ring-theoretic construction.

Finally, we study another class of isomorphisms between $\Q K_f$ and $\Q A_g$,
defined by their action on a fixed basis $u_1, \dots, u_n$ of $\Lambda$, which we call \emph{exceptional isomorphisms}.
In particular, we consider the truncated Chern character, which sends $T^{u_i} \mapsto 1 + \partial_{u_i}$, and characterize when it gives rise to an integral isomorphism $K_f \cong A_g$.
\begin{theorem}[{\Cref{thm:exceptional-general} and \Cref{prop:exceptional and dual exceptional}}]
    The map sending $T^{u_i} \mapsto 1 + \partial_{u_i}$ (resp.\ $T^{u_i}\mapsto \frac{1}{1-\partial_{u_i}}$) for all $i$ defines an isomorphism $K_f \cong A_g$ if and only if
    $f \sim \Poch^\downarrow(g)$ (resp.\ $f \sim \Poch^\uparrow(g)$), where the Pochhammer transform $\Poch^\downarrow(g)$ (resp. $\Poch^\uparrow(g)$) is the polynomial obtained from $g$ by replacing powers with falling (resp.\ rising) factorials. 
\end{theorem}
We give a concrete application of this theorem for matroids. 
Let $M$ be a matroid, and $\Lambda$ the lattice with basis $u_F$, for all the flats $F$ of $M$.
The authors of \cite{LLPP24} introduced natural elements $[\cL_F]$ in the $K$-ring $K(M)$ and $c_1(\cL_F)\in A(M)$ in the Chow ring $A(M)$ (here, we use the notation from the case where $M$ is realizable, but the elements exist in general). $K(M)$ and $A(M)$ are Frobenius rings which are generated by these elements, so we can represent them as $K_{\Snap_M}$, resp. $A_{\Vol_M}$, where $\Snap_M$ and $\Vol_M$ are polynomials on $\Lambda$. Using the description of the exceptional Todd class of \cite{LLPP24}, we obtain the following identity.

\begin{corollary}[\Cref{cor:wonderful-compactification-exceptional}]
There is an equality of polynomials
\[
    \Snap_M = T^{u_E} \cdot \Poch^\uparrow(\Vol_M),
\]    
where $E$ is the ground set of $M$ and the Pochhammer transform is taken with respect to the basis $\{u_F\}_{F \in \cL(M)}$.
\end{corollary}
In particular, this provides an efficient way to compute the Snapper polynomial of a matroid.

\medskip

Riemann-Roch pairs are also well-behaved with respect to lattice homomorphisms. Given a lattice homomorphism $\phi: \Lambda_1 \to \Lambda_2$, under certain conditions, we obtain a pushforward map $\phi_*: K_{f_1} \to K_{f_2}$ and a pullback map $\phi^*: K_{f_2} \to K_{f_1}$, and similarly for the continuous Frobenius quotients. In \Cref{thm:grr} we show that the pullbacks\footnote{Here the \emph{pullback} of the lattice homomorphism would correspond to the \emph{pushforward} of the corresponding map of smooth algebraic varieties in the algebro-geometric setting. See a warning in the beginning of \Cref{sec:grr} for more details.} satisfy a combinatorial analogue of the Grothendieck-Riemann-Roch theorem. That is, the following diagram is commutative:
\[
\begin{tikzcd}
\Q K_{f_2} \arrow[d, "\td_2\cdot \ch(-)"'] \arrow[r, "\phi^*"] & \Q K_{f_1} \arrow[d, "\td_1\cdot\ch(-)"] \\
\Q A_{g_2} \arrow[r, "\phi^*"]                          & \Q A_{g_1}.                    
\end{tikzcd}
\]

\medskip

The paper is structured as follows. In \Cref{sec:rings-with-duality} we study discrete and continuous Frobenius quotients and give a criterion for the existence of a Chern character connecting the two.

\Cref{sec:standard} focuses on the standard graded case where $\Sigma=\Gamma=\Lambda$ and develops the structural properties like the existence of a $\lambda$-structure, Chern classes, dualizing classes, etc.\ Moreover, we give a criterion for the existence of exceptional isomorphisms, with applications for matroids and full flag varieties.

In \Cref{sec:hrr-pairs}, we generalize the notion of Riemann-Roch pairs to the case where $\Sigma$ and $\Gamma$ may be distinct, but the rings $K_f$ and $A_g$ are nevertheless related via an isomorphism $\ch: \Q K_f \to \Q A_g$, behaving like the Chern character. We show that the pair $(K_f, A_g)$ can be embedded into a standard graded Riemann-Roch pair, which allows us to make use of the splitting principle to transfer all of the structures above---the $\lambda$-ring structure, Serre duality, and the associated graded ring---to the general case. This generality is needed for examples such as the Grassmannian $\Gr(2,4)$, whose $K$-ring is not generated by line elements (\Cref{example:grassmannian-2-4}).

Finally, we provide a range of applications in \Cref{sec:examples}. We connect our results on dualizing classes to the anticanonical polytopes of \cite{kavvil} in the context of linear families of polytopes. We further express McMullen's polytope algebra as a limit of Frobenius quotients and show that such limits inherit a Chern character, well-defined Chern classes and a $\lambda$-ring structure (\Cref{thm:limit-lambda-chern}). Moreover, we give explicit formulas expressing the Snapper and volume polynomials on toric variety bundles, allowing us to give concrete expressions of the Chow and $K$-rings of these varieties as Frobenius quotients in \Cref{thm:bundle K and Chow}. As a corollary, we show that the ring of conditions and its $K$-theoretic version of horospherical homogeneous space have a Chern character, well-defined Chern classes and a $\lambda$-ring structure (\Cref{cor:conditions-ring-horospherical}). Lastly, we study the Chow and $K$-theory of fans. We translate the Ehrharticity condition of \cite{EhrhartFans} into a system of difference equations and give a necessary condition for the existence of exceptional isomorphisms in this setting.

\subsection*{Acknowledgements}
We would like to thank Ben Briggs, Christian Haase, Matt Larson, Dustin Ross and Evgeny Smirnov for stimulating discussions on the topic. During the final stages of the project, LM was a visitor at the Institute for Theoretical Studies at ETH Z\"urich. LM thanks the ITS for its  hospitality and inspiring research environment.

The work of MD and LM was supported by the SPP 2458 ``Combinatorial Synergies", funded by the Deutsche Forschungsgemeinschaft (DFG, German Research Foundation)-- project number 539974215, as well as Swiss National Science Foundation (SNSF) grant -- 
224099. AG has received funding from the Deutsche Forschungsgemeinschaft (DFG, German Research Foundation) TRR 326 \emph{Geometry and Arithmetic of Uniformized Structures}, project number 444845124, from the Deutsche Forschungsgemeinschaft (DFG, German Research Foundation) Sachbeihilfe \emph{From Riemann surfaces to tropical curves (and back again)}, project number 456557832, and from the Marie-Sk\l{}odowska-Curie-Stipendium Hessen (as part of the HESSEN HORIZON initiative).

\subsection*{Notation}
Below you may find a concise reference for the notation which will be used throughout the paper.
\begin{center}
\begin{tabular}{r|l}
    $\N$ & non-negative integers \\
    $\vec u$ & a tuple $\vec u = (u_1, \dots, u_n)$ \\
    $\Q R$ & $\Q \otimes_\Z R$, where $R$ is a ring \\
    $\Lambda$ & a lattice \\
    $\Gamma$ & a \emph{graded} lattice \\
    $\Sigma$ & an affine semigroup \\
    $\Diff(\Gamma)$ & differential operators on $\Gamma$ with $\Z$ coefficients \\
    $\Sh(\Sigma)$ & shift operators on $\Sigma$ with $\Z$ coefficients \\
    $K_f$ & the discrete Frobenius quotient $\Sh(\Sigma)/\Ann(f)$ \\
    $A_g$ & the continuous Frobenius quotient $\Diff(\Gamma)/\Ann(g)$ \\
    $T^u$ & shift by $u$, $T^uf(x) = f(x + u)$\\
    $\Delta_u$ & $T^u - 1$\\
    $D_u$ & $1 - T^{-u}$
\end{tabular}
\end{center}
\medskip

Throughout the paper we will use multi-index notation, i.e.\ for two tuples $\vec a, \vec b$ of the same size $n$, we let
\[
    \vec a^{\vec b} := \prod_{i = 1}^n a_i^{b_i},
    \qquad
    \vec a! = \prod_{i = 1}^n a_i!,
    \qquad\textrm{and}\qquad
    \|\vec a\| = \sum_{i = 1}^n |a_i|.
\]


\section{Rings with Frobenius duality}
\label{sec:rings-with-duality}

In this section we will explore two constructions of commutative Frobenius rings. We will construct these rings as quotients of shift and differential operator algebras respectively, and call these the discrete, resp. continuous Frobenius quotients. Let us start by recalling the definition of a commutative Frobenius ring.

\begin{definition}
    A commutative ring $A$ with identity is called a (commutative) \emph{Frobenius ring}
    if it admits an additive homomorphism $\ell: A \to \Q$ such that the pairing
    \begin{align*}
    A \times A &\to \Q\\
       (a, b) &\mapsto \ell(a\cdot b) 
    \end{align*}
    is non-degenerate.
\end{definition}

The objects motivating the study of such rings are numerical $K$-rings and numerical Chow rings of smooth complete algebraic varieties.

\begin{example}
\label{ex:numerical-K}
    Let $X$ be a smooth complete variety of dimension $d$. Then its Grothendieck group of vector bundles $K(X)$, also called its $K$-ring, admits the Euler characteristic $\chi: K(X) \to \Z$, which is given by the pushforward to a point.
    This allows us to define a pairing on $K(X)$ given by $\chi(-\cdot -)$. In general, this pairing might be degenerate, so $K(X)$ is not necessarily a Frobenius ring.

    We can however take a quotient of $K(X)$ to make it into a Frobenius ring.
    We will say that two classes $[E], [E']$ are numerically equivalent, denoted $[E] \sim [E']$,
    whenever $\chi([E] \cdot [F]) = \chi([E'] \cdot [F])$ for all $[F] \in K(X)$. The quotient
    \[K_\mathrm{num}(X) := K(X)/\sim\]
    is called the \emph{numerical $K$-ring} of $X$.
    By construction, the pairing induced on $K_\mathrm{num}(X)$ is non-degenerate and so $K_\mathrm{num}(X)$ is a Frobenius ring.
\end{example}

\begin{example}
\label{ex:numerical-chow}
    Let $X$ be any smooth complete variety of dimension $d$. Then $A(X)$ admits the degree map $\deg: A(X) \to \Z$, which is given by the pushforward to a point. In general, the induced pairing may be degenerate. However, if we quotient out by numerical equivalence like in \Cref{ex:numerical-K}, the induced pairing on the resulting ring $A_\mathrm{num}(X)$ is non-degenerate, making $A_\mathrm{num}(X)$ a Frobenius ring.
\end{example}

Chow rings are in addition graded rings, and the Frobenius pairing is compatible with this grading. In such case one usually speaks about \emph{Poincar\'e duality rings}.

\begin{definition}\label{def:poincare-duality}
    When $A = A^0 \oplus \dots \oplus A^d$ is a graded commutative ring
    with identity, we say that $A$ satisfies \emph{Poincar\'e duality} (of dimension $d$)
    when there is an additive homomorphism $\deg: A^d \to \Q$ such that the induced pairing
    \begin{align*}
    A^i \times A^{d-i} &\to \Q\\
    (\alpha, \beta) &\mapsto \deg(\alpha \cdot \beta),
    \end{align*}
    is non-degenerate for all $i$.
\end{definition}

The map $\deg$ may be seen as a map $A \to \Z$ after extending by zero to graded components of lower degree.
The non-degeneracy of the pairing $A^i \times A^{d-i} \to \Z$ for all $i$ implies that the pairing $A \times A \to \Z, (\alpha, \beta) \mapsto \deg(\alpha \cdot \beta)$ is also non-degenerate. 

In certain cases, $A(X)$ turns out to be a Poincar\'e duality ring, and $K(X)$ a Frobenius ring even before passing to the numerical quotients. We give two examples of where this happens.

\begin{example}
    Any smooth complex projective variety $X$ with an affine stratification has $A(X)$ freely generated by the classes of the strata by \cite{totaro14}. In particular, the same holds for $H^{\bullet}(X, \Z)$, and hence odd degrees vanish, and we have $A(X) \cong H^{2\bullet}(X, \Z)$. Since there is no torsion, $H^{2\bullet}(X, \Z)$ has Poincar\'e duality in the sense of \Cref{def:poincare-duality}. We conclude that in this case $A(X)$ also has Poincar\'e duality.
\end{example}

\begin{example}
    A wonderful compactification $W_\cA$ of a hyperplane arrangement complement $\cA$ in $\p^n$ also has $A(W_\cA) \cong H^{2\bullet}(W_\cA, \Z)$, because it's a repeated blowup of $\p^n$ along (strict transforms of) linear subspaces of $\p^n$. In particular it has Poincar\'e duality.
    Let $M$ be the matroid associated to $\cA$ and $\Sigma_M$ its Bergman fan.
    By \cite[Theorem 3]{FY04}, the Chow ring of the toric variety $X_{\Sigma_M}$ is isomorphic to $A(W_\cA)$, and so this implies that it's also a Poincar\'e duality ring.
    Note that the socle degree of $A(X_{\Sigma_M})$ is not $\dim(X_{\Sigma_M})$, but rather $\dim(W_\cA)$.
\end{example}

We will now delve into the aforementioned constructions yielding Frobenius rings.

\subsection{Discrete Frobenius quotients}

The first construction, which we call the \emph{discrete Frobenius quotient}, is the quotient of the ring of shift operators by the ideal annihilating a function on a given affine semigroup.

Let $\Sigma$ be an affine semigroup (i.e.\ up to isomorphism a finitely generated submonoid of $\N^k$ for some $k$).
We will denote by $\Fun(\Sigma,\Q)$ the set of maps $f:\Sigma\to \Q$.
Let us further denote by $\Z[\Sigma]$ the semigroup algebra of $\Sigma$. An element of $\Z[\Sigma]$ can be written as a formal sum
\[
\sum_{i = 1}^r a_i T^{u_i},
\]
where $a_i \in \Z$ and $u_i \in \Sigma$.

The semigroup algebra $\Z[\Sigma]$ is acting on the set of functions
$\Fun(\Sigma,\Q)$ via shifts.
That is we define $T^u f(x) = f(x + u)$ and extend by linearity to all of $\Z[\Sigma]$.
To emphasize this action, we denote $\Z[\Sigma]$ by $\Sh(\Sigma)$, and we call it the ring of shift operators.
This makes $\Fun(\Sigma, \Q)$ into a $\Sh(\Sigma)$-module.

Let $f: \Sigma \to \Q$ be a function, then 
\[
\Ann_{\Sh(\Sigma)}(f) = \{E \in \Sh(\Sigma) \,|\, E \cdot f \equiv 0\}
\]
is an ideal of $\Sh(\Sigma)$ by associativity of the action of $\Sh(\Sigma)$ on $\Fun(\Sigma, \Q)$.

\begin{definition}
We call the ring 
\[K_f := \Sh(\Sigma)/\Ann(f)\]
the \emph{discrete Frobenius quotient} associated to the function $f$.
\end{definition}

A similar construction appears in \cite[Theorem 2.2, Theorem 2.3]{MoninSmirnov23} for algebras over a field, where it was used to describe the $K$-ring of smooth complete toric varieties.

\begin{theorem}\label{thm:shiftalg}
    Let $f:\Sigma \to \Q$ be a function.
    Then $K_f$ is a Frobenius ring whose non-degenerate pairing is
    given by the linear map
    \[
    \chi_f: K_f\to \Q,\qquad [E] \mapsto (E \cdot f)(0).
    \]
    
    Furthermore, any finitely generated commutative Frobenius ring comes from this construction.
\end{theorem}

\begin{proof}
    The map $\chi_f$ is well-defined, since by definition $[E] = [F]$ if and only if $E \cdot f = F \cdot f$.
    We need to show that $\chi_f$ induces a non-degenerate pairing. Let $[E] \in K_f$.
    When $[E] \neq 0$, then $E \cdot f \neq 0$ and so there exists some $u \in \Sigma$ such that
    $(E \cdot f)(u) \neq 0$. Hence
    \[
    \chi_f([T^u][E]) = (T^u E \cdot f)(0) = (E \cdot f)(u) \neq 0,
    \]
    which shows that the pairing is non-degenerate.
    
    Now, let $K$ be a commutative Frobenius algebra with Frobenius form given by $\ell$.
    Choose a set $S \subseteq K$, which generates $K$ as a ring,
    and consider $\Sigma := \N S$ the free commutative monoid generated by $S$.
    We have a surjection $\psi: \Z[\Sigma] \to K$. Let $f: \Sigma \to \Q$ be the restriction of
    $$
    \Z[\Sigma] \xrightarrow{\psi} K \xrightarrow{\ell} \Q
    $$
    to $\Sigma \subseteq \Z[\Sigma]$.
    It follows that for any $E = \sum_{i}a_iT^{u_i} \in \Z[\Sigma]$,
    \[\ell(\psi(E)) = \sum_{i}a_i\ell(\psi(T^{u_i})) = \sum_{i}a_i f(u_i)
    = \sum_{i}a_i T^{u_i}f(0) =  (E \cdot f)(0).\]
    
    Let $E \in \ker(\psi)$ and $u \in \Sigma$, then
    $(E \cdot f)(u) = (T^u E \cdot f)(0) = \ell(\psi(T^u E)) = 0,$ which implies that
    $E \cdot f = 0$, and hence $E \in \Ann(f)$.
    Reciprocally, if $E \in \Ann(f)$, then for
    all $\xi \in K$, there is some $F \in \Z[\Sigma]$ such that $\xi = \psi(F)$ and so
    $\ell(\psi(E) \xi) = \ell(\psi(EF)) = (EF \cdot f)(0) = 0$,
    from which follows by non-degeneracy of $\ell(-\cdot -)$ that 
    $\psi(E) = 0$ and hence $E \in \ker(\psi)$.

    We conclude that $\ker(\psi) = \Ann(f)$, which shows that $K \cong K_f$ as rings, and $\ell$ agrees with $\chi_f$ via this identification, making $K$ and $K_f$ isomorphic as Frobenius algebras.
\end{proof}

\begin{remark}\label{rem:action-orbit-sh}
    By the commutativity of the action of $\Sh(\Sigma)$ on $\Fun(\Sigma, \Q)$, 
    we obtain naturally 
    an action of $K_f = \Sh(\Sigma)/\Ann(f)$ on the orbit $\Sh(\Sigma) \cdot f$.
\end{remark}

There are many representations of a fixed commutative Frobenius ring, so the choice of a pair $(\Sigma, f)$ carries strictly more information than just the data of the ring $K_f$, along with the non-degenerate pairing. For example, we get from this data a notion of \emph{positive} elements, and \emph{line} elements.

\begin{definition}
A \emph{positive} element of $K_f$ is an element that may be written as $\sum_i [T^{u_i}]$ for some $u_i \in \Sigma$. We denote the set of positive elements by $K_f^+$. Note that as an additive group, $K_f$ is the groupification of $K_f^+$.

We call \emph{line} elements the invertible positive elements of $K_f$,
and we denote the set of line elements by
$K_f^1$.
\end{definition}

We will also consider \emph{$\Q$-algebras} with Frobenius duality. We could construct such an algebra like a discrete Frobenius quotient, where we instead take the quotient ring of shift operators \emph{with rational coefficients} by the annihilator of some function. As we will now show, this is equivalent to tensoring the discrete Frobenius quotient with $\Q$.

Let $\Sh_\Q(\Sigma) := \Sh(\Sigma) \otimes_\Z \Q$. The action of $\Sh(\Sigma)$ on $\Fun(\Sigma, \Q)$ extends naturally to an action of $\Sh_\Q(\Sigma)$.

\begin{proposition}
\label{prop:ann-sh-rational}
    Let $f \in \Fun(\Sigma, \Q)$ be a function.
    Then
    $\Ann_{\Sh_\Q(\Sigma)}(f) = \Q\Ann_{\Sh(\Sigma)}(f)$.
\end{proposition}

\begin{proof}
The inclusion $\supseteq$ is obvious. Let $E = \sum_{i}a_iT^{u_i} \in \Ann_{\Sh_\Q(\Sigma)}(f)$, then there exists some $q \in \Z$ such that $qa_i \in \Z$ for all $i$. In particular,
\[
E = \frac{1}{q}\sum_{i} qa_i T^{u_i} \in \Q \Ann_{\Sh}(f).
\qedhere
\]
\end{proof}

\begin{corollary}
    Let $f \in \Fun(\Sigma, \Q)$ be a function.
    Then $\Sh_\Q(\Sigma)/\Ann_{\Sh_\Q(\Sigma)}(f) = K_f \otimes_\Z \Q$.
\end{corollary}

For simplicity, in the future we will denote $\Q K_f := K_f \otimes_\Z \Q$. 
Since $K_f$ has no $\Z$-torsion, it may be seen as a subring of $\Q K_f$. When $\xi \in K_f \subseteq \Q K_f$, we say $\xi$ is \emph{integral}.

\begin{remark}
    It can be easily seen that any finitely generated commutative Frobenius $\Q$-algebra, may be obtained as $\Q K_f$ for some function $f: \Sigma \to \Q$, as in \Cref{thm:shiftalg}.
\end{remark}

\subsection{Continuous Frobenius quotients}

We will now define \emph{continuous} Frobenius quotients, which are defined similarly as discrete Frobenius quotients, but in this case we consider the action of differential operators on the ring of polynomials on some graded lattice.

Let $\Gamma = \Gamma_1 \oplus \Gamma_2 \oplus \dots \oplus \Gamma_d$ be a graded lattice. Choose a basis $e_1, \dots, e_n$ of homogeneous elements
and let $x_1, \dots, x_n$ be the dual basis.
We denote by $P(\Gamma)$ the set of polynomials functions with rational coefficients on $\Gamma$, that is we identify
\[P(\Gamma) \cong \Q[x_1, \dots, x_n].\]
The grading on $\Gamma$ induces a grading on $P(\Gamma)$, such that the $x_i$ are homogeneous elements of degree equal to the degree of $e_i$.

Denote by $\Diff(\Gamma)$ the ring of differential operators with constant integer coefficients acting on $P(\Gamma)$. Via the above identification we get
\[
    \Diff(\Gamma) \cong \Z\left[\frac{\partial}{\partial x_1},\ldots,\frac{\partial}{\partial x_n}\right].
\]
Differentiation of polynomials endows $P(\Gamma)$ with a
$\Diff(\Gamma)$-module structure. $\Diff(\Gamma)$ also inherits a grading from $\Gamma$, and the grading is compatible with that on $P(\Gamma)$, in the sense that $\Diff(\Gamma)_k \cdot P(\Gamma)_d \subseteq P(\Gamma)_{d-k}$.

Consider the map $\Gamma \to \Diff(\Gamma)$ generated by
\[e_i \in \Gamma \mapsto \frac{\partial}{\partial x_i} \in \Diff(\Gamma).\]
We denote the image of $u\in \Gamma$ by $\partial_u$ and call it the directional derivative along $u$.

As in the discrete case, given a polynomial $g$, we consider here the ideal of differential operators annihilating $g$.
\[
\Ann_{\Diff(\Gamma)}(g) = \{D\in \Diff(\Gamma) \,|\, D\cdot g \equiv 0\}
\]

\begin{definition}
We call the ring
\[A_g := \Diff(\Gamma)/\Ann_{\Diff(\Gamma)}(g)\]
the \emph{continuous Frobenius quotient} associated to the polynomial $g$.
\end{definition}

The following Lemma shows that when $g$ is weighted homogeneous, then $A_g$ is a graded ring.

\begin{lemma}\label{lem:lambda-homogeneous}
    Let $g\in P(\Gamma)$ be a weighted homogeneous polynomial of degree $d$, then $\Ann_{\Diff(\Gamma)}(g)$
    is a homogeneous ideal in $\Diff(\Gamma)$.
\end{lemma}

\begin{proof}
    Let $D \in \Ann_{\Diff(\Gamma)}(g)$ and write $D = D_0 + \cdots + D_k$, where each $D_i$ is homogeneous of degree $i$. Since $D_i \cdot g \in P(\Gamma)_{d - i}$, it follows that for each $i$ we must have $D_i\cdot g = 0$, and hence $D_i \in \Ann_{\Diff(\Gamma)}(g)$.
\end{proof}

As the following proposition shows, the continuous Frobenius quotient $A_g$ is tightly related to Macaulay's inverse system, see for example \cite{iarrobino1999}.
\begin{proposition}
\label{prop:solution set to differential equation}
Let $\Lambda$ be a lattice, let $I\subseteq \Diff(\Lambda)$ be an ideal, and $A=
\Diff(\Lambda)/I$. Moreover, let $S$ 
be the set of polynomial functions on $\Lambda$ that are solutions to the partial differential equations given by $I$. Then the map
\[
    \Psi: \Hom_\Z(A,\Q) \to S, \;\; \varphi\mapsto
    \left( \mathbf{u} \mapsto 
    \sum_{\mathbf{k} \in \N^n} \frac{\mathbf{u}^{\mathbf{k}}}{\mathbf{k}!} \varphi([\partial^{\mathbf{k}}]) \right)
\]
is an isomorphism of $A$-modules with inverse
\[
    S \to \Hom_\Z(A,\Q),\;\; f\mapsto (\alpha \mapsto (\alpha \cdot f)(0)).
\]
Here we identify $\Lambda \cong \Z^n$ by fixing a basis $(e_1,\dots, e_n)$, and denote
\[
    \partial^\mathbf{k} := \partial_{e_1}^{k_1} \cdots \partial_{e_n}^{k_n}
\]
The infinite sum is well-defined, because $[\partial^\mathbf{k}] = 0$ for $\mathbf{k}$ large enough.
\end{proposition}

\begin{proof}
    Let $\varphi \in \Hom_\Z(A, \Q)$, then $\Psi(\varphi)$ is clearly a polynomial. The fact that $\Psi(\varphi) \in S$ will follow if we show that $\Psi$ is an $A$-module homomorphism. Since the map is $\Z$-linear it suffices to show this for monomials $\partial^\mathbf{a}$.
    We have that
    \begin{align*}
        \partial^\mathbf{a} \Psi(\varphi)(\mathbf{u}) &=
        \sum_{\substack{\mathbf{k} \in \N^n\\\mathbf{k} \geq \mathbf{a}}}
        \frac{\mathbf{u}^{\mathbf{k} - \mathbf{a}}}{(\mathbf{k} - \mathbf{a})!} \varphi([\partial^{\mathbf{k}}])\\
        &= \sum_{\mathbf{k'} \in \N^n }
        \frac{\mathbf{u}^{\mathbf{k'}}}{\mathbf{k'}!} \varphi([\partial^\mathbf{a}][\partial^{\mathbf{k'}}])
        = \Psi([\partial^\mathbf{a}] \cdot \varphi)(\mathbf{u}).
    \end{align*}
    This shows that $\Psi$ is a well-defined $A$-module homomorphism.
    
    Note that $\partial^\mathbf{k} \Psi(\varphi)(0) = \varphi([\partial^\mathbf{k}])$, and so it follows by linearity that for any $\alpha \in A$, $(\alpha \cdot \Psi(\varphi))(0) = \varphi(\alpha)$.
    
    Reciprocally, let $f \in S$ and $\vec u \in \Lambda$, then it follows from the Taylor expansion of $f$ that
    \[
        f(\vec u) = \sum_{\mathbf{k} \in \N^n} \frac{\mathbf{u}^\mathbf{k}}{\mathbf{k}!} (\partial^{\mathbf{k}}f)(0)
        = \Psi\big(\alpha \mapsto (\alpha \cdot f)(0)\big)
    \]
    and hence this shows that the two maps are inverses to each other.
\end{proof}

As in the discrete case, $A_g$ is a Frobenius ring, and furthermore any finitely generated Frobenius ring is of this form.
Similar statements are shown for example in
\cite{iarrobino1999}, \cite[Ex. 21.7]{Eis} or \cite[Theorem 1.1]{Kaveh11} for algebras over a field.
In \cite[Theorem 4.8]{khovanskii2021gorenstein} a more general version of this statement is shown for possibly infinitely generated graded algebras over a field and which are not necessarily generated in degree 1.

\begin{theorem}\label{thm:diffalg}
Let $g\in P(\Gamma)$ be a polynomial of degree $d$.
Then $A_g$ is a Frobenius algebra with Frobenius form defined by the function 
\[\deg_g: A_g \to \Q, \qquad [D] \mapsto (D \cdot  g)(0),\]
Furthermore, if $g$ is weighted homogeneous, then $A_g$ is a graded ring of socle degree $d$, and $\deg_g$ is zero outside of the top graded component, so in particular $A_g$ is a
Poincar\'e duality ring  of dimension~$d$.

Furthermore, any commutative Frobenius ring, which is generated by finitely many nilpotent elements, arises from this construction. 
\end{theorem}

\begin{proof}
    Clearly $\deg$ is well-defined. Suppose there is some $[D] \in A_g$ such that
    $\deg([D] \cdot [D']) = 0$ for all $[D'] \in A_g$. Then in particular
    $(D' \cdot (D \cdot g)) (0) = 0$ for all $D' \in \Diff(\Gamma)$,
    which implies that $D \cdot g = 0$, and hence $[D] = 0$ by definition of $A_g$. This shows that the pairing is non-degenerate.

    The fact that $A_g$ is graded when $g$ is homogeneous follows by \Cref{lem:lambda-homogeneous}.

    Now, suppose $A$ is a finitely generated commutative Frobenius ring. 
    Let $\alpha_1, \dots, \alpha_n \in A$ be a set of nilpotent generators.
    Fix $e_1, \dots, e_n$ a basis of $\Gamma$.
    Consider the homomorphism
    $\varphi: \Diff(\Gamma) \to A$
    generated by $\partial_i \mapsto \alpha_i$. Clearly, $\varphi$ is surjective. Denote by $I$ the kernel of this map and identify $A \cong \Diff(\Gamma)/I$.

    Let $g: \Gamma \to \Q$ be the polynomial defined by
    \[g(\mathbf{u}) = \sum_{\mathbf{k} \in \N^n} \frac{\mathbf{u}^\mathbf{k}}{\mathbf{k}!}\deg([\partial^\mathbf{k}]).\]
    The fact that $g$ is a polynomial follows from the assumption that the generators are nilpotent. Then $I \subseteq \Ann(g)$ by \Cref{prop:solution set to differential equation}, since $g = \Psi(\deg)$.
    
    Now, suppose $D \in \Ann(g)$.
    Then for all $D' \in \Diff(\Gamma)$,
    $\deg([D'][D]) = D'D \cdot g = 0$.
    By non-degeneracy of the pairing $\deg(-\cdot-)$, it follows that $[D] = 0$,
    which shows that $D \in I$.

    We conclude that $I = \Ann(g)$, and hence this gives an isomorphism $A \cong A_g$.
\end{proof}

\begin{remark}
    As in \Cref{rem:action-orbit-sh}, we have an induced action of $A_g$ on the orbit
    $\Diff(\Gamma) \cdot g$.
\end{remark}

Write $\Diff_\Q(\Gamma) := \Diff(\Gamma) \otimes_\Z \Q$ for the ring of differential operators on $\Gamma$ with constant rational coefficients.
As was the case for the ring of shift operators, the action of $\Diff(\Gamma)$ on $P(\Gamma)$
extends naturally to an action of $\Diff_\Q(\Gamma)$, and we get the following fact.

\begin{proposition}\label{prop:ann-diff-rational}
    Let $g \in P(\Gamma)$ be a polynomial.
    Then
    $\Ann_{\Diff_\Q(\Gamma)}(g) = \Q\Ann_{\Diff(\Gamma)}(g)$.
    It follows that 
    \[\Diff_\Q(\Gamma)/\Ann_{\Diff_\Q(\Gamma)}(g) = \Q A_g,\]
    where we denote $\Q A_g := A_g \otimes_\Z \Q$.
\end{proposition}

\begin{remark}
    As in the discrete case, any finitely generated commutative $\Q$-algebra with Poincar\'e duality may be obtained as $\Q A_g$ for some weighted homogeneous polynomial $g$, as in \Cref{thm:diffalg}.
\end{remark}

\subsection{Isomorphisms and Todd Classes}

As seen in Examples \ref{ex:numerical-K} and \ref{ex:numerical-chow}, the rings $K_f$ and $A_g$ model the numerical $K$-rings and Chow rings of any smooth complete variety.
For such a variety $X$, the Hirzebruch-Riemann-Roch theorem states that for any vector bundle $E$ on $X$, the following formula holds
\[\chi(E) = \deg(\ch(E) \cdot \td(X)).\]
In particular, this implies that two vector bundles are numerically equivalent if and only if their images by the Chern character are numerically equivalent, and so the Chern character descends to an isomorphism $\Q K_\num(X) \to \Q A_\num(X)$. This illustrates that there is a natural connection between the existence of an isomorphism between the numerical $K$- and Chow rings, and the Todd class, which relates the pairings on the two rings. In what follows, we will generalize this connection to Frobenius quotients.

Let $g\in P(\Gamma)$ be a polynomial and let
\[A_g = \Diff(\Gamma)/\Ann(g)\]
be its associated continuous Frobenius quotient.

\begin{theorem}
\label{thm:map-descends-iff-hrr}
    Let $\varphi: \Sh_\Q(\Sigma) \to \Q A_g$ be a surjective map and $f: \Sigma \to \Q$ a function. Then $\varphi$ descends to an 
    isomorphism $\Q K_f \cong \Q A_g$ if and only if there exists an invertible element
    $\td \in \Q A_g^\times$ such that
    \begin{equation}
    \label{eq:hrr}
        \chi_f(E) = \deg_g(\td \cdot \varphi(E))\qquad\textrm{for all }E \in \Sh(\Sigma).
    \end{equation}
    We call $\td$ the \emph{Todd class} of $\varphi$.
\end{theorem}

\begin{proof}
    Suppose that $\varphi$ induces an isomorphism $\hat{\varphi}: \Q K_f \to \Q A_g$. Since $\Q A_g$ is a finite-dimensional vector space, the Frobenius pairing is perfect, and hence $\Q A_g \cong \Hom_\Q(\Q A_g, \Q)$.
    So let $\td \in \Q A_g$ be the dual of the composite 
    \[\Q A_g \xrightarrow{\hat{\varphi}^{-1}} \Q K_f \xrightarrow{\chi_f} \Q.\]
    Then by definition
    \[\chi_f(\hat{\varphi}^{-1}(\alpha)) = \deg_g(\td\cdot \alpha)\textrm{ for all }\alpha \in A_g,\]
    and we obtain \Cref{eq:hrr} by setting $\alpha=\varphi(E)$ for $E\in \Sh(\Sigma)$.

    Because $A_g$ is zero-dimensional, to show that $\td$ is invertible, it suffices to show that it is not a zero-divisor. Let $E\in \Sh(\Sigma)$ be such that $\td\cdot \varphi(E)=0$. Then for all $u\in \Sigma$ we have
        \begin{equation}\label{eq:hrr-shift}
        (E\cdot f)(u) = \chi_f(T^u\cdot E) = \deg_g(\td\cdot\varphi(T^u)\cdot\varphi(E) )
        = 0,
        \end{equation}
    and hence $E \in \Ann(f)$. Since $\hat\varphi$ is an isomorphism, $\Ann(f) = \ker(\varphi)$, and hence
    $\varphi(E)=0$. By surjectivity of $\varphi$, we conclude that $\td$ is not a zero-divisor.

    Now, suppose that \Cref{eq:hrr} holds for some $\td \in \Q A_g^\times$. We will show that
    the $\varphi$ descends to the quotient and induces an isomorphism by showing that $\ker(\varphi)=\Ann(f)$. Suppose that $E \in \ker(\varphi)$.
    Then for any $u \in \Sigma$, \Cref{eq:hrr-shift} shows that
    $E \in \Ann(f)$. 
    
    Now, let $E \in \Ann(f)$. Then for any $F \in \Sh_\Q(\Sigma)$ by \Cref{eq:hrr} we have
    \[\deg_g(\td \cdot \varphi(F) \cdot \varphi(E)) = \chi_f(E \cdot F) = (F \cdot E \cdot f)(0) = 0.\]
    Since $\td$ is invertible and $\varphi$ surjective, this implies that for any $\alpha \in \Q A_g$, 
    \[\deg_g(\alpha \cdot \varphi(E)) = 0,\]
    and so since $\deg_g(-\cdot -)$ is non-degenerate, we get $\varphi(E) = 0$. We conclude that the map induced by $\varphi$ on the quotient is an isomorphism.
\end{proof}

\begin{remark}
    \label{rem:hrr-equivalent-statement}
    By linearity, \Cref{eq:hrr} is equivalent to the statement
    \[
        f(u) = \deg_g(\td \cdot \varphi(T^u))\qquad\textrm{for all }u \in \Sigma.
    \]
    It follows that in this case $f$ is fully determined by $g$, $\td$, and $\varphi$.
\end{remark}

\begin{remark}
    Note in the situation of \Cref{thm:map-descends-iff-hrr},
    $\Ann_{\Sh_\Q}(f) = \ker(\varphi)$. In particular, if $f$ and $f'$ satisfy the condition of 
    \Cref{thm:map-descends-iff-hrr} for a fixed $g$, then $\Ann_{\Sh_\Q}(f) = \Ann_{\Sh_\Q}(f')$
\end{remark}

\section{Standard graded Riemann-Roch pairs}
\label{sec:standard}

In this section we will reduce generality and consider Frobenius quotients $K_f$ and $A_g$, where $f$ and $g$ are two polynomials on some fixed lattice $\Lambda$.
We call this setting \emph{standard graded}. If $g$ is homogeneous, then $A_g$ is a standard graded Poincar\'e duality ring, which explains the choice of terminology.

Later in this section we will take a closer look at the situation, when $K_f$ and $A_g$ are connected via the \emph{Chern character}. Such pairs are what we are going to call \emph{Riemann-Roch pairs}.

The main motivating examples for considering the standard graded case come from toric geometry, and from full flag varieties.
\begin{example}
    \label{example:toric}
    Let $P$ be a full-dimensional smooth lattice polytope and $\Sigma_P$ its normal fan. Denote $X_P$ the associated smooth complete toric variety. In particular, its Chow ring
    $A(X_P)$ is isomorphic to the cohomology ring of $X_P$ (see \cite[\S 5.2]{fulton93}), so has Poincar\'e duality and since by the Hirzebruch-Riemann-Roch theorem,
    \[\chi(X_P, E) = \int_{X_P} \td(X) \ch(E)\]
    for all vector bundles $E$, it follows that $K(X_P)$ is a Frobenius ring.

    For each ray $\rho \in \Sigma_P(1)$, let $D_\rho$ be the associated torus-invariant divisor.
    Then the line bundle classes $[\cO(D_\rho)]$ generate the $K$-ring of $X_P$ and
    the cycles $[D_\rho]$ generate the Chow ring.
    It follows that if we let $\Lambda := \Z^{\Sigma_P(1)}$, then we have surjections
    \begin{align*}
        \Sh(\Lambda) \to K(X_P),& \qquad T^{e_\rho} \mapsto [\cO(D_\rho)], \\
        \Diff(\Lambda) \to A(X_P),& \qquad \partial_{e_\rho} \mapsto [D_\rho],
    \end{align*}
    whose kernels are the annihilators of two polynomials, $\Snap_P$ and $\Vol_P$ respectively, which correspond to the linear maps $\chi$ and $\int_{X_P}(-)$,
    as in Theorems \ref{thm:shiftalg} and \ref{thm:diffalg}.
    More explicitly, letting $\mathbf{t} \in \Lambda$, and $D_\mathbf{t}$ the associated divisor, we have that 
    \begin{align*}
        \Snap_P(\mathbf{t}) &= \chi(X_P, \cO(D_\mathbf{t})), \\
        \Vol_P(\mathbf{t}) &= \frac{1}{d!} \int_{X_P}D_{\mathbf{t}}^d.
    \end{align*}
    When $D_\mathbf{t}$ is nef, and $P_{\mathbf t}$ is the polytope associated to $D_\mathbf{t}$, then 
    $\Vol_P(\mathbf{t})$
    is equal to the volume of $P_{\mathbf t}$ and
    $\Snap_P(\mathbf{t}) = \#(P_{\mathbf t} \cap M)$.
    For classes that are not nef the above can be extended if we allow working with virtual polytopes as in \cite{PK92}.
    In other words, the Snapper polynomial in this case agrees with the Ehrhart lattice point counting function of a linear family of polytopes, see \Cref{subsec:linear-families}.
\end{example}

\begin{example}
\label{ex:full-flags}
Let $G$ be a reductive group with Borel subgroup $B$. Let $\Lambda$ be the character lattice of $B$. 
Every character $\lambda\in \Lambda$ defines a line bundle $L_\lambda$ on the full flag variety $G/B$ via:
\[
L_\lambda = G\times_B \C_\lambda \coloneqq \frac{G\times \C_\lambda}{B},
\]
where $\C_\lambda$ is the 1-dimensional representation of $B$ given by the character $\lambda$.
Moreover, the line bundles $L_\lambda$ generate the $K$-ring of $G/B$ and their Chern classes $c_1(L_\lambda)$ generate the Chow ring. Hence we have surjective maps:
\begin{align*}
        \Sh(\Lambda) \to K(G/B),& \qquad T^{\lambda} \mapsto [L_\lambda],\\
        \Diff(\Lambda) \to A(G/B),& \qquad \partial_{\lambda} \mapsto [c_1(L_\lambda)].
    \end{align*}
The kernels of the above maps are the annihilators of the Snapper, resp. volume polynomial:
\begin{align*}
   \mathrm{Snap}_{G/B}(\lambda) &= \chi(G/B, L_\lambda),\\
    \Vol_{G/B}(\lambda) &= \frac{1}{d!} \int_{G/B} c_1(L_\lambda)^{d},
\end{align*}
where $d = \dim(G/B)$.
By the Borel-Weil-Bott theorem, for every dominant weight $\lambda$, we get
\begin{align*}
H^0(G/B,L_\lambda) &= V_\lambda,\\
H^i(G/B,L_\lambda) &= 0 \text{ for } i>0,
\end{align*}
where $V_\lambda$ is the irreducible representation of $G$ with highest weight $\lambda$. Hence the 
Snapper polynomial for $G/B$ is given by Weyl dimension polynomial and $\Vol_{G/B}$ by its top homogeneous component:
\begin{align*}
   \mathrm{Snap}_{G/B}(\lambda) &= \frac{\prod_{\alpha\in \Delta^+}(\lambda+\rho,\alpha)}{\prod_{\alpha\in \Delta^+}(\rho,\alpha)},\\
    \Vol_{G/B}(\lambda) &= \frac{\prod_{\alpha\in \Delta^+}(\lambda,\alpha)}{\prod_{\alpha\in \Delta^+}(\rho,\alpha)}.
\end{align*}
Here, $\Delta^+$ denotes the set of positive roots and $\rho$ is the half-sum of all positive roots.

As in toric case, $\Vol$ and $\mathrm{Snap}$ can also be related  to volumes and lattice count polynomials. There exists a full-dimensional cone $C\subset \Lambda_\R$ and polytopes $\mathrm{S}(\lambda)\subseteq \R^d$ called string polytopes \cite{littelmanncones,BerZel} with the property that
\[
\mathrm{S}(\lambda+\mu)=\mathrm{S}(\lambda) + \mathrm{S}(\mu) \text{ for every } \lambda,\mu\in C,
\]
and such that $\Vol_{G/B}(\lambda)$ is the volume of $\mathrm{S}(\lambda)$
and $\mathrm{Snap}_{G/B}(\lambda) = \#(\mathrm{S}(\lambda) \cap \Z^d)$.
One can extend the above construction to all characters $\lambda$ if we allow working with virtual polytopes.
For more details on this example see \cite{MoninSmirnov23,Kaveh11}.
As in the toric case, the Snapper polynomial agrees with the Ehrhart lattice point counting function of a linear family of polytopes, see \Cref{subsec:linear-families}.
\end{example}

\subsection{\texorpdfstring{$K_f$}{Kf} in the standard graded setting}

The discrete Frobenius quotient $K_f$ behaves nicely in this setting. First of all, it is generated by invertible elements and $K_f$ is a finitely generated $\Z$-module for any choice of polynomial $f$.

\begin{proposition}\label{prop:polyart}
  If $f \in P(\Lambda)$ is a polynomial on $\Lambda$, then $K_f$ is a finitely generated $\Z$-module.
\end{proposition}

\begin{proof}
For any $\mathbf{u} \in \Lambda$, let $\Delta_\mathbf{u} := T^\mathbf{u} - 1$ be the difference operator with respect to $\mathbf{u}$.

Choose a basis $e_1,\ldots, e_n$ of $\Lambda$,
so that we identify $\Lambda \cong \Z^n$.
We claim that $\Delta_{\pm e_1}, \dots, \Delta_{\pm e_r}$ generate
$\Sh(\Lambda)$ as a $\Z$-algebra. Let $\vec u \in \Lambda$ and suppose without loss of generality that $u_i \geq 0$ for all $i$.
We will show by induction on $\|\vec u\|$ that $T^{\vec u}$ is in the subalgebra generated by
$\Delta_{e_1}, \dots, \Delta_{e_n}$. We have that 
$$\Delta_{e_1}^{u_1} \cdots \Delta_{e_n}^{u_n} = T^{\vec u} + \sum_{i = 1}^s \alpha_i T^{\vec v_i},$$
where for all $i$, $\|\vec v_i\| < \|\vec u\|$. By induction each $T^{\vec v_i}$ is in the span of $\Delta_{e_1}, \dots, \Delta_{e_r}$, and hence the same holds for $T^{\vec u}$.

We deduce that the equivalence classes of $\Delta_{\pm e_1}, \dots, \Delta_{\pm e_r}$ generate $K_f = \Sh(\Lambda)/\Ann(f)$ as a ring.
However since any monomials in $\Delta_{\pm e_1}, \dots, \Delta_{\pm e_r}$
of degree larger than $d$
vanish in $K_f$, we also obtain that $K_f$ is finitely generated as a $\Z$-module.
\end{proof}

In this setting, we also have that $K_f$ admits a natural notion of \emph{rank}, which comes from the rank function on $\Sh(\Lambda)$.

\begin{definition}    
Let $E = \sum_{i} a_iT^{u_i} \in \Sh(\Lambda)$ be a shift operator.
We define the \emph{rank} of $E$ to be $\rk(E) := \sum_{i} a_i$.
\end{definition}

Note that $\rk$ is a ring homomorphism 
$\Sh(\Lambda) \to \Z$. As the following proposition shows, when $f$ is a polynomial, $\rk$ descends to a ring homomorphism
$K_f \to \Z$.

\begin{proposition}
    \label{lemma:annihilator-rank-0}
    Let $f$ be a polynomial function on $\Lambda$. Then for every $E\in \Sh(\Lambda)$ with $E\cdot f=0$ we have $\rk(E)=0$. 
\end{proposition}

\begin{proof}
    By fixing a basis of $\Lambda$, identify $\Lambda \cong \Z^n$.
    Using the Taylor series expansion of $f$ around $\vec x$,
    we have
    \[
        T^{\vec u}f(x) = f(\vec x + \vec u) = \sum_{\mathbf{k} \in \N^n} \frac{\mathbf{u}^\mathbf{k}}{\mathbf{k}!} (\partial^{\mathbf{k}}f)(\vec x).
    \]
    In particular, the polynomials $f$ and $T^{\vec u}f$ have the same top-degree homogeneous component. It follows that the top homogeneous component of $E\cdot f$ is just $\rk(E)$ times the top homogeneous component of $f$, from which the result follows.
\end{proof}

\begin{remark}
    Line elements are precisely the positive elements of rank 1, justifying the notation $K_f^1$.
\end{remark}

\subsection{The Chern character}
\label{subsec:chern-character}

In this subsection we construct the \emph{Chern character}, which under certain conditions yields an isomorphism $\Q K_f \cong \Q A_g$. Pairs $(K_f, A_g)$ which are related by the Chern character, and where $A_g$ is a Poincar\'e duality ring, will be called \emph{Riemann-Roch pairs}.

We denote the formal completion of $\Diff_\Q(\Lambda)$ at the maximal ideal generated by all directional derivatives by $\Diff_\Q\llbracket \Lambda \rrbracket$. After fixing a basis of $\Lambda$ we get the identification
\[
\Diff_\Q\llbracket \Lambda \rrbracket \cong 
\Q\left\llbracket\frac{\partial}{\partial x_1},\ldots,\frac{\partial}{\partial x_n}\right\rrbracket
\]

For every $u \in \Lambda$ it we define the exponential $\exp(\partial_u)\in \Diff_\Q\llbracket \Lambda \rrbracket$ by
\[
\exp(\partial_u)= \sum_{i=0}^\infty \frac{\partial_u^i}{i!}.
\]
By the standard argument, $\exp$ is multiplicative, i.e.\ for $u, v \in \Lambda$, 
\[\exp(\partial_{u + v}) = \exp(\partial_u + \partial_v) = \exp(\partial_u)\exp(\partial_v),\]
and hence induces a ring homomorphism
\[
\ch\colon \Sh(\Lambda)\to \Diff_\Q\llbracket \Lambda \rrbracket,
\]
sending $T^u$ to $\exp(\partial_u)$ for all $u\in\Lambda$.
We call this homomorphism the \emph{Chern character}. 

The following result, which can be generalized to hold for analytic functions, gives us a way to relate the action of shift operators and differential operators on polynomials.

\begin{lemma}
\label{lem:exponential and shifts}
For every $u\in \Lambda$, and polynomial $h \in P(\Lambda)$, we have 
\[
\exp(\partial_u) \cdot h = T^u h .
\]
In particular, we have 
\[
\ch(E) \cdot h=E \cdot h
\]
for all $E\in \Sh(\Lambda)$.
\end{lemma}

\begin{proof}
    After fixing a basis $(e_1, \dots, e_n)$,
    identify $\Lambda \cong \Z^n$. Let $\vec u \in \Lambda$. First note that
    \[
    \exp(\partial_{\vec u}) = \sum_{k \geq 0} \frac{\partial_{\vec u}^k}{k!}
    = \sum_{k \geq 0} \frac{(u_1 \partial_{e_1} + \dots + u_n \partial_{e_n})^k}{k!}
    = \sum_{\vec k \in \N^n}\frac{\vec u ^{\vec k}}{\vec k !}\partial^{\vec k}.
    \]
    Then using the Taylor series expansion we obtain the following identity
    \[
        (T^{\vec u}h)(\vec x)= h(\vec x + \vec u) = \sum_{\vec k \in \N^n} \frac{\vec u ^{\vec k}}{\vec k!} (\partial^{\vec k} h)(\vec x) = (\exp(\partial_{\vec u})\cdot h)(\vec x).
    \qedhere
    \]
\end{proof}

\begin{lemma}
    \label{lem:ch-surjective}
    Let $g \in P(\Lambda)$ be a polynomial. The Chern character induces a surjective map
    $\ch: \Sh_\Q(\Lambda) \to \Q A_g$. 
\end{lemma}

\begin{proof}
    Since $\Q A_g$ is generated by elements of the form $[\partial_u]$ with $u \in \Lambda$,
    it suffices to show that these belong to the image. Recall that the Taylor series expansion of $\log(1 + x)$
    is $\sum_{k \geq 1} (-1)^{k + 1} x^k/k$. It follows that
    \[
        x = \log(1 + (\exp(x) - 1)) = \sum_{k = 1}^\infty (-1)^{k+1} \frac{(\exp(x) - 1)^k}{k}
    \]
    as formal series. Note that $\exp([\partial_u]) - 1$ is a well-defined and nilpotent element of $\Q A_g$, more precisely, $(\exp([\partial_u]) - 1)^k = 0$ for all $k > d$.
    It follows that in $\Q A_g$ we have the equality
    \[
        [\partial_u] = \sum_{k = 1}^d (-1)^{k+1} \frac{(\exp([\partial_u]) - 1)^k}{k}
        = \ch\left(\sum_{k = 1}^d (-1)^{k+1} \frac{(T^u - 1)^k}{k}\right).
        \qedhere
    \]
\end{proof}

Thanks to these lemmas, we may restate \Cref{thm:map-descends-iff-hrr} in this setting in a simpler form.

\begin{theorem}
\label{thm:ch-descends-iff-todd}
    Let $f : \Lambda \to \Q$ a function and $g \in P(\Lambda)$ a polynomial.
    The Chern character induces an isomorphism $\ch: \Q K_f \to \Q A_g$ if and only if there exists an invertible element $\td \in \Q A_g^\times$ such that $f = \td \cdot g$.
\end{theorem}

\begin{proof}
    This follows from \Cref{lem:exponential and shifts} and \Cref{lem:ch-surjective}, along with \Cref{thm:map-descends-iff-hrr}, and \Cref{rem:hrr-equivalent-statement}.
\end{proof}

Note that $f = \td \cdot g$ implies that $f$ is a polynomial too.

\begin{proposition}\label{prop:polynomial-equivalence-inclusion}
    Let $f, g \in P(\Lambda)$ be polynomials.
    The following are equivalent:
    \begin{enumerate}
        \item $\Ann_{\Sh_\Q}(g) \subseteq \Ann_{\Sh_\Q}(f)$
        \item $\Ann_{\Diff_\Q}(g) \subseteq \Ann_{\Diff_\Q}(f)$
        \item $\Ann_{\Sh}(g) \subseteq \Ann_{\Sh}(f)$
        \item $\Ann_{\Diff}(g) \subseteq \Ann_{\Diff}(f)$
        \item There exists some $\alpha \in \Q A_g$ such that $f = \alpha \cdot g$.
        \item There exists some $\xi \in \Q K_g$ such that $f = \xi \cdot g$.
    \end{enumerate}
\end{proposition}

\begin{proof}
    The equivalences $(1) \iff (3)$ and $(2) \iff (4)$ follow from \Cref{prop:ann-sh-rational} and \Cref{prop:ann-diff-rational} noting that annihilators are $\Z$-saturated and none of the rings involved have $\Z$-torsion.

    Suppose $f = \alpha \cdot g$ for some $\alpha \in \Q A_g$, then clearly $\Ann_\Diff(g) \subseteq \Ann_\Diff(f)$.
    Reciprocally, if $\Ann_\Diff(g) \subseteq \Ann_\Diff(f)$, then $f$ belongs to the set $S$ of solutions to the differential equations given by $\Ann_\Diff(g)$. By \Cref{prop:solution set to differential equation}, $S$ is isomorphic as an $\Q A_g$-module to $\Hom_\Q (\Q A_g, \Q)$. The latter is isomorphic to $\Q A_g$ by Poincar\'e duality. In particular, this implies that $S$ is a cyclic $\Q A_g$-module generated by $g$, and so there exists some $\alpha \in \Q A_g$ such that $f = \alpha \cdot g$. This shows $(4) \iff (5)$.

    (1) is equivalent to the statement that the quotient map $\Q K_g \to \Q K_f$ is well-defined, which via the Chern character is equivalent to the statement that the quotient map $\Q A_g \to \Q A_f$ is well-defined. The latter is equivalent to (2).

    Finally, (5) is equivalent to (6), because the Chern character induces an isomorphism $\Q K_g \cong \Q A_g$.
\end{proof}

As an immediate corollary we obtain:
\begin{corollary} \label{cor:polynomial-equivalence}
    Let $f, g \in P(\Lambda)$ be polynomials.
    The following are equivalent:
    \begin{enumerate}
        \item $\Ann_{\Sh_\Q}(f) = \Ann_{\Sh_\Q}(g)$
        \item $\Ann_{\Diff_\Q}(f) = \Ann_{\Diff_\Q}(g)$
        \item $\Ann_{\Sh}(f) = \Ann_{\Sh}(g)$
        \item $\Ann_{\Diff}(f) = \Ann_{\Diff}(g)$
        \item There exists some invertible $\alpha \in \Q A_g^\times$ such that $f = \alpha \cdot g$.
        \item There exists some invertible $\xi \in \Q K_g^\times$ such that $f = \xi \cdot g$.
    \end{enumerate}
\end{corollary}

\begin{definition}
    If any of the conditions of \Cref{cor:polynomial-equivalence} is met, we will say $f$ and $g$ are \emph{differentially equivalent},
    and denote it by $f \sim g$.
    A polynomial $f$ is \emph{differentially homogeneous} whenever it is differentially equivalent to some homogeneous polynomial $g$.
\end{definition}

It turns out that when $g$ is homogeneous, and the Frobenius quotients $K_f$ and $A_g$ are connected via the Chern character, the pair $(K_f, A_g)$ satisfies many properties one may expect from algebraic geometry.

\begin{definition}
\label{def:RR-pair-std-graded}
    When $f$ is a polynomial differentially equivalent to a homogeneous polynomial $g$, we call $(K_f, A_g)$ a \emph{Riemann-Roch pair}.

    The associated \emph{Todd} class is the element $\td \in \Q A_g^\times$, such that $f = \td \cdot g$.
\end{definition}

\subsection{Grothendieck-Riemann-Roch}\label{sec:grr}

In this subsection we study maps between Riemann-Roch pairs coming from maps of lattices, and state a version of the Grothendieck-Riemann-Roch theorem in our setting.

Let $\phi\colon \Lambda_1\to \Lambda_2$ be a homomorphism of lattices. Then $\phi$ induces a pullback map
$\phi^*:\Fun(\Lambda_2, \Q) \to \Fun(\Lambda_1, \Q)$, which sends polynomials to polynomials.
Let $\phi_*: \Sh(\Lambda_1) \to \Sh(\Lambda_2)$ be the pushforward given by $T^u \mapsto T^{\phi(u)}$. 
When $f_i$ are polynomials on $\Lambda_i$, these maps induce, under certain conditions, pushforwards and pullbacks between the quotient rings $K_{f_i}$.

\begin{warning}
    If $f: X \to Y$ is a proper morphism of smooth algebraic varieties, one also has the pushforwards and pullbacks between the respective $K$- and Chow rings. The map between lattices corresponds to the pullback on the Picard groups, and so goes in the opposite direction. For this reason, pullbacks (resp. pushforwards) of a map of lattices would correspond to pushforwards (resp. pullbacks) associated to the morphism of algebraic varieties via our combinatorial--geometric dictionary. Hence it may appear that in the following statements, pullbacks are interchanged with pushforwards, but this is just an artifact of our construction.
\end{warning}

The following lemma is the combinatorial analog of the projection formula,
which tells us how the pullback and pushforward are related.

\begin{lemma}\label{lemma:pushforward-pullback-relation}
For every $E\in \Sh(\Lambda_1)$ and $h\in \Fun(\Lambda_2,\Q)$, we have
\[
    \phi^*(\phi_*(E)\cdot h) = E\cdot \phi^*(h).
\]
\end{lemma}

\begin{proof}
Suppose $E=T^u$ with $u \in \Lambda_1$, then for any $v\in \Lambda_1$, we have that
\begin{align*}
    \phi^*(\phi_*(T^u)\cdot h)(v) &= (T^{\phi(u)}\cdot h) (\phi(v))\\
    &= h(\phi(u + v)) = (T^u \cdot \phi^*(h))(v),
\end{align*}
and hence $\phi^*(\phi_*(T^u)\cdot h) = T^u\cdot \phi^*(h)$.
The result then follows by linearity.
\end{proof}

Let $f_1, f_2$ be two polynomials on $\Lambda_1, \Lambda_2$ respectively.
Then $\phi_*$ descends to a map $K_{f_1} \to K_{f_2}$ if and only if $\phi_*(\Ann_\Sh(f_1)) \subseteq \Ann_\Sh(f_2)$.

If $\phi^*(\Sh(\Lambda_2)\cdot f_2) \subseteq \Sh(\Lambda_1) \cdot f_1$ (seen as subsets of $\Fun(\Lambda_i, \Q)$), then
$\phi^*$ induces a map $K_{f_2} \to K_{f_1}$ via the isomorphisms 
$K_{f_i} \cong \Sh(\Lambda_i) \cdot f_i$ defined by $\xi \mapsto \xi \cdot f_i$.

Over the rational numbers, these conditions are equivalent. 

\begin{proposition}\label{prop:pushforward-pullback-rational}
    $\phi^*(\Sh_\Q(\Lambda_2)\cdot f_2) \subseteq \Sh_\Q(\Lambda_1) \cdot f_1$ if and only if $\phi_*(\Ann(f_1)) \subseteq \Ann(f_2)$.
\end{proposition}
\begin{proof}
    If $\phi_*(\Ann(f_1)) \subseteq \Ann(f_2)$, then $\phi$ gives rise to the pushforward $\phi_*: \Q K_{f_1} \to \Q K_{f_2}$, and we also get a pullback $(\phi_*)^*: \Hom_\Q(\Q K_{f_2}, \Q) \to \Hom_\Q(\Q K_{f_1}, \Q)$. We have that $\Hom_\Q(\Q K_{f_i}, \Q) \cong \Sh_\Q(\Lambda_i) \cdot f_i$ via the map $\varphi \mapsto \varphi|_{\Lambda_i}$.
    \footnote{
    The fact that this isomorphism is well-defined can be seen using a discrete analog of \Cref{prop:solution set to differential equation}, where $\Psi$ is replaced by the restriction map $\varphi \mapsto \varphi|_\Sigma$.
    }
    The induced map $\Sh_\Q(\Lambda_2) \cdot f_2 \to \Sh_\Q(\Lambda_1) \cdot f_1$ is precisely the pullback of $\phi$, and so $\phi^*(\Sh_\Q(\Lambda_2)\cdot f_2) \subseteq \Sh_\Q(\Lambda_1) \cdot f_1$ holds.
    
    Similarly, if $\phi$ gives rise to a pullback $\phi^*: \Q K_{f_2} \to \Q K_{f_1}$, then we can recover the pushforward using the fact that
    $\Hom_\Q(\Q K_{f_i}, \Q) \cong \Q K_{f_i}$ via the Frobenius pairing.
\end{proof}

Similarly, we have a map $\phi_*: \Diff(\Lambda_1) \to \Diff(\Lambda_2)$.
It is clear from the definitions that both pushforwards $\phi_*$ commute with the Chern character. In particular, if $(K_{f_1}, A_{g_1})$
and $(K_{f_2}, A_{g_2})$ are two standard graded Riemann-Roch pairs, $\phi_*$ induces a map on the quotient $\Q A_{g_1} \to \Q A_{g_2}$ if and only if it induces a map on the quotient $\Q K_{f_1} \to \Q K_{f_2}$.
The same argument as in the proof of \Cref{prop:pushforward-pullback-rational} shows that this is equivalent to $\phi^*(\Diff_\Q(\Lambda_2)\cdot g_2) \subseteq \Diff_\Q(\Lambda_1) \cdot g_1$ and hence we also obtain a pullback $\phi^*: \Q A_{g_2} \to \Q A_{g_1}$.

This is enough to formulate the Grothendieck-Riemann-Roch theorem in our setting.

\begin{theorem}\label{thm:grr}
When $\phi$ is such that $\phi_*(\Ann_\Sh(f_1)) \subseteq \Ann_\Sh(f_2)$, we get that the following diagram commutes
\[
\begin{tikzcd}
\Q K_{f_2} \arrow[d, "\td_2\cdot \ch(-)"'] \arrow[r, "\phi^*"] & \Q K_{f_1} \arrow[d, "\td_1\cdot\ch(-)"] \\
\Q A_{g_2} \arrow[r, "\phi^*"]                          & \Q A_{g_1}                         
\end{tikzcd}
\]
\end{theorem}

\begin{proof}
It is clear from the definition of the pushforward that the diagram
\[
\begin{tikzcd}
\Q A_{g_1} \arrow[d, "\ch^{-1}"] \arrow[r, "\phi_*"] & \Q A_{g_2} \arrow[d, "\ch^{-1}"] \\
\Q K_{f_1} \arrow[r, "\phi_*"]                       & \Q K_{f_2}                    
\end{tikzcd}
\]
commutes. We get an induced commutative diagram
\[
\begin{tikzcd}
\Hom_\Q(\Q K_{f_2}, \Q) \arrow[r, "(\phi_*)^*"] \arrow[d, "(\ch^{-1})^*"] & \Hom_\Q(\Q K_{f_1}, \Q) \arrow[d, "(\ch^{-1})^*"] \\
\Hom_\Q(\Q A_{g_2}, \Q) \arrow[r, "(\phi_*)^*"]                           & \Hom_\Q(\Q A_{g_1}, \Q)                          
\end{tikzcd}
\]
Note that for any $\xi \in \Q K_{f_i}$, we have by \Cref{thm:map-descends-iff-hrr} that
\begin{align*}
    (\ch^{-1})^*(\chi_{f_i}(-\cdot \xi)) &= 
    \chi_{f_i}(\ch^{-1}(-) \cdot \xi)\\
    &= \deg_{g_i}(- \cdot \ch(\xi) \cdot \td_i).
\end{align*}
Using the identifications $\Q K_{f_i} \cong \Hom_\Q(\Q K_{f_i}, \Q)$
and $\Q A_{g_i} \cong \Hom_\Q(\Q A_{g_i}, \Q)$ coming from the Frobenius pairings, this diagram corresponds exactly to the diagram in the theorem.
\end{proof}

We now study further the conditions for when $\phi_*(\Ann_\Sh(f_1)) \subseteq \Ann_\Sh(f_2)$. Suppose that's the case, then \Cref{lemma:pushforward-pullback-relation} implies that $\Ann_\Sh(f_1) \subseteq \Ann_\Sh(\phi^*(f_2))$, which by \Cref{prop:polynomial-equivalence-inclusion} is the case if and only if $\phi^*(f_2) \in \Sh_\Q(\Lambda_1) \cdot f_1$. However, this is not a sufficient condition, as the following example shows.

\begin{example}
    Consider the following map of lattices:
\[
\phi\colon \Z \to \Z^2,\quad x\mapsto (x,0).
\]
Let $g_1(x)=0$ and $g_2(x,y) = xy$. Then  $g_1=\phi^*(g_2)$ and, in particular, 
$\phi^*(g_2) \in \Sh_\Q(\Lambda_1) \cdot g_1$. However, we get 
\[
\Ann_{\Sh(\Z)}(g_1)=\Sh(\Z),\quad \Ann_{\Sh(\Z^2)}(g_2)=\langle D_x^2,D_y^2\rangle,
\]
and in particular the map $\phi$ does not induce a map $K_{g_1}\to K_{g_2}$.
\end{example}

Choose a decomposition $\Lambda_2 = \Lambda \oplus \Lambda'$, where $\phi(\Lambda_1)\subseteq \Lambda$ is full-rank.
Then we have an isomorphism between the rings of polynomials $P(\Lambda_2) \cong P(\Lambda) \otimes P(\Lambda')$.
In particular, we may write
\[
    f_2 = \sum_{i} h_i \otimes h_i',
\]
where $h_i \in P(\Lambda)$, and $h_i' \in P(\Lambda')$ for all $i$, with the $h_i'$ being linearly independent.

\begin{proposition}\label{prop:push-forward}
We have that $\phi_*(\Ann_\Sh(f_1)) \subseteq \Ann_\Sh(f_2)$ if and only if
$\phi^*(h_i) \in \Sh_\Q(\Lambda_1)\cdot f_1$ for all $i$.
\end{proposition}

\begin{proof}
    We have the following chain of equivalences:
    \begin{align*}
        \phi_*(\Ann_\Sh(f_1)) \subseteq \Ann_\Sh(f_2) &\iff
        \phi_*(E)\cdot f_2 = 0\text{ for every }E\in \Ann_\Sh(f_1)\\
        &\iff
        \phi_*(E) \cdot h_i = 0\text{ for all }i\text{ and }E \in \Ann_\Sh(f_1)\\
        &\iff
        E \cdot \phi^*(h_i) = 0\text{ for all }i\text{ and }E \in \Ann_\Sh(f_1)\\
        &\iff
        \Ann_\Sh(f_1) \subseteq \Ann_\Sh(\phi^*(h_i))\text{ for all }i.
    \end{align*}
    The second equivalence follows from the fact that the $h_i'$ are assumed to be linearly independent. The third equivalence is \Cref{lemma:pushforward-pullback-relation} along with the fact that $\phi^*: P(\Lambda) \to P(\Lambda_1)$ is injective, because $\phi(\Lambda_1)$ is full rank in $\Lambda$.
    The conclusion follows by \Cref{prop:polynomial-equivalence-inclusion}.
\end{proof}

\subsection{Multiplicative maps}
\label{subsec:multiplicative-maps}

In this section, we present a general approach to constructing maps on $K_f$. This will allow us to define Chern and Todd classes, determinants, and exterior powers of elements of $K_f$. Note that in the geometric setting (when the pair $(K_f, A_g)$ corresponds to the numerical $K$-ring and Chow ring of a smooth algebraic variety), these constructions will recover the geometric notions carrying the same name.

\begin{theorem}
    \label{thm:mult-map-A_g}
    Let $R$ be a ring of characteristic $0$ and let
    $a \in R \llbracket x \rrbracket$ be a formal series with $a(0)$ invertible in $R$.
    Then there exists a unique multiplicative map
    \[
        \psi\colon K_f\to R\otimes A_g
    \]
    with $\psi([T^u])=a([\partial_u])$ for all $u\in \Lambda$.    
\end{theorem}

\begin{proof}

Let $\psibar: \Sh(\Lambda) \to R \otimes A_g$ be the map uniquely determined by 
$\psibar(T^u) = a([\partial_u])$ for all $u \in \Lambda$. Since for any $u \in \Lambda$, the degree-$0$ coefficient of $a([\partial_u])$ is $a(0)$, which is invertible, we have that $a([\partial_u])$ is also invertible, and so the map is well-defined. The statement of the theorem is now equivalent to the statement that $\psibar$ descends to the quotient, yielding a map $\psi: K_f \to R \otimes A_g$. Since $K_f = K_g$, we may assume w.l.o.g. that $f=g$, and so that $f$ is homogeneous.

We claim that it suffices to show that if for all $E, F \in \Sh(\Lambda)^+$ with $E\cdot f = F\cdot f$, we have $\psibar(E) = \psibar(F)$.
Indeed, suppose that's the case, and let $E, F \in \Sh(\Lambda)$ not necessarily positive, with
$E \cdot f = F \cdot f$.
Split $E = E_+ - E_-$ and $F = F_+ - F_-$ with $E_\pm, F_\pm$ positive, then $E \cdot f = F \cdot f$ is equivalent to $(E_+ + F_-) \cdot f = (F_+ + E_-)\cdot f$. 
Then by assumption we have that $\psibar(E_+ + F_-) = \psibar(F_+ + E_-)$.
Multiplying this equality by $\psibar(-E_- -F_-)$, we obtain $\psibar(E) = \psibar(F)$, showing that $\psibar$ descends to the quotient.

So suppose $E$, $F$ are positive elements with $E \cdot f = F \cdot f$.
\Cref{lemma:annihilator-rank-0} implies that $E$ and $F$ have the same rank $d\coloneqq\rk(E)=\rk(F)$.
Hence we may write $E=\sum_{i=1}^d T^{u_i}$ and $F=\sum_{i=1}^d T^{v_i}$ for some $u_i,v_i\in \Lambda$.

By \Cref{lem:exponential and shifts}, we have 
\begin{equation}
\label{eq:apply chern character to prove lambda-ideal}
\ch(E)\cdot f=E \cdot f=F \cdot f=\ch(F) \cdot f.  
\end{equation}
Note that 
\begin{equation*}
    \ch(E)= \sum_{i=1}^d \ch(T^{u_i})=\sum_{k=0}^\infty \frac 1{k!} p_k(\partial_{u_1}, \dots, \partial_{u_d}) \ ,
\end{equation*}
where $p_k$ is the degree-$k$ power sum symmetric polynomial. A similar equality holds for $F$.
Since we assumed $f$ is homogeneous, \Cref{eq:apply chern character to prove lambda-ideal} implies that 
\[
p_k(\partial_{u_1}, \dots, \partial_{u_d}) \cdot f=
p_k(\partial_{v_1}, \dots, \partial_{v_d}) \cdot f.
\]
for all $k>0$. We note that the set 
\[
W=\{q\in R\otimes \Q[t_1,\ldots,t_d]\mid
q(\partial_{u_1}, \dots, \partial_{u_d}) \cdot f=
q(\partial_{v_1}, \dots, \partial_{v_d}) \cdot f\}
\]
defines a subring of $R\otimes \Q[t_1, \dots, t_d]$.
As the power sum symmetric polynomials generate the ring of symmetric polynomials (since $R \otimes \Q \supseteq \Q$ as $R$ is of characteristic $0$), it follows that the ring of symmetric polynomials is contained in $W$.
In particular, this implies that
\[
    \psibar(E) \cdot f = \left( \prod_{i = 1} ^ d a(\partial_{u_i}) \right) \cdot f
    = \left( \prod_{i = 1} ^ d a(\partial_{v_i}) \right) \cdot f = \psibar(F) \cdot f,
\]
which finishes the proof.
\end{proof}

\begin{definition}\label{def:chern-class-todd}
\begin{enumerate}
    \item
    The \emph{total Chern class} is the multiplicative map $c: K_f \to \Q A_g$ generated by
    \[
        c([T^u]) = 1 + [\partial_u].
    \]
    The $k$-th graded component of $c(\xi)$ is denoted by $c_k(\xi)$ and is called the $k$-th Chern class of $\xi$.
    \item
    The \emph{Todd operator} is the multiplicative map $\Td: K_f \to \Q A_g$ generated by
    \[
        \Td([T^u]) = \frac{[\partial_u]}{1 - \exp(-[\partial_u])}.
    \]
\end{enumerate} 
\end{definition}

\begin{corollary}
    Let $p \in R[x]$ be a polynomial with $p(1)$ invertible in $R$.
    There is a unique multiplicative map
    $\mu: K_f \to R \otimes K_f$, such that for all $u \in \Lambda$,
    $\mu([T^u]) = p([T^u])$.
\end{corollary}

\begin{proof} 
    Set $a = p \circ \exp$ and let $\psi: K_f \to R \otimes A_g$ be the map defined in \Cref{thm:mult-map-A_g}, and let $\mu: K_f \to R \otimes \Q K_f$ be the map given by the composition $\mu := \ch^{-1} \circ \psi$.
    Then for any $u \in \Lambda$, we have that 
    \[\mu([T^u]) = \ch^{-1}\left(p(\exp([\partial_u])\right)
    = p(\ch^{-1}(\exp([\partial_u]))) = p([T^u]).\]
    We claim that the image of $\mu$ is contained in $R \otimes K_f$.
    Let $u \in \Lambda$,
    then $p([T^u]) = p(1 + [\Delta_u]) = p(1) + \xi$, where $\xi \in [\Delta_u]\cdot (R\otimes K_f)$ is a nilpotent element in $R \otimes K_f$. It follows that $p([T^u])$ is invertible in $R\otimes K_f$, and hence in particular
    $\mu(-[T^u]) = p([T^u])^{-1} \in R \otimes K_f$, which shows that $\im(\mu) \subseteq R \otimes K_f$.
\end{proof}

\begin{definition}\label{def:det-lambda}
\begin{enumerate}
    \item []
    \item 
    The \emph{determinant} map is the multiplicative map $\det: K_f \to K^1_f$
    generated by
    \[\det([T^u]) = [T^u].\]

    \item
    The \emph{$\lambda$-operator} is the multiplicative map
    $\lambda: K_f \to K_f\llbracket t \rrbracket$ generated by
    \[\lambda([T^u]) = 1 + [T^u] t.\]
    We denote by $\lambda^k(\xi) \in K_f$ the coefficient of $t^k$ in $\lambda(\xi)$
\end{enumerate} 
\end{definition}

\begin{remark}
    For any positive element $\xi \in K_f^+$, we have that $\det(\xi) = \lambda^{\rk(\xi)}(\xi)$.
\end{remark}

\subsection{\texorpdfstring{$\lambda$}{λ}-ring structure}

The $\lambda$-operator allows us to equip $K_f$ with the structure of a $\lambda$-ring.
In the geometric setting, this structure arises naturally by taking exterior powers of vector bundles.

Let $\xi = \sum_{i = 1}^r [T^{u_i}]$ be a positive element, where $u_i \in \Lambda$.
From the definition it follows that $\lambda^k(\xi) = [\sigma_k(T^{u_1}, \dots, T^{u_r})]$,
where $\sigma_k$ is the elementary symmetric polynomial of degree $k$.

\begin{theorem}\label{thm:lambda-ring}
    The $\lambda$-operator induces on $K_f$ the structure of a $\lambda$-ring.
    Furthermore, $\lambda$ extends multiplicatively to $\Q K_f$, also making it a $\lambda$-ring.
\end{theorem}

\begin{proof}
    The fact that $K_f$ is a $\lambda$-ring follows from \cite[p. 17]{Knu73}.
    
    Note that $\Q$ is a binomial ring, which means that the binomial coefficient
    \[
        \binom{a}{k} := \frac{a(a-1) \dots (a-k + 1)}{k!}
    \]
    is well-defined. Setting $\lambda^k(a) = \binom{a}{k}$ endows $\Q$ with the structure of a $\lambda$-ring, and hence also the tensor product $\Q K_f = K_f \otimes_\Z \Q$, by \cite[p. 21]{Knu73}.
\end{proof}

The $\lambda$-ring structures on $K_f$ and $\Q K_f$ as constructed above agree with the natural $\lambda$-ring structures on $\Sh(\Lambda)$ and $\Sh_\Q(\Lambda)$. 
We may evaluate the $\lambda$-operator on $\Sh_\Q(\Lambda)$ as follows. One may show using the splitting principle that for all $a \in \Q$, $u \in \Lambda$, 
\[
    \lambda(aT^u)
    = \sum_{k=0}^\infty \lambda^k(a) T^{ku} t^u
    = \sum_{k=0}^\infty \binom{a}{k}T^{ku}t^u =: (1 + T^ut)^a.
\]
It follows that for any $\sum_i a_i T^{u_i} \in \Sh_\Q(\Lambda)$,
\begin{equation}
\label{rem:lambda-rational}
    \lambda\left(\sum_i a_iT^{u_i}\right)
    = \prod_i \lambda(a_iT^{u_i})
    = \prod_i (1 + T^{u_i}t)^{a_i}.
\end{equation}

\subsection{Serre duality}\label{subsec:serre-duality}

Let $X$ be a smooth variety of dimension $d$, which is proper over some field $k$. Serre duality states that for any algebraic vector bundle $E$ on $X$, and any $0 \leq i \leq d$, we have that
\[
    H^i(X, E) \cong H^{d-i}(X, \omega_X \otimes E^\vee)^\vee,
\]
where $\omega_X$ is the canonical line bundle of $X$.

Since the Euler characteristic of $E$ is defined by
\[
\chi(X, E) = \sum_{i \geq 0} (-1)^i \dim_k H^i(X, E),
\]
Serre duality implies the relation
\[
\chi(X, E) = (-1)^d \chi(X, \omega_X \otimes E^\vee).
\]

We can make sense of this statement for Riemann-Roch pairs $(K_f, A_g)$, but first we have to translate the dualizing morphism into our setting.

\begin{proposition} \label{prop:dualizing-morphism}
    There is an involution $(-)^\vee$ on $K_f$, generated by
    $[T^u]^\vee = [T^{-u}]$. We call this the \emph{dualizing morphism} on $K_f$.
\end{proposition}

\begin{proof}
    On $\Diff_\Q(\Lambda)$ we have an involution generated by $(\partial_u)^\vee = -\partial_{u}$.
    Since $\Ann(g)$ is a homogeneous ideal and $(-)^\vee$ acts on each graded component $\Diff_\Q^i(\Lambda)$ by multiplication with $(-1)^{i}$, this involution descends to the quotient
    $\Diff_\Q(\Lambda)/\Ann(g) = \Q A_g$.
    We may pull back this involution via the Chern character, to obtain
    an involution on $K_f$, which by construction satisfies
    $[T^u]^\vee = [T^{-u}]$.
\end{proof}

\begin{definition}
    \label{def:dualizing-class}
    An element of $\Q K_f$ is called the \emph{dualizing class} if for all $\xi \in \Q K_f$,
    \[\chi_f(\xi) = (-1)^d\chi_f(\omega \cdot \xi^\vee),\]
    where $d=\deg(g)$
\end{definition}

\begin{proposition}\label{prop:serre duality}
    The dualizing class always exists in $\Q K_f$
    and is unique.
\end{proposition}

\begin{proof}
    Consider the $\Q$-linear map
    \begin{align*}
        \Q K_f &\to \Q\\
        \xi &\mapsto (-1)^d \chi_f(\xi^\vee),
    \end{align*}
    where $d$ is the degree of $f$.
    Then because the Frobenius pairing on $\Q K_f$ is perfect, this map is the dual of a unique class $\omega \in \Q K_f$.
\end{proof}

Note that without any assumption on $f$, the dualizing class $\omega$ might be not be integral, and even if that was the case, it might not be a line element. However, at least $\omega$ has the expected rank.

\begin{lemma}
    We have that $\rk(\omega) = 1$.
\end{lemma}
\begin{proof}
    We may rewrite the equation in \Cref{def:dualizing-class} as
    \[f(x) = (-1)^d (\omega \cdot f)(-x).\]
    The result follows by comparing the top homogeneous components of both sides.
\end{proof}

\begin{proposition}
    Write $\omega = \sum_{i} a_i [T^{u_i}]$.
    Then $\td_1 = -\frac{1}{2}\sum_{i}a_i[\partial_{u_i}] = -\frac{1}{2}c_1(\omega)$.
\end{proposition}

\begin{proof}
    We have that
    \[
        f(-x) = (-1)^d\sum_i a_i f(x + u_i)
        = (-1)^d\sum_i\sum_{k} a_i\frac{1}{k!}(\partial_{u_i}^kf)(x).
    \]
    Write $\td = 1 + \td_1 + \dots + \td_d$, where $\td_i$ is homogeneous component of degree $i$, then if we replace $f = \td \cdot g$ and equate the degree $d-1$ homogeneous parts in the above equation, we obtain 
    \[
        (-1)^{d+1}\td_1\cdot g = (-1)^d\sum_i a_i (\td_1 \cdot g + \partial_{u_1} g).
    \]
    Note that $\sum a_i = 1$, and so we get
    \[\td_1 \cdot g = -\frac{1}{2}\sum_i a_i \partial_{u_i} g,\]
    which finishes the proof.
\end{proof}

In general $\omega$ is not a line element, but the following proposition specifies under which conditions on $\td$ it turns out to be the case.

\begin{proposition}\label{prop:dualizing-class-characterization}
    The dualizing class $\omega$ is a line element if and only if $\exp(-\td_1) \cdot \td$ is zero in odd degrees in $\Q A_g$ and $\td_1 = -[\partial_w]/2$ for some $w\in \Lambda$.
\end{proposition}

\begin{proof}

Suppose $\omega = [T^w]$ for some $w \in \Lambda$,
then $\td_1 = -c_1(\omega)/2= -[\partial_w]/2$. Moreover, the defining property of $\omega$ is given by the equation
\begin{equation}\label{eq:serre-line-bundle}
f(-x) = (-1)^d f(x + w).
\end{equation}
In particular, if we define $h(x) = f(x + w/2)$, then $h$ is $(-1)^d$-symmetric. Note that $h = \alpha \cdot g$, where $\alpha = \exp([\partial_w]/2) \cdot \td$. Denote by $\alpha_i$ the homogeneous component of $\alpha$ of degree $i$. Then $h(-x) = (-1)^d h(x)$ implies that $\alpha_i = 0$ for all odd $i$.

Reciprocally, 
suppose $\exp(-\td_1) \cdot \td$ is zero in odd degrees and $\td_1 = -[\partial_w]/2$ for some $w\in \Lambda$.
Then the function 
\[ h = \exp([\partial_w]/2) \cdot \td \cdot g\]
is $(-1)^d$-symmetric.
Since $h(x) = f(x + w/2)$, it follows that \Cref{eq:serre-line-bundle} holds and
hence the dualizing class is the line element $\omega = [T^w]$.
\end{proof}

\begin{example}
    In \Cref{subsec:linear-families}, we show how the canonical class of a linear family of polytopes in the sense of \cite{kavvil} may be interpreted as the dualizing class in a corresponding discrete Frobenius quotient.
\end{example}

\subsection{Cotangent class}

When $X$ is a smooth complete scheme, its dualizing sheaf is the canonical line bundle, that is,
it is a line element.
This tells us that in order to model what is happening on the algebro-geometric side, we should restrict the class of admissible choices of $\td \in A_g^\times$. \Cref{prop:dualizing-class-characterization} gives a first property we should expect holds.
Recall that the Todd class of $X$ is obtained by applying the Todd operator to the tangent bundle. Furthermore, the canonical bundle is just the determinant of the cotangent bundle. These observations hint at the fact that instead of fixing $\td \in A_g^\times$, we can fix an analog of the (co)tangent bundle.

\begin{remark}
    Following the above intuition, in \cite[Theorem 1.1]{Cheng2025}, the author defines the (co)tangent class in the $K$-ring of a matroid, which satisfies, among other conditions, the desired relation with the Todd class. In \cite[Proposition 3.20]{Cheng2025}, they express the dualizing class as the determinant of the cotangent class and use it to write down Serre duality for matroids.
\end{remark}

Cheng's result generalizes readily to our setting.

\begin{proposition}
\label{prop:cotangent-todd-dualizing}
    Suppose there is an element $\Omega \in K_f$, such that $\td = \Td(\Omega^\vee)$.
    We call the choice of such an element a \emph{cotangent class}.
    Then $\omega = \det(\Omega)$ is the dualizing class of $\Q K_f$, that is, the equation
    \begin{equation}\label{eq:serre-duality}
        \chi_f(\xi) = (-1)^d \chi_f(\omega \cdot \xi^\vee)
    \end{equation}
    holds for all $\xi \in \Q K_f$.
\end{proposition}

\begin{proof}
    By the Hirzebruch-Riemann-Roch theorem, the left hand side is equal to \[\deg_g(\ch(\xi)\cdot \Td(\Omega^\vee))\] and the right hand side to
    \begin{align*}
        (-1)^d\deg_g(\ch(\omega) \cdot \ch(\xi)^\vee \cdot \Td(\Omega^\vee)) &= 
        \deg_g(\ch(\xi) \cdot \ch(\omega)^\vee \cdot \Td(\Omega))\\
        &=\deg_g\big(\ch(\xi) \cdot \big(\ch(\det(\Omega))^\vee \cdot \Td(\Omega)\big)\big)
    \end{align*}
    where we used the fact $\deg_g(\alpha^\vee) = (-1)^d\deg_g(\alpha)$, since $g$ is homogeneous, and that $\Td(\Omega^\vee) = \Td(\Omega)^\vee$.
    
    A standard computation (as in \cite[Proposition 5.2]{FultonLang85}, or \cite[Proposition 3.20]{Cheng2025}) shows that
    \[\ch(\det(\Omega))^\vee\Td(\Omega) = \Td(\Omega^\vee).\]
    Indeed, we may write $\Omega = \sum_i a_i \ell_i$ for some line elements $\ell_i$ and $a_i \in \Z$.
    Denote also $\alpha_i = c_1(\ell_i)$.
    Then
    \begin{align*}
        \ch(\det(\Omega))^\vee\Td(\Omega) = \exp(-c_1(\Omega))\Td(\Omega)
        &= \prod_i\left(\exp(-\alpha_i)\frac{\alpha_i}{1 - \exp(-\alpha_i)}\right)^{a_i}\\
        &= \prod_i\left(\frac{-\alpha_i}{1 - \exp(\alpha_i)}\right)^{a_i} = \Td(\Omega^\vee).
    \end{align*}
    Therefore, we get that \Cref{eq:serre-duality} holds for $\omega = \det(\Omega)$.
\end{proof}

\subsection{Associated graded ring of \texorpdfstring{$K_f$}{Kf}}
\label{sec:associated-graded}

In this section we define the topological filtration on $K_f$, which yields an associated graded ring $\gr K_f$. We show that there is a canonical isomorphism $\Q\gr K_f \cong \Q A_g$.

Let $I = \langle [D_u] \mid u \in \Lambda \rangle \subseteq K_f$ be the ideal generated by the classes of all the difference operators $D_u = 1 - T^{-u}$. This ideal is prime, because it's equal to $\ker(\rk)$. Note we may also write $I = \langle [\Delta_u] \mid u \in \Lambda \rangle$.
This defines a filtration on $K_f$:
\[
    K_f = I^{(0)} \supset I^{(1)} \supset I^{(2)} \supset \dots\,,
\]
where $I^{(i)} = \{\xi \in K_f \mid \eta\xi \in I^i\textrm{ for some }\eta \in K_f\setminus I\}$ is the $i$-th symbolic power of the ideal $I$. This filtration is multiplicative, and so
we obtain an associated graded ring
\[
    \gr K_f = \bigoplus_{i=0}^\infty I^{(i)}/I^{(i+1)}.
\]
Note that since $\Z \subseteq K_f \setminus I$ and the elements of $K_f \setminus I$ are invertible in $\Q K_f$,
we have that $(K_f)_I = \Q K_f$, and so in particular $I^{(i)} = \Q I^i \cap K_f$.

\begin{lemma}
    \label{lem:filtration-characterization}
We have equality
    \[I^{(i)} = \{\xi \in K_f \mid \xi \cdot f \textrm{ has degree at most $d-i$}\}.\]
\end{lemma}

\begin{proof}
W.l.o.g.\ $f = g$ is homogeneous.
First, suppose we have $\xi \in I^{(i)}$, then there exists some $\eta \in K_f \setminus I$ such that
$\eta\xi \in I^i$. In particular, $\eta\xi \cdot f$ has degree at most $d-i$. Note however that since $\eta \notin I$, $\ch(\eta)$ has non-zero constant term and so $\deg(\eta\cdot(\xi \cdot f)) = \deg(\xi \cdot f)$, which implies that $\xi \cdot f$ has degree at most $d-i$.

For the other inclusion, let $\xi \in K_f$ be such that $\xi \cdot f$ has degree $d-i$. 
Since $\xi \cdot f = \ch(\xi) \cdot f$, all the homogeneous terms of $\ch(\xi)$ of degree lower than $i$ have to annihilate $f$. Let $\alpha \in \Diff_\Q(\Lambda)$ a representative of $\ch(\xi)$ which has no homogeneous component in degree lower than $i$. Now, consider the map $\zeta: \Diff_\Q(\Lambda) \to \Sh_\Q(\Lambda)$ given by
\[
\zeta(\partial_u)
= \sum_{k=1}^d (-1)^{k + 1}\frac{\Delta_u^k}{k}.
\]
Since higher powers of difference operators annihilate $f$, this map agrees with $\ch^{-1}$ modulo the annihilator of $f$. Let $k$ be large enough, so that $\zeta(k\alpha) \in \Sh(\Lambda)$.
Then we get that $k\xi = [\zeta(k \alpha)] \in \Q I^i \cap K_f = I^{(i)}$.
\end{proof}

\begin{theorem}\label{thm:ass-graded}
    There is an isomorphism of graded rings $\Phi: \Q A_g \xrightarrow{\sim} \Q \gr K_f$,
    which sends $[\partial_u]$
    to $[D_u]$ in the degree 1 graded component of $\Q \gr K_f$ for all $u \in \Lambda$.
\end{theorem}

\begin{proof}
    We need to check that this map is well-defined.

    Let $\widetilde{\Phi}: \Diff_\Q(\Lambda) \to \Q \gr K_f$ be the ring homomorphism obtained
    by sending $\partial_u$
    to $[D_u]$ in the degree 1 graded component of $\gr K_f$ for all $u \in \Lambda$.
    
    This map is well-defined, since for all $u, v \in \Lambda$,
    \begin{align*}
    D_{u + v} &= 1 - T^{-u-v}\\
    &= 1 - T^{-u} + T^{-u} - T^{-u-v}\\
    &= D_u + T^{-u}D_v\\
    &\sim D_u + D_v\mod I^{(2)}
    \end{align*}
    The congruence on the last line holds, because $(1 - T^{-u})D_v = D_uD_{v} \sim 0$ in $\Q I^{(1)}/\Q I^{(2)}$.

    We now check that $\widetilde{\Phi}$ descends to the quotient, showing that $\Phi$ is well-defined.
    Note that two elements $\xi, \eta \in \Q I^{(k)}$ represent the same class in $\Q \gr^k K_f$ iff $\xi - \eta \in \Q I^{(k + 1)}$.
    By \Cref{lem:filtration-characterization}, this means
    that $(\xi - \eta) \cdot f$ is of degree at most $d - k - 1$, or in other terms, 
    $\xi \cdot f$ and $\eta \cdot f$ have the same homogeneous component of degree $d - k$.
    We will denote this by $[\xi \cdot f]_{d - k} = [\eta \cdot f]_{d - k}$.
    
    Let $D \in \Q \Ann_\Diff(g)$ be homogeneous of degree $k$. We have that
    \[
        D_u \cdot g = (1 - \exp(-\partial_u)) \cdot g = (\partial_u + \textrm{``higher order terms"}) \cdot g.
    \]
    Hence it follows that
    \[
        [\widetilde{\Phi}(D) \cdot f]_{d - k} = [(D + \textrm{``higher order terms"}) \cdot \td \cdot g]_{d - k}
        = [D \cdot \td \cdot g]_{d - k} = 0.
    \]
    and therefore $\widetilde{\Phi}(D)$ is $0$ in $\Q \gr K_f$. We conclude that $\Phi$ is well-defined.

    Note that $\Phi$ is clearly surjective. Furthermore, $\Q A_g$ and $\Q \gr K_f$ have the same dimension as $\Q$-vector spaces, since the Chern character yields an isomorphism $\Q A_g \cong \Q K_f$ and $\gr \Q K_f$ has the same dimension as $\Q K_f$. It follows that $\Phi$ is an isomorphism.
\end{proof}

\begin{remark}
    Note that this filtration agrees with the $\gamma$-filtration defined in \cite[Ch. III]{FultonLang85}.
\end{remark}

\subsection{Exceptional isomorphisms}
\label{sec:exceptional-isoms}

In \cite{LLPP24}, an exceptional isomorphism $K(M)\to A(M)$ between the integral $K$-ring and Chow-ring of a loopless matroid $M$ is given. If $M$ is represented by a hyperplane arrangement $\cA$, then $K(M)$ and $A(M)$ are isomorphic to the $K$-ring and Chow-ring of the wonderful variety $\mc W_\cA$, and there is an explicit generating set consisting of line bundles $[\mc L_F^{-1}]$ for each flat $F \in \cL(M)$ (in the notation of \emph{loc.\ cit.}) which are sent to $1+c_1(\cL_F^{-1})$ under the exceptional isomorphism. This is precisely the truncation of the Chern character of $\cL_F^{-1}$ at degree $1$, motivating the study of such maps more generally.

In this section, we study under which conditions the truncated Chern character gives a well-defined isomorphism $K_f\to A_g$. More generally, we consider maps $\Sh_\Q(\Lambda) \to \Q A_g$ which act symmetrically on a fixed basis, and characterize when they descend to isomorphisms
$\Q K_f \cong \Q A_g$. 
In case no denominators appear, as is the case for the truncated Chern character, these isomorphisms are integral, meaning that they restrict to an isomorphism $K_f \cong A_g$. Given an exceptional isomorphism, we obtain as a direct consequence of \Cref{thm:map-descends-iff-hrr} an exceptional Todd class and a corresponding Hirzebruch-Riemann-Roch formula.

Fix a basis $e_1, \dots, e_n$ of $\Lambda$, and a series of the form
\[
    p(t) = 1 + \sum_{i = 1}^d a_i t^i  \in \Q\llbracket t \rrbracket
\] with $a_1 \neq 0$.
Define the map
\begin{align*}
    \varphi_p: \Sh_\Q(\Lambda) &\to \Diff_\Q\llbracket \Lambda \rrbracket\\
    T^{e_i} &\mapsto p(\partial_{e_i})
\end{align*}
Using the chosen basis, we identify the ring of polynomials on $\Lambda$ with $\Q[x_1, \dots, x_n]$.

In the following lemma we define, and characterize a linear mapping $\Psi_p: P(\Lambda) \to P(\Lambda)$, which will help us understand when $\varphi_p$ descends to an isomorphism.

\begin{lemma}\label{prop:psi-formula}
    For any polynomial $h \in P(\Lambda)$, let
    $\Psi_p(h)$ be the function on $\Lambda$ defined by
    $\Psi_p(h)(u) = (\varphi_p(T^u) \cdot h) (0)$.
    Then $\Psi_p$ defines a linear mapping $P(\Lambda) \to P(\Lambda)$. Furthermore,
    if we set
    \begin{equation*}
    \psi_p(t^l) := \sum_{\substack{\mathbf{k} \in \N^d\\ \mathbf{k} \cdot [d] = l}} \binom{t}{\mathbf{k}} \mathbf{a}^\mathbf{k}l!\ ,
    \end{equation*}
    where $[d] = (1, \dots, d)$, $\vec a = (a_1, \dots, a_d)$ and
    \[
    \binom{t}{\vec k} := \frac{t(t - 1) \cdots (t - |\mathbf{k}| + 1)}{\mathbf{k}!},
    \]
    then we have
    \[
        \Psi_p\left(\prod_{i = 1}^n x_i^{k_i}\right) = \prod_{i= 1}^n \psi_p(t^{k_i})|_{t = x_i}.
    \]
\end{lemma}

\begin{proof}
    Linearity is clear, so it suffices to show the formula for $h = \prod_{i = 1}^n x_i^{l_i}$ a monomial in $\Q[x_1, \dots, x_n]$.
    Then for any $\vec u \in \Lambda$,
\begin{align*}
    \Psi_p(h)(\vec u) &= (\varphi_p(T^{\vec u}) \cdot h) (0) = \left(\prod_{i = 1}^n p(\partial_i)^{u_i}\cdot h\right)(0)\\
    &= \left(\prod_{i = 1}^n \left(1 + \sum_{j = 1}^d a_j \partial_i^j\right)^{u_i} \cdot h\right)(0)\\
    &= \left(\left(\prod_{i = 1}^n\sum_{\mathbf{k} \in \N^d} 
    \binom{u_i}{\mathbf{k}}\mathbf{a}^\mathbf{k}\partial_i^{\sum_{j = 1}^d j k_j}\right) \cdot h\right)(0),
\end{align*}
noting that this also holds if $u_i < 0$ by the generalized multinomial theorem.
Using the fact that $h = \prod_{i = 1}^n x_i^{l_i}$ is a monomial, we get from the above
\begin{equation*}
\Psi_p(h)(\vec u)
= \prod_{i = 1}^n
\sum_{\substack{\mathbf{k} \in \N^d\\ \mathbf{k} \cdot [d] = l_i}} \binom{u_i}{\mathbf{k}} \mathbf{a}^\mathbf{k}l_i!\ .\qedhere
\end{equation*}
\end{proof}

The next result tells us how $\Psi_p$ interacts with the action of shift operators on $P(\Lambda)$.

\begin{lemma}\label{lem:Psi-shift}
    For any $E \in \Sh_\Q(\Lambda)$ and $h \in P(\Lambda)$, we have that
    \[
        E \cdot \Psi_p(h) = \Psi_p(\varphi_p(E) \cdot h).
    \]
\end{lemma}

\begin{proof}
    By linearity of $\Psi_p$ and $\varphi_p$, it suffices to show this for simple shifts.
    Let $v \in \Lambda$, then
    \begin{align*}
        (T^v \cdot \Psi_p(h))(u) &= \Psi_p(h)(u + v)\\
        &= 
        (\varphi_p(T^{u + v}) \cdot h)(0)\\
        &= (\varphi_p(T^u)\varphi_p(T^v) \cdot h)(0)\\
        &= \Psi_p(\varphi_p(T^v) \cdot h)(u).
        \qedhere
    \end{align*}
\end{proof}

Using the above two lemmas, we are able to completely characterize when $\varphi_p$ descends to an isomorphism.

\begin{theorem}\label{thm:exceptional-general}
    The following are equivalent
    \begin{enumerate}
        \item The map $\varphi_p$ descends to an isomorphism $\Q K_f \cong \Q A_g$.
        \item There exists $\td_p \in \Q A_g^\times$ such that the equation 
            \begin{equation*}
                \chi_f(E) = \deg_g(\td_p \cdot \varphi_p(E))
            \end{equation*}
            holds for all $E \in \Sh(\Sigma)$.
        \item There exists $\td_p \in \Q A_g^\times$ such that
            \[
                f = \Psi_p(\td_p \cdot g).
            \]
        \item $f$ is differentially equivalent to $\Psi_p(g)$.
    \end{enumerate}
    
\end{theorem}

\begin{proof}
    Note that since $a_1 \neq 0$, the composite map
    \[\Sh_\Q(\Lambda) \xrightarrow{\varphi_p} \Diff_\Q\llbracket \Lambda\rrbracket
    \to \Q A_g\]
    contains the elements $[\partial_{e_i}]$ in its image, and since these generate $\Q A_g$, it is surjective.

    (1) is equivalent to (2) by \Cref{thm:map-descends-iff-hrr}.
    (2) is equivalent to (3) by \Cref{rem:hrr-equivalent-statement} and the definition of $\Psi_p$.
    Note that for all $E \in \Sh(\Sigma)$, $\varphi_p(E)$ has constant term $\rk(E)$, and hence $\varphi_p(E)$ is invertible if and only if $E$ is invertible. It follows using \Cref{lem:Psi-shift} and surjectivity of $\varphi_p$, that (3) is equivalent to (4).
\end{proof}

\begin{example}\label{ex:exceptional-full-flag}
In this example we would like to study exceptional isomorphisms for full flag variety in type $A$, that is, for $G = \SL_n$. Recall from \Cref{ex:full-flags} that the volume polynomial on flag variety $G/B$ is given by
\[
 \Vol_{G/B}(\lambda) = \frac{\prod_{\alpha\in \Delta^+}(\lambda,\alpha)}{\prod_{\alpha\in \Delta^+}(\rho,\alpha)},
\]
which in Type A gives
\[
\Vol_{G/B}(\lambda_1,\ldots, \lambda_n) = \prod_{i<j}\frac{\lambda_i-\lambda_j}{i-j}.
\]
Thus up to constant factor we get 
\[
\Vol_{G/B}(\lambda) =    
    \prod_{i<j}(\lambda_j-\lambda_i)=\det \begin{pmatrix}
1  & \cdots & 1\\
\lambda_1 & \cdots & \lambda_n\\
\vdots & \ddots & \vdots\\
\lambda_1^{n-1} & \cdots & \lambda_n^{n-1}
\end{pmatrix}.
\]

Now let $p(t) = 1 + \sum_{i = 1}^d a_i t^i$ be a polynomial, then by definition
$\psi_p(t^k)$ is of degree at most $k$ with coefficient of $t^k$ equal to $a_1^k$.
Since $a_1 \neq 0$, we have that $\psi_p$ sends monomials to polynomials of the same degree.
Then $\Psi_p(\Vol_{G/B})= a_1^{\frac{n(n-1)}{2}}\Vol_{G/B}$. Indeed, by \Cref{prop:psi-formula},
we have
\begin{equation}
\label{eq:psi-vandermonde}
\Psi_p(\Vol_{G/B}) = \det
\begin{pmatrix}
    1 & \cdots & 1 \\
    \psi_p(t)|_{t = \lambda_1} & \cdots & \psi_p(t)|_{t = \lambda_n} \\
    \vdots & \ddots & \vdots \\
    \psi_p(t^{n - 1})|_{t = \lambda_1} & \cdots & \psi_p(t^{n - 1})|_{t = \lambda_n}\\
\end{pmatrix}
\end{equation}
Since for each $k$, the leading coefficient of $\psi_p(t^k)$ is $a_1^{k}$,
it follows that we may obtain the matrix 
from \Cref{eq:psi-vandermonde} by multiplying the Vandermonde matrix defining $\Vol_{G/B}$ on the left with a lower triangular matrix, whose diagonal terms are $(1, a_1, \dots, a_1^{n-1})$, and so we have that
\[\Psi_p(\Vol_{G/B}) = a_1^{\frac{n(n-1)}{2}} \Vol_{G/B}.\]
We conclude that $\varphi_p$ descends to an isomorphism
$\Q K_{\Snap_{G/B}} \cong \Q A_{\Vol_{G/B}}$ for any polynomial $p = 1 + \sum_{i} a_i t^i$ such that $a_i \neq 0$. 
\end{example}

A particular case of interest is when $p(t) = 1 + t$. In this case,
we call $\varphi_{1 + t}$ the \emph{exceptional Chern character}, which is given by the truncation
\[
    \varphi_{1 + t}(T^{e_i})=1+\partial_{e_i}.
\]
of the usual Chern character on the chosen bases.  Tightly related to the truncated Chern character is when $p(t) = \frac{1}{1 - t} = 1 + t + t^2 + \dots$.
Then $\varphi_p$ is precisely the map we would obtain if we constructed the exceptional Chern character with respect to the basis $\{-e_i\}_i$; we call it \emph{dual exceptional Chern character}.  It follows that its Todd class is equal to the exceptional Todd class with respect to that basis.

\begin{proposition}\label{prop:exceptional and dual exceptional}
    We have the following expressions for $\psi_p$ for the exceptional Chern character and its dual:
    \begin{align*}
     \psi_{1 + t}(t^k)  & = t(t - 1)\cdots (t - k + 1) =: t^{k\downarrow}\\
        \psi_{\frac{1}{1-t}}(t^k) & = t(t + 1)\cdots(t + k - 1) =: t^{k\uparrow}.
    \end{align*}
\end{proposition}
\begin{proof}
    For the exceptional Chern character we apply the definition of $\psi_p$ directly to get
    \[
    \psi_{1 + t}(t^k) = k!\binom{t}{k} = t(t - 1)\cdots (t - k + 1).
    \]
For the dual exceptional, this follows from the expression
\begin{align*}
    \Psi_{\frac{1}{1-t}}(h)(\vec u) &= \left(\prod_{i = 1}^n (1 - \partial_i)^{-u_i}\cdot h\right)(0)\\
    &= \left(\prod_{i = 1}^n \sum_{k = 1}^\infty \frac{u_i(u_i + 1)\cdots(u_i + k - 1)}{k!} \partial_i^k \cdot h\right)(0).
    \qedhere
\end{align*}
\end{proof}

We call $t^{k\downarrow}$ the \emph{falling factorial}
and $\Poch^{\downarrow} := \Psi_{1 + t}$ the \emph{falling Pochhammer transform}. Similarly, we call
$t^{k\uparrow}$ the \emph{rising factorial} and $\Poch^{\uparrow} := \Psi_{\frac{1}{1 - t}}$
the \emph{rising Pochhammer transform}.

\begin{example}
\label{ex:perumutahedron-no-exceptional}
    Let $P$ be the 3-dimensional permutahedron, and let $\Vol_P$ and $\Snap_P$ be its associated volume and Snapper polynomials, which we defined in \Cref{example:toric}.
    One may verify by direct computation that $\Snap_P \not\sim \Poch^\downarrow(\Vol_P)$.
    In particular, there exists no isomorphism $\zeta: K(X_P) \to A(X_P)$, sending $[\cO(D_\rho)]$ to $1 + [D_\rho]$.
    However, $X_P$ is the wonderful compactification of the Boolean arrangement, and as noted at the beginning of this section, there is a well-defined exceptional isomorphism, which sends $[\cL_F^{-1}]$ to $1 + c_1(\cL_F^{-1})$. For the permutahedral variety, this was first shown in \cite[Theorem D]{BEST2023}.
\end{example}

The exceptional isomorphism from \Cref{ex:perumutahedron-no-exceptional} was generalized to matroids in \cite{LLPP24} which allows us to give a new presentation for Snapper polynomial of matroids.
Let $M$ be a matroid. We denote by $K(M)$ and $A(M)$ the $K$- and Chow rings of $M$. These are defined as the $K$- and Chow ring of the toric variety given by the Bergman fan of $M$. When $M$ is realizable by a hyperplane arrangement $\cA$, these are also isomorphic to the $K$- and Chow ring of the wonderful compactification of $\cA$.
These are also Frobenius rings with a natural degree map $\deg_M: A(M) \to \Z$ and Euler characteristic analog $\chi_M: K(M) \to \Z$. \cite{FY04, CombinatorialHodge, LLPP24}

Suppose for simplicity that $M$ is realizable. Let $\Lambda$ be a lattice with basis $u_F$ for all flats $F$ of $M$. Then the maps $T^{u_F} \mapsto [\cL_F]$ and
$\partial_{u_F} \mapsto c_1(\cL_F)$ give surjections $\Sh(\Lambda) \to K(M)$ and $\Diff(\Lambda) \to A(M)$.
Setting
\begin{align*}
    \Snap_M((t_F)_{F \in \cL(M)}) &:= \chi_M\left(\prod_{F \in \cL(M)} [\cL_F]^{t_F}\right),\\
    \Vol_M((t_F)_{F \in \cL(M)}) &:= \frac{1}{d!}\deg_M\left(\bigg(\sum_{F \in \cL(M)} t_F c_1(\cL_F)\bigg)^d\right),
\end{align*}
we obtain the representations $K_{\Snap_M}$ and $A_{\Vol_M}$ of $K(M)$ and $A(M)$ as Frobenius rings.

By \cite[Theorem 1.2]{LLPP24}, the map $\varphi_{\frac{1}{1 - t}}$ sending $[T^{u_F}] \mapsto \frac{1}{1 -[\partial_{u_F}]}$ gives an isomorphism $K(M) \cong A(M)$, and furthermore
\[\chi_M(\xi) = \deg_M\left(\frac{1}{1 - \partial_{u_E}} \cdot \varphi_{\frac{1}{1 - t}}(\xi)\right).\]
In other words, $\td_{\frac{1}{1 - t}} = \frac{1}{1 - \partial_{u_E}}$ and so by \Cref{thm:exceptional-general}, we have 
\[
    \Snap_M=\Poch^\uparrow\left(\frac{1}{1 - \partial_{u_E}}\cdot \Vol_M\right).
\]
Note that after plugging in the expression for $\Vol_M$ provided by \cite[Theorem 7.1]{LLPP24}, this yields directly the formula in \cite[Corollary 7.5]{LLPP24}.

We extend on this by noting that $\td_{\frac 1{1-t}} = \varphi_{\frac 1{1-t}}(T^{u_E})$. The following corollary is then a direct consequence of \Cref{lem:Psi-shift}.

\begin{corollary}\label{cor:wonderful-compactification-exceptional}
   Let $M$ be realizable matroid. Then we have the following relation between the Snapper and volume polynomials in the basis $\{u_F\}$:
    \[
        \Snap_M = T^{u_E} \cdot \Poch^\uparrow(\Vol_M),
    \]
    where $E$ is the ground set of $M$.
\end{corollary}

\begin{remark}
    This generalizes to all matroids: there is a suitable sets of generators for $K(M)$ and $A(M)$ parametrized by the flats of $M$, for which the polynomials $\Snap_M$ and $\Poch^\uparrow(\Vol_M)$ are translates of each other.
\end{remark}

\section{General Riemann-Roch pairs}
\label{sec:hrr-pairs}

In the previous section, we defined Riemann-Roch pairs in the standard graded setting as pairs $(K_f, A_g)$, connected via the Chern character. The definition of the Chern character in that case was simple, because every positive element splits as a sum of line elements, on which the Chern character acts by $T^u \mapsto \exp(\partial_u)$.

In this section, we will generalize the notion of Riemann-Roch pairs to the case when $K_f$ may not be generated by invertible elements and $A_g$ by degree 1 elements. To do so, we start with a set of positive elements with specified $\lambda$-classes, which generate $K_f$ as a ring. We may then define the Chern character by specifying the Chern class of each generator.

Given such a Riemann-Roch pair, we construct an embedding into a standard graded Riemann-Roch pair that preserves positivity, which makes the splitting principle available and allows all the structures of the previous section -- Chern classes, $\lambda$-ring structure, Serre duality, and the associated graded ring -- to transfer to this general setting.

This generality is essential in practice: the $K$-ring of an algebraic variety is rarely generated by line elements. A fundamental example is the Grassmannian $\Gr(2,4)$, whose tautological bundle $\cS$ does not split into a sum of line elements in the $K$-ring.

\subsection{The Chern character}

Let $S = \{E_1, \dots, E_n\}$ be a set with a rank function $\rk: S \to \N_{>0}$.
Let 
\[\Sigma = \N\{\lambda^k(E_i) \mid 1 \leq i \leq n, 1 \leq k \leq \rk(E_i)\},\]
and 
\[\Gamma = \Z\{c_k(E_i) \mid 1 \leq i \leq n, 1 \leq k \leq \rk(E_i)\}.\]
For now we see the $\lambda^k(E_i)$ and $c_k(E_i)$ as independent variables. $\Gamma$ is a graded lattice with the $c_k(E_i)$ having degree $k$. We will identify $\Sigma$ to a multiplicative subgroup of $\Sh(\Sigma)$ and $\Gamma$ to an additive subgroup of $\Diff(\Gamma)$.

We will now define the Chern character $\ch: \Sh_\Q(\Sigma) \to \Diff_\Q\llbracket \Gamma \rrbracket$. To do so, we will embed $\Sh(\Sigma)$ into a larger ring, which is generated by line elements.
Let
\[\Lambda = \Z\{u_{i, j} \mid 1 \leq i \leq n, 1 \leq k \leq \rk(E_i)\},\]
then we have an embedding
\begin{align*}
    \iota: \Sh(\Sigma) &\to \Sh(\Lambda)\\
    \lambda^k(E_i) &\mapsto \sigma_k(T^{u_{i, 1}}, \dots, T^{u_{i, \rk(E_i)}}).
\end{align*}
Similarly, we have an embedding
\begin{align*}
    \nu: \Diff(\Gamma)&\to \Diff(\Lambda)\\
    c_k(E_i) &\mapsto \sigma_k(\partial_{u_{i, 1}}, \dots, \partial_{u_{i, \rk(E_i)}}).
\end{align*}
Let $\ch: \Sh_\Q(\Lambda) \to \Diff_\Q \llbracket \Lambda\rrbracket$
be the Chern character as defined in \Cref{subsec:chern-character}.
An element in the image of the composition 
\[\Sh_\Q(\Sigma) \xhookrightarrow{\iota} \Sh_\Q(\Lambda) \xhookrightarrow{\ch} \Diff_\Q\llbracket \Lambda\rrbracket\]
is by definition a symmetric function in the $\partial_{u_{i, j}}$ for any fixed $i$. In particular,
it is in the image of $\nu: \Diff_\Q\llbracket \Gamma \rrbracket \hookrightarrow \Diff_\Q\llbracket \Lambda \rrbracket$. It follows that the Chern character restricts to a well-defined map
$\Sh_\Q(\Sigma) \hookrightarrow \Diff_\Q\llbracket\Gamma\rrbracket$.

Let $g: \Gamma \to \Q$ be a weighted homogeneous polynomial. The following proposition shows that $\ch$ induces a surjective map to $\Q A_g$.

\begin{proposition}
    The map $\ch: \Sh_\Q(\Sigma) \to \Q A_g$ is surjective.
\end{proposition}

\begin{proof}
    Let $d = \deg(g)$, and let $I = \Diff_\Q\llbracket \Gamma \rrbracket_{> d}$ and $J = \Diff_\Q \llbracket \Lambda \rrbracket_{> d}$ be the ideals of differential operators which only have terms of degrees larger than $d$.
    Let $\alpha \in \Diff_\Q \llbracket \Gamma \rrbracket / I$, then $\nu(\alpha)$ is a symmetric polynomial in the $\partial_{u_{i, j}}$ for each fixed $i$. As in \Cref{lem:ch-surjective}, we have that
    \[
    \partial_{u_{i , j}} = \ch\left(\sum_{i = 1}^d (-1)^{k + 1}\frac{(T^{u_{i, j}} - 1)^k}{k}\right)
    \mod J
    \]
    In particular, there is an element $\eta \in \Sh_{\Q}(\Lambda)$, which may be written as a symmetric polynomial in the $T^{u_{i, j}}$ for each fixed $i$, such that $\ch(\eta) = \nu(\alpha) \mod J$.
    Since all symmetric polynomials in the $T^{u_{i, j}}$
    are in the image of $\iota$, there is some $\xi \in \Sh_\Q(\Sigma)$ 
    such that $\iota(\xi) = \eta$. Since $\nu$ induces to an injective map
    $\Diff_\Q\llbracket \Gamma \rrbracket/I \to\Diff_\Q\llbracket \Lambda \rrbracket/J$,
    we conclude that $\ch(\xi) = \alpha \mod I$, and so this shows that the map $\ch: \Sh_\Q(\Sigma) \to \Diff_\Q \llbracket \Gamma \rrbracket / I$ is surjective. Since we have $\Ann_{\Diff_\Q(\Gamma)}(g) \subseteq I$, this finishes the proof.
\end{proof}

Let $f: \Sigma \to \Q$ be a function, then by 
\Cref{thm:map-descends-iff-hrr}, $\ch$ induces an isomorphism $\Q K_f \cong \Q A_g$ if and only if there exists
some $\td \in \Q A_g^\times$ such that 
\[f(\xi) = \deg_g(\td \cdot \ch(\xi))\qquad\textrm{for all }\xi \in \Sigma.\]

\begin{definition}
    Suppose the Chern character gives an isomorphism $\Q K_f \cong \Q A_g$,
    then we call the data defining $(K_f, A_g)$ a \emph{Riemann-Roch pair}.
\end{definition}

Note that this definition agrees with \Cref{def:RR-pair-std-graded}
in the standard graded case.

\begin{example}\label{example:grassmannian-2-4}
     The $K$-ring and Chow ring of the Grassmannian $\Gr(2, 4)$ form a Riemann-Roch pair. We will now compute its presentation as Frobenius quotients.
    
    Let $\cS$ be the tautological bundle of $\Gr(2, 4)$.
    The $K$-ring of $\Gr(2, 4)$ is generated by the classes $[\lambda^1(\cS)] = [\cS]$ and $[\lambda^2(\cS)] = [\det(\cS)]$, and
    its Chow ring is generated by the Chern classes $c_1(\cS)$ and $c_2(\cS)$.

    The Todd class of the Grassmannian is given by $\Td(\cS^\vee \otimes \cQ)$, where $\cQ$ is the quotient bundle. Using the short exact sequence
    \[
        0 \to \cS \to \cO_{\Gr(2, 4)} \to \cQ \to 0,
    \]
    we can relate the Chern classes of $\cS$ and $\cQ$, and a calculation shows that
    \[
        \Td(\cS^\vee \otimes \cQ) = 
        1 - 2 c_1(\cS) + \frac{23c_1(\cS)^2}{12}  - \frac{7c_1(\cS)^3}{6}
        -\frac{c_1(\cS)^2 c_2(\cS) - c_2(\cS)^2}{180} + \frac{c_1(\cS)^4}{2}.
    \]

    From the same short exact sequence follow the relations 
    \[
        2c_1(\cS)c_2(\cS) = c_1(\cS)^3\textrm{ and } c_2(\cS)^2 = c_1(\cS)^2c_2(\cS)
    \]
    so we obtain
    \[
        \Td(\cS^\vee \otimes \cQ) = 
        1 - 2 c_1(\cS) + \frac{23}{12}c_1(\cS)^2  - \frac{7}{6}c_1(\cS)^3
        + \frac{1}{2}c_1(\cS)^4.
    \]
    Since $c_1(\cS) = -\sigma_{(1)}$, where $\sigma_{(1)}$ is the Schubert class corresponding to the partition $(1)$, and $\sigma_{(1)}^4 = 2\sigma_{(2, 2)}$, we get that 
    \[
        \int_{\Gr(2, 4)} c_1(\cS)^4 = \int_{\Gr(2, 4)} 2\sigma_{(2, 2)} = 2.
    \]
    Let $\Gamma = c_1(\cS)\Z \oplus c_2(\cS)\Z$, then we get that the degree polynomial on $\Gamma$ is given by
    \begin{align*}
        g(x_1, x_2) &= \frac{\int c_1(S)^4}{4!}x_1^4 + \frac{\int c_1(\cS)^2c_2(\cS)}{2!}x_1^2x_2 + \frac{\int c_2(\cS)^2}{2!}x_2^2\\
        &= \frac{1}{12}x_1^4 + \frac{1}{2}x_1^2 x_2 + \frac{1}{2}x_2^2,
    \end{align*}
    in other words, $A(\Gr(2, 4)) \cong A_g$ as Poincar\'e duality rings.

    Write $K(\Gr(2, 4)) \cong K_f$ for some function $f: \Sigma \to \Q$, where $\Sigma = [\cS] \N \oplus [\det(\cS)] \Z$.
    By \Cref{rem:hrr-equivalent-statement}, $f$ is uniquely determined from $g$ and $\td = \Td(\cS^\vee \otimes \cQ)$ via the equation
    \[f(u) := (\td \cdot \ch(T^u) \cdot g) (0).\]

    Note that $[\cS]$ has rank 2 and does not split into line elements in $K_f$,
    so this example lies outside the standard graded setting of \Cref{sec:standard}.
\end{example}

We may also formulate Grothendieck-Riemann-Roch in this setting.

\begin{definition}
    A \emph{morphism of Riemann-Roch pairs} $\phi: (K_{f_1}, A_{g_1}) \to (K_{f_2}, A_{g_2})$
    is a pair of maps $\phi_K: K_{f_1} \to K_{f_2}$
    and $\phi_A: A_{g_1} \to A_{g_2}$,
    such that the diagram
    \[
        \begin{tikzcd}
            \Q K_{f_1} \arrow[d, "\ch"] \arrow[r, "\phi_K"] & \Q K_{f_2} \arrow[d, "\ch"] \\
            \Q A_{g_1} \arrow[r, "\phi_A"]                  & \Q A_{g_2}                 
        \end{tikzcd}
    \]
    commutes.
\end{definition}

We have that $\phi_K$ induces a pullback $\phi_K^*: \Hom_\Q(\Q K_{f_2}, \Q) \to \Hom_\Q(\Q K_{f_1}, \Q)$. Using the Frobenius pairing, we identify $\Q K_{f_i} \cong \Hom_\Q(\Q K_{g_2}, \Q)$, so we obtain an induced map $\phi_K^*: \Q K_{f_2} \to \Q_{K_{f_1}}$. Similarly, we have a pullback $\phi_A^*: \Q A_{g_2} \to \Q_{A_{g_1}}$.
The same proof as in \Cref{thm:grr} shows the following theorem.

\begin{theorem}
    Let $\phi: (K_{f_1}, A_{g_1}) \to (K_{f_2}, A_{g_2})$ be a morphism of Riemann-Roch pairs. Then the following diagram commutes:
\[
\begin{tikzcd}
\Q K_{f_2} \arrow[d, "\td_2\cdot \ch(-)"'] \arrow[r, "\phi_K^*"] & \Q K_{f_1} \arrow[d, "\td_1\cdot\ch(-)"] \\
\Q A_{g_2} \arrow[r, "\phi_A^*"]                          & \Q A_{g_1}                         
\end{tikzcd}
\]
\end{theorem}

\subsection{The splitting principle}

Given a Riemann-Roch pair as above, we will show that we can embed it in a standard graded Riemann-Roch pair, enabling the use of the splitting principle.
This will allow us to carry out calculations on $K_f$ assuming positive elements split into sums of line elements, provided that the result does not depend on said elements.
By embedding $(K_f, A_g)$ in a standard graded pair, we may also pull back the structures which we explored in the previous section.

To construct the embedding, we will find a polynomial $h$, such that the diagram
\[
\begin{tikzcd}
	\Sh_\Q(\Sigma) & \Diff_\Q \llbracket \Gamma \rrbracket\\
	\Sh_\Q(\Lambda) & \Diff_\Q\llbracket \Lambda \rrbracket
	\arrow["\ch", from=1-1, to=1-2]
	\arrow["\iota", hook, from=1-1, to=2-1]
	\arrow["\nu", hook, from=1-2, to=2-2]
	\arrow["\ch", from=2-1, to=2-2]
\end{tikzcd}
\]
descends to a commutative diagram
\[\begin{tikzcd}
	\Q K_f & \Q A_g \\
	\Q K_{h} & \Q A_{h}
	\arrow["\ch", from=1-1, to=1-2]
	\arrow["\iota", hook, from=1-1, to=2-1]
	\arrow["\nu", hook, from=1-2, to=2-2]
	\arrow["\ch", from=2-1, to=2-2]
\end{tikzcd}\]

We will construct this polynomial $h$ by induction. Recall that we defined $\Sigma$ and $\Gamma$ using a finite set $S$ with a rank function $\rk: S \to \N_{> 0}$. We call $e(S) = \sum_{E \in S} (\rk(E) - 1)$ the \emph{excess rank} of $S$.

Suppose $e(S) > 0$, and let $E \in S$ of rank $r > 1$. Then $\nu(E) = T^{u_1} + \dots + T^{u_r}$ for some $u_i \in \Lambda$.
Let $S' = S\setminus\{E\}\cup \{Q, L\}$, where we set $\rk(Q) = r - 1$ and $\rk(L) = 1$.
Let $\Gamma'$ be the graded lattice given by $S'$ as before.
Define $\nu' : \Diff(\Gamma') \hookrightarrow \Diff(\Lambda)$ by 
\begin{align*}
    \nu'(c_k(Q)) &= \sigma_k(\partial_{u_1}, \dots, \partial_{u_{r - 1}}),\\
    \nu'(c_1(L)) &= \partial_{u_r}.
\end{align*}
Then $\im(\nu') \supseteq \im(\nu)$, and so in particular $\nu$ factors as $\nu' \circ \hat \nu$ for some
$\hat \nu : \Diff(\Gamma) \hookrightarrow \Diff(\Gamma')$. It can be easily seen that $\hat \nu$
sends $c_k(E)$ to $c_k(Q) + c_{k-1}(Q)c_1(L)$ for all $1 \leq k \leq r-1$,
and $c_r(E)$ to $c_{r-1}(Q) c_1(L)$.

\begin{proposition}\label{prop:proj-bundle-construction}
    There exists a homogeneous polynomial $g'$ on $\Gamma'$, such that $\hat \nu$ induces an embedding on the quotients $A_g \hookrightarrow A_{g'}$.
\end{proposition}

\begin{proof}
    Let \[A := A_g[\eta]/(c_r(E) - c_{r-1}(E)\eta + \dots + (-1)^r\eta^r).\]
    The ring $A$ is a free $A_g$-module with basis $1, \dots, \eta^{r-1}$. $A$ is also a graded ring, where $\eta$ has degree~$1$. We define a linear map $\deg_A: A \to \Q$ by
    \[
        \sum_{i = 0}^{r - 1}\xi_i \eta^i \mapsto \deg_g(\xi_{r-1}).
    \]
    This linear map equips $A$ with the structure of a Poincar\'e duality ring.
    Indeed, suppose $\alpha = \sum_{i = 0}^{r-1} \xi_i \eta^i \neq 0$, then there is some $i$ such that $\xi_i \neq 0$. Since $A_g$ is a Poincar\'e duality ring, there is some $\xi' \in A_g$ such that $\deg_g(\xi_i \cdot \xi') \neq 0$. In particular,
    $\deg_A(\alpha \cdot \eta^{r - i - 1} \xi') = \deg_g(\xi_i \cdot \xi') \neq 0$, and hence the bilinear form $\deg_A(-\cdot -)$ is non-degenerate.
    
    We have a map $\psi: \Diff(\Gamma') \to A$, given by
    \begin{align*}    
        c_k(Q) &\mapsto \sum_{i = 0}^k (-1)^i c_{k - i}(E) \eta^i,\\
        c_1(L) &\mapsto \eta.
    \end{align*}
    This map is surjective, because we have
    \[
        c_k(E) = \psi(c_k(Q) + c_{k-1}(Q)c_1(L))
    \]
    for $k = 1, \dots, r-1$, and because in $A$ we have the relation 
    \[
        c_r(E) = c_{r-1}(E)\eta + \dots + (-1)^{r+1}\eta^r = \psi(c_{r-1}(Q)c_1(L)).
    \]
    
    The kernel of $\psi$ is the annihilator of a homogeneous polynomial $g'$ on $\Gamma'$, hence $A \cong A_{g'}$.
    It can be easily verified that the diagram
    \[
    \begin{tikzcd}
    \Diff(\Gamma) \arrow[d, two heads] \arrow[r, "\hat \nu", hook] & \Diff(\Gamma') \arrow[d, two heads] \\
    A_g \arrow[r, hook]                                         & A_{g'}                                 
    \end{tikzcd}
    \]
    commutes, where the bottom map is the inclusion $A_g \hookrightarrow A = A_{g'}$.
\end{proof}

\begin{remark}
    If $A_g$ is the Chow ring of some smooth complete variety, then $E$ corresponds to a vector bundle of rank $r$, and $A_{g'}$ is the Chow ring of the projective bundle of $E$. Here $L$ represents the tautological bundle on $\mathbb{P}(E)$ and $Q$ is the quotient bundle.
\end{remark}

\begin{corollary}\label{cor:h-exists-A_h}
    There exists some homogeneous polynomial $h$ on $\Lambda$, such that the inclusion $\nu:~\Diff(\Gamma) \to \Diff(\Lambda)$ induces an inclusion $A_g \hookrightarrow A_h$.
\end{corollary}

\begin{proof}
    We proceed by induction on the excess rank $e(S)$. If $e(S) = 0$, then $\Gamma = \Lambda$ and we can just pick $g = h$. Otherwise we construct $S'$ as above and let $g'$ be given by \Cref{prop:proj-bundle-construction}, then $e(S') = e(S) - 1 < e(S)$. By induction hypothesis there exists some $h$ such that the inclusion 
     $\nu': \Diff(\Gamma_{S'}) \to \Diff(\Lambda)$ induces an inclusion $A_{g'} \hookrightarrow A_h$ on the quotient.
     Since also $\hat \nu$ induces an inclusion $A_{g} \hookrightarrow A_{g'}$, and $\nu = \nu' \circ \hat\nu$, this finishes the proof.
\end{proof}

We now have constructed a polynomial $h$, such that we get an inclusion of the continuous Frobenius quotients. To conclude, it suffices to note that the same polynomial can be used for the discrete Frobenius quotient to get the desired splitting.

\begin{theorem}
    There exists a homogeneous polynomial $h$ on $\Lambda$, such that the diagram
    \[
    \begin{tikzcd}
    	\Sh_\Q(\Sigma) & \Diff_\Q \llbracket \Gamma \rrbracket\\
    	\Sh_\Q(\Lambda) & \Diff_\Q\llbracket \Lambda \rrbracket
    	\arrow["\ch", from=1-1, to=1-2]
    	\arrow["\iota", hook, from=1-1, to=2-1]
    	\arrow["\nu", hook, from=1-2, to=2-2]
    	\arrow["\ch", from=2-1, to=2-2]
    \end{tikzcd}
    \]
    descends to a commutative diagram
    \[\begin{tikzcd}
    	\Q K_f & \Q A_g \\
    	\Q K_{h} & \Q A_{h}
    	\arrow["\ch", from=1-1, to=1-2]
    	\arrow["\iota", hook, from=1-1, to=2-1]
    	\arrow["\nu", hook, from=1-2, to=2-2]
    	\arrow["\ch", from=2-1, to=2-2]
    \end{tikzcd}\]
    on the quotients.
\end{theorem}

\begin{proof}
    Let $h$ be as in \Cref{cor:h-exists-A_h}. Then since the Chern characters are isomorphisms, we have an inclusion $\Q K_f \to \Q K_h$, defined by $\ch^{-1} \circ \nu \circ \ch$. The fact that this inclusion agrees with $\iota: \Sh_\Q(\Sigma) \to \Sh_\Q(\Lambda)$ is equivalent to saying that in the following diagram
    \[
    \begin{tikzcd}
    \Q K_f \arrow[rrr, "\ch", bend left] \arrow[d, "\ch^{-1}\circ \nu \circ \ch", hook] & \Sh_\Q(\Sigma) \arrow[l, two heads] \arrow[r, "\ch"] \arrow[d, "\iota", hook] & \Diff_\Q\llbracket \Gamma \rrbracket \arrow[r, two heads] \arrow[d, "\nu", hook] & \Q A_g \arrow[d, "\nu", hook]\\
    \Q K_h \arrow[rrr, "\ch", hook', bend right]     & \Sh_\Q(\Lambda) \arrow[l, two heads] \arrow[r, "\ch"]                         & \Diff_\Q\llbracket \Lambda \rrbracket \arrow[r, two heads]                       & \Q A_h
    \end{tikzcd}
    \]
    the left square commutes. This follows from the fact that all the other squares commute.
\end{proof}

In the rest of this section, we will see $K_f$ and $A_g$ as a subset of $K_h$ and $A_h$, so that we may assume for computations that any positive element of $K_f$ splits into a sum of line elements.

\subsection{Multiplicative maps}

We will use the splitting $(K_h, A_h)$ of $(K_f, A_g)$ to define multiplicative maps, like we did in \Cref{subsec:multiplicative-maps}.

\begin{proposition}
    \label{prop:mult-map-A_g-general}
    Let $R$ be a ring of characteristic $0$ and let
    $a \in R \llbracket x \rrbracket$ be a formal series with $a(0)$ invertible in $R$.
    There exists a unique multiplicative map $\psi: K_f \to R \otimes A_g$, such that for any 
    $\xi = \left[\sum_{i = 1}^r T^{u_i}\right] \in K_f^+$, we have
    $\psi(\xi) = \left[\prod_{i = 1}^r a(\partial_{u_i})\right]$.
\end{proposition}

\begin{proof}
    Let $\psi: K_h \to R\otimes A_h$ be the unique multiplicative map determined by
    $\psi([T^u]) = a([\partial_u])$ for all $u \in \Lambda$. Then clearly $\psi(\xi) = \left[\prod_{i = 1}^r a(\partial_{u_i})\right]$ for all $\xi = \left[ \sum_{i = 1}^r T^{u_i} \right] \in K_f^+$ by multiplicativity of $\psi$. It suffices to show that $\psi$ restricts to a map $K_f \to R \otimes A_g$.
    
    For any $E_i \in S$, we have that $\lambda^k(E_i) = \sigma_k(T^{u_{i, 1}}, \dots, T^{u_{i, \rk(E_i)}})$. It follows by multiplicativity of $\psi$ that
    \[
    \psi([\lambda^k(E_i)]) = \prod_{1 \leq i_1 < \dots < i_k \leq \rk(E_i)}a([\partial_{u_{i_1}} + \dots + \partial_{u_{i_k}}])
    \]
    is symmetric in the $[\partial_{u_{i}, j}]$ for any fixed $i$, and so in particular $\psi([\lambda^k(E_i)]) \in A_g$. Since the elements of this form generate $K_f^+$, we have that
    $\psi(K_f^+) \subseteq A_g$.

    Note that for any $\xi \in K_f^+$, $\psi(\xi)$ has constant coefficient $a(0)^{\rk(\xi)}$, which is invertible in $R$, and so in particular $\psi(\xi) \in (R \otimes A_g)^\times$. We conclude that $\psi(-\xi) \in R \otimes A_g$. Since $K_f$ is the group of differences associated to $K_f^+$, we get by multiplicativity of $\psi$
    that $\psi(K_f) \subseteq R \otimes A_g$, which finishes the proof.
\end{proof}

For $a(x) = 1 + x$ we recover the total Chern class
$c: K_f \to A_g$, and for $a(x) = \frac{x}{1 - \exp(-x)}$, we get the Todd operator $\Td: K_f \to \Q A_g$.
It is easy to see that $c([E_i]) = 1 + \sum_{k = 1}^{\rk(E_i)}[c_k(E_i)]$, which justifies the choice of notation for the generators of $\Gamma$.

\begin{proposition}
    \label{prop:mult-map-Kf-general}
    Let $p \in R[x]$ be a polynomial with $p(1)$ invertible in $R$. There exists a unique multiplicative map
    $\mu: K_f \to R \otimes K_f$, such that for any $\xi = \left[\sum_{i = 1}^r T^{u_i}\right] \in K_f^+$, we have
    $\mu(\xi) = \left[ \prod_{i = 1}^{r}p(T^{u_{i}}) \right]$.
\end{proposition}

\begin{proof}
    Let $a = p \circ \exp$. Then $a(0) = p(1)$ is invertible in $R$, and so let $\psi: K_f \to R \otimes A_g$ the map given by \Cref{prop:mult-map-A_g-general}. Let $\mu := \ch^{-1} \circ \psi$, then by construction we have that $\mu(\xi) = \left[ \prod_{i = 1}^{r}p(T^{u_{i}}) \right]$ for all $\xi = \left[\sum_{i = 1}^r T^{u_i}\right] \in K_f^+$. To conclude, it suffices to show that $\im(\mu) \subseteq R \otimes K_f$.
    
    For any $\xi \in K_{f}^+$, we can write $\xi = \sum_{i = 1}^r [T^{u_i}]$, for some $u_i \in \Lambda$, where $r = \rk(\xi)$. In particular, 
    \[
        \mu(\xi) = \prod_{i = 1}^r p([T^{u_i}]) = \prod_{i = 1}^r p(1 - [\Delta_{u_i}]) = p(1)^r + \eta,
    \]
    where $\eta \in ([\Delta_{u_1}], \dots, [\Delta_{u_r}]) \cdot R \otimes K_{h}$, and hence $\eta$ is nilpotent.
    We also have $\eta = \mu(\xi) - p(1)^r$, so $\eta \in R \otimes K_{f}$. We conclude that $\mu(\xi)$ is invertible in $R \otimes K_f$, and hence $\mu(-\xi) \in R \otimes K_{f}$. Since $K_f$ is the group of differences associated to $K_f^+$,
    we get by multiplicativity of $\mu$ that $\mu(K_f) \subseteq R \otimes K_f$, which finishes the proof.
\end{proof}

Setting $p(x) = 1 + tx$ recovers the $\lambda$-operator $\lambda: K_f \to K_f\llbracket t \rrbracket$,
and $p(x) = x$ yields the determinant $\det: K_f \to K_f^1$. Again, it is easy to see that $\lambda([E_i]) = 1 + \sum_{k = 1}^{\rk(E_i)} \lambda^k(E_i)$, which justifies the choice of notation for the generators of $\Sigma$.

\begin{corollary}
    $K_f$ is a $\lambda$-ring.
\end{corollary}

\subsection{Serre duality}

The splitting principle will also allow us to show that Serre duality and the notion of dualizing and cotangent classes naturally transfer to this general setting.
In order to be able to talk about Serre duality, we first need to verify that the dualizing morphism is also well-defined in this situation. Recall that there is an involution on $A_g$, which is given by multiplying the $k$-th graded component by $(-1)^k$ for each $k$.

\begin{proposition}
    There is an involution $(-)^\vee: K_f \to K_f$, such that $\ch(\xi^\vee) = \ch(\xi)^\vee$.
\end{proposition}

\begin{proof}
    Let $\xi \in K_f^+$ be a positive element. Then $\iota(\xi) = \sum_{i = 0}^r \ell_i$ for some $\ell_i \in K_h^1$.
    Then
    \[
        \xi^\vee = \sum_{i = 1}^r \ell_i^{-1}
        = \frac{\sigma_{r - 1}(\ell_1, \dots, \ell_r)}{\ell_1 \cdots \ell_r}
        = \frac{\lambda^{r - 1}(\xi)}{\lambda^r(\xi)},
    \]
    which is well-defined in $K_f$ since $\lambda^r(\xi)$ is an element of rank $1$ and hence a unit in $K_f$.
    Since $\ch(\xi^\vee) = \ch(\xi)^\vee$ for all $\xi \in K_h$, the same holds for the restricted involution.
\end{proof}

\begin{definition}
We define the dualizing class $\omega$ as the unique element of $\Q K_f$, such that the equation 
\[
    \chi_f(\xi) = (-1)^d \chi_f(\omega \cdot \xi^\vee)
\]
is satisfied for all $\xi \in \Q K_f$.
\end{definition}

Now, essentially the same proof
as in \ref{prop:cotangent-todd-dualizing} along with the splitting principle shows the
following statement.

\begin{proposition}\label{prop:canonical-equals-cotangent-determinant-general}
Suppose there is some positive $\Omega \in K_f$ is such that
the Todd class is equal to $\Td(\Omega^\vee)$. Then
$\det(\Omega) \in K_f$ is the dualizing class of $\Q K_f$.
\end{proposition}

\subsection{Associated graded ring}

We will now construct the associated graded ring of $\Q K_f$ and show that it is naturally isomorphic to $\Q A_g$.

Define
\[J_i :=  \{ \xi \in K_f \mid \ch(\xi)_k = 0 \textrm{ for all $0 \leq k < i$} \}.\]
Let $(K_h, A_h)$ be a splitting of $(K_f, A_g)$ as before, and let 
\[K_h = I^{(0)} \supset I^{(1)} \supset I^{(2)} \supset \dots\]
be the filtration on $K_h$ as defined in \Cref{sec:associated-graded}.
Then it is not hard to see that $J_i = I^{(i)} \cap K_f$. In particular, 
\[K_f = J_0 \supset J_1 \supset J_2 \supset \dots\]
defines a multiplicative filtration, 
and so we may consider the associated graded ring
\[
    \gr K_f := \bigoplus_{i = 0}^\infty J_{i}/J_{i + 1}.
\]
Since $J_i = I^{(i)} \cap K_f$ for all $i$, it follows that the inclusion $\iota: K_f \hookrightarrow K_{h}$ induces an inclusion $J_i/J_{i + 1} \hookrightarrow I^{(i)}/I^{(i+1)}$ and hence also 
$\gr K_f \hookrightarrow \gr K_{h}$, which we denote by $\bar\iota$.

\begin{proposition}
    The isomorphism of graded rings $\Phi': \Q A_{h} \xrightarrow{\sim} \Q \gr K_{h}$
    restricts to an isomorphism $\Phi: \Q A_g \xrightarrow{\sim} \Q \gr K_f$.
\end{proposition}

\begin{proof}
    We claim that the image of the composite
    \[
    \begin{tikzcd}
        \Q A_g \arrow[r, "\nu", hook] & \Q A_{h} \arrow[r, "\Phi'"] & \Q \gr K_{h}
    \end{tikzcd}
    \]
    is contained in $\bar\iota(\Q \gr K_f) \subseteq \Q \gr K_{h}$. This would
    then imply we have an injective map $\Phi: \Q A_g \hookrightarrow \Q \gr K_{f}$.
    Because the two rings have the same dimension, since we have an isomorphism $\Q K_f \cong \Q A_g$, and $\Q \gr K_f$ has the same dimension as $\Q K_f$, this forces $\Phi$ to be surjective, and hence an isomorphism.

    So let $\xi \in K_f^+$, and write $\iota(\xi) = [T^{u_1}] + \dots + [T^{u_r}]$ for some $u_i \in \Lambda$,
    then the composite map from above sends
    \[
    \begin{tikzcd}
        c_k(\xi) \arrow[r, "\nu", maps to] & \left[\sigma_k(\partial_{u_1}, \dots, \partial_{u_r})\right] \arrow[r, "\Phi'", maps to] & \left[\sigma_k(D_{u_1}, \dots, D_{u_r})\right] \in I^{(k)}/I^{(k+1)}
    \end{tikzcd}.
    \]
    Note that we may write $[\sigma_k(D_{u_1}, \dots, D_{u_r})]$ as a symmetric polynomial in the $T^{u_i}$, it lies in the image of $\bar\iota$, and so the claim follows.
\end{proof}

\begin{remark}
    There is another filtration
    \[
    K_f = F_0^\gamma \supseteq F_1^\gamma \supseteq F_2^\gamma \supseteq \dots,
    \]
    called the $\gamma$-filtration, which is defined purely from the $\lambda$-ring structure of $K_f$ (see \cite[Chapter III]{FultonLang85}). These filtrations agree after tensoring with $\Q$, and so yield the same associated graded ring (over $\Q$).
    More precisely, we have the relation $J_i = \Q F_i^\gamma \cap K_f$.
\end{remark}

\section{Further examples and applications}\label{sec:examples}
\subsection{Linear Families of Polytopes}
\label{subsec:linear-families}

Let $C$ be a full-dimensional rational polyhedral cone in some finite-dimensional vector space $V$.
A linear family of polytopes is a linear map $\Delta: C \to \cP_n(\Q)$, where $\cP_n(\Q)$ is the collection of all rational polytopes in $\R^n$, equipped with Minkowski sum and scaling by elements of $\Q_{\geq 0}$. The group of differences of $\cP_n(\Q)$ is denoted $\cV_n(\Q)$ and its elements are called \emph{virtual polytopes}.
Let $\Lambda \subseteq V$ be a full-rank sublattice such that $\Delta(\Lambda \cap C) \subseteq \cP_n(\Z)$ are all lattice polytopes.
In \cite{kavvil}, the authors define the notion of the anticanonical polytope of such a linear family, to be a polytope $\Delta(\kappa) \in \cV_n(\Z)$, where $\kappa \in \Lambda$, such that for all $\gamma \in C \cap \Lambda$,
\begin{equation}\label{eq:serre-duality-lattice-counts}
    \#(\Delta(\gamma - \kappa) \cap \Z^n) = \#(\Delta(\gamma)^\circ \cap \Z^n).
\end{equation}
We may express this in our setting as follows.

The function $\gamma \in C \cap \Lambda \mapsto \#(\Delta(\gamma) \cap \Z^n)$ agrees with a polynomial on $C \cap \Lambda$. We extend this polynomial uniquely to a polynomial function on $\Lambda$, which we denote by $\Ehr_\Delta$.
Note that by Ehrhart reciprocity, we have
\[
    \Ehr_\Delta(-\gamma) = (-1)^d \#(\Delta(\gamma)^\circ \cap \Z^n),
\]
and so \Cref{eq:serre-duality-lattice-counts} may be rewritten as
\[
    \Ehr_\Delta(\gamma - \kappa) = (-1)^d\Ehr_\Delta(-\gamma).
\]
This means precisely that $\kappa$ is the dualizing class of $K_\Ehr$.

Similarly, the volume polynomial $\Vol_\Delta: \Lambda \cap C \to \R, \gamma \mapsto \Vol_d(\Delta(\gamma))$ is a polynomial on $\Lambda \cap C$ is a polynomial which may be extended to all of $\Lambda$.
Assuming that we have $\Ehr_\Delta = \td \cdot \Vol_\Delta$, we get that $c_1(\kappa) = -\frac{1}{2} \td_1$,
which determines $\kappa$ uniquely.

\subsection{McMullen's polytope algebra as a limit of Frobenius quotients}
\label{subsec:polytope-algebra-lambda-structure}

In this section we will equip McMullen's polytope algebra $\Pi(\Q^n)$ with a natural $\lambda$-ring structure.

By \cite[Theorem 5.2]{FS97}, $\Pi(\Q^n)$ is isomorphic to the direct limit
$\varinjlim \Q A(X)$ of $n$-dimensional complete toric varieties. Since $n$-dimensional smooth projective
toric varieties form a cofinal directed subset, it suffices to take the direct limit over toric varieties of this kind. Let $f: X \to Y$ be an equivariant morphism of such toric varieties, then $f$ induces inclusions
$f^*: \Q A(Y) \to \Q A(X)$ and $f^*: \Q K(Y) \to \Q K(X)$, and the following diagram commutes
\[
\begin{tikzcd}
\Q K(Y) \arrow[d, "\ch"] \arrow[r, "f^*"] & \Q K(X) \arrow[d, "\ch"] \\
\Q A(Y) \arrow[r, "f^*"]          & \Q A(X)         
\end{tikzcd}
\]
It follows that $\Pi(\Q^n) \cong \varinjlim \Q K(X)$. Each one of these rings is a $\lambda$-ring and the maps
$\lambda^i$ clearly commute with the pullbacks, inducing the structure of $\lambda$-ring on the direct limit.

\begin{remark}\label{rem:ring-conditions}
The direct limit $\varinjlim A(X)$ of the Chow rings of complete $n$-dimensional toric varieties is an example of the \emph{ring of conditions} introduced by De Concini and Procesi in  \cite{de1985complete}. More generally, they defined the ring of conditions $C(G/H)$ for any spherical homogeneous space $G/H$ and showed that it is given by the direct limit 
\[
C(G/H) = \varinjlim A(X),
\]
where the limit is taken over all $G$-equivariant compactifications of $G/H$.
In the case of $G/H= (\C^*)^n$ as the ring of conditions we exactly obtain the limit of Chow rings of complete toric varieties of dimension $n$. The ring of conditions of $(\C^*)^n$ was intensively studied; see \cite[Section 7]{kazarnovskii2021newton} for a survey. More generally, the ring of conditions of more general spherical homogeneous spaces was computed in \cite{strickland2008ring, strickland2011equivariant, esterov2017characteristic, gibson2018rings, khovanskii2021gorenstein}.
In particular, \cite{khovanskii2021gorenstein} realizes the ring of conditions of a horospherical homogeneous space as a (infinitely generated) continuous Frobenius quotient.

Motivated by the definition of the ring of conditions we would like to introduce its $K$-theoretic version $KC(G/H)=\varinjlim K(X)$ given by the direct limit of $K$-rings of $G$-equivariant compactifications of $G/H$. 
\end{remark}

This is an example of the following general construction. 
Let $I$ be a directed set and assume that for every $\alpha\in I$ we get a triple $\Lambda_\alpha, f_\alpha, g_\alpha$ where $\Lambda_\alpha$ is a finite rank lattice and $f_\alpha, g_\alpha$ are polynomials on $\Lambda_\alpha$ such that $(K_{f_n}, A_{g_n})$ is a Riemann-Roch pair. Assume further that $\Lambda_\alpha$ form a directed system such that for any $\alpha\geq\beta \in I$ the lattice homomorphism $\phi_{\alpha\beta}\colon \Lambda_\alpha \to \Lambda_\beta$ induce well-defined pushforwards:
\[
(\phi_{\alpha\beta})_*\colon K_{f_\alpha}\to K_{f_\beta},\quad (\phi_{\alpha\beta})_*
\colon A_{g_\alpha}\to A_{g_\beta}.
\]
In that case we can define the direct limits
\[
K = \varinjlim K_{f_\alpha},\quad A = \varinjlim A_{g_\alpha}.
\]
\begin{theorem}\label{thm:limit-lambda-chern}
  Let $K = \varinjlim K_{f_\alpha}, A = \varinjlim A_{g_\alpha}$  be as before. Then the direct limits inherit
  \begin{itemize}
      \item a Chern character $\ch: \Q K \xrightarrow{\cong} \Q A$;
      \item a total Chern class map $c: \Q K \to \Q A$;
      \item a $\lambda$-ring structure on $\Q K$.
  \end{itemize}
\end{theorem}
\begin{proof}
   By definition of the pushforwards, they commute with the Chern character, the total Chern class map and $\lambda$-operator. 
\end{proof}

\subsection{Chow and \texorpdfstring{$K$}{K}-rings of toric variety bundles}\label{sec:toric variety bundles}
In this subsection we briefly discuss another family of algebraic varieties for which we can provide a combinatorial model for Chow and $K$-rings using continuous and discrete quotient constructions.

Let us start with a quick introduction to toric variety bundles. For a more thorough introduction see \cite[Section 1]{carocci2026chow} and references there. Toric variety bundles are certain partial compactifications of principal $T$-bundles over an irreducible base $X$. The isomorphism classes of toric variety bundles over $X$ are given by \emph{bundle data}.

\begin{definition} \emph{Bundle data} over a base variety $X$ consists of:
\begin{enumerate}
\item \emph{Fibre fan:} A fan $\Sigma$ in a lattice $N$.\smallskip
\item \emph{Mixing collection:} For each $m \in M=\Hom(N,\Z)$ a line bundle $L(m)$ on $X$ together with compatible isomorphisms
\[ L(m_1) \otimes L(m_2) \cong L(m_1+m_2), \qquad L(0) \cong \Om_X. \]
\end{enumerate}
\end{definition}

Fix some bundle data $(\Sigma,L)$ over $X$. The fiber fan $\Sigma$ produces a toric variety $Z_\Sigma$ with dense torus $T$. The mixing collection $L$ produces a principal $T$-bundle $P \to X$.

\begin{definition} \label{def: toric variety bundle}The \emph{toric variety bundle} associated to the bundle data $(\Sigma,L)$ is obtained as the geometric quotient
\[
Y_\Sigma = (P\times Z_\Sigma) \slash T \xrightarrow{p}  X
\]
where $T$ acts anti-diagonally on $P \times Z_\Sigma$, so that $(tp,z) \sim (p,tz)$. 
\end{definition}
Toric variety bundles are equivalent to Zariski-locally trivial toric variety fibrations containing a dense principal torus bundle \cite[Lemma~3.4]{CarocciNabijou2}. They are constructed by gluing trivial bundles via transition functions valued in $T$ 
(as opposed to some larger subgroup of $\operatorname{Aut}
(Z_\Sigma)$).

For the purpose of this paper it will be enough to consider \emph{coarse mixing collections}, i.e. maps recording only the isomorphism classes of the $L(m)$:
\[ 
[L] \colon M \to  \Pic X.
\]
This determines the associated toric variety bundle up to (non-unique) isomorphism. 

For the rest of the subsection we fix some bundle data $(\Sigma,L)$ over a base variety $X$ and study the associated toric variety bundle $p \colon Y_\Sigma \to X$. We also assume that both the base variety $X$ and the fiber $Z_\Sigma$ are smooth and projective. 

In this case, the total space of the toric variety bundle satisfies the Leray-Hirsch theorem and, in particular, the Chow ring $A(Y_\Sigma)$ is generated as an algebra over $A(X)$ by the classes of torus invariant divisor of the form
\[
D_{ h}=\sum_{\rho\in \Sigma(1)} h_\rho D_\rho \text{ for some } h = (h_\rho)_\rho \in \Z^{\Sigma(1)}. 
\]
In particular, $A_\mathrm{num}(Y_\Sigma)$ is generated as a ring by the classes $D_h$ and $\pi^*\alpha_1,\ldots, \pi^*\alpha_r$ for any set of generators $\alpha_1,\ldots, \alpha_r$ of $A_\mathrm{num}(X)$.
Further, as in the classical toric case, to every divisor $D_h$ one can associate a (possibly virtual) polytope $P_h$ in $M_\R$. The main result of
\cite{hof2020} computes the degrees of elements of the form $\pi^*\alpha \cdot D_h^d$ in terms of the polytope $P_h$ and the intersection theory of the base variety $X$.
\begin{theorem}[{\cite[Theorem 4.1]{hof2020}}]\label{thm:bundleBKK}
    Let $\pi\colon Y_\Sigma \to X$ be as before with $\dim X = k$ and $\dim Y_\Sigma = k+n$. Let further $\alpha\in A^{k-i}(X)$, and $h \in \Z^{\Sigma(1)}$, then we get
    \[
    \deg_{Y_\Sigma} (D_h^{n+i}\cdot \pi^*\alpha) = \frac{(n+i)!}{i!}\int_{P_h} \deg_X(c_1(L(u))^i\cdot \alpha) du.
    \]
\end{theorem}
In the statement of \Cref{thm:bundleBKK}, we extend the map $[L]\colon M\to \Pic(X)$ to a map $M_\R\to \Pic(X)\otimes \R$ by linearity. Thus the integration on the right hand side of the formula makes sense.

Similar to the Chow ring, the numerical $K$-theory of the toric variety bundle $Y_\Sigma$ is generated by the classes  $\Om(D_h)$  and  $\pi^*\beta_1,\ldots \pi^*\beta_r$ for any generators $\beta_1,\ldots,\beta_r$ of $K_\mathrm{num}(X)$. The following theorem is a $K$-theoretic analogue of \Cref{thm:bundleBKK}.

\begin{theorem}\label{thm:coh}
Let $\pi\colon Y_\Sigma \to X$ and $D_h$ be as before and assume further that  the restriction of $\Om(D_h)$ to any fiber $Z_\Sigma$ of a toric variety bundle is nef. Let further $\mathcal{F}$ be a locally free sheaf on $X$.
Then the cohomology of the twisted line bundle $\Om(D_h) \otimes \pi^* \mathcal{F}$ 
on the total space $Y_\Sigma$ is given by:
\[
H^k(Y_\Sigma, \Om(D_h) \otimes \pi^* \mathcal{F}) \cong \bigoplus_{u \in P_h \cap M} H^k(X, L(u)\otimes \mathcal{F}).
\]
\end{theorem}

\begin{proof}
We compute the cohomology using the Leray spectral sequence for the fibration $\pi: Y_\Sigma \to X$:
\[
E_2^{p,q} = H^p(X, R^q \pi_* (\Om(D_h) \otimes \pi^* \mathcal{F})) \Rightarrow H^{p+q}(E_\Sigma, \Om(D_h) \otimes \pi^* \mathcal{F}).
\]
By the projection formula for the morphism $\pi$, we have an isomorphism of sheaves on $X$:
\[
R^q \pi_* (\Om(D_h) \otimes \pi^* \mathcal{F}) \cong (R^q \pi_* \Om(D_h)) \otimes \mathcal{F}.
\]
We proceed to compute the higher direct image sheaves $R^q \pi_*\Om(D_h)$. The fiber of $\pi$ over a point $x \in X$ is isomorphic to the toric variety $Z_\Sigma$. The restriction of $\Om(D_h)$ to the fiber corresponds to the torus-invariant divisor $D_h$ with corresponding polytope $P_h$. Since this restriction is ample, by the Demazure vanishing Theorem, the cohomology of the fiber satisfies:
\begin{enumerate}
    \item $H^q(Z_\Sigma, \Om_{Z_\Sigma}(D_h)) = 0$ for all $q > 0$.
    \item $H^0(Z_\Sigma, \Om_{Z_\Sigma}(D_h)) \cong \bigoplus_{u \in P_h \cap M} \C_u$ as a $T$-representation.
\end{enumerate}
Consequently, the higher direct images vanish for all $q > 0$:
\[
R^q \pi_* \Om(D_h) = 0 \quad \text{for all } q > 0.
\]
For $q=0$, the direct image $\pi_* \Om(D_h)$ is the vector bundle associated to the principal bundle $P$ and the $T$-representation $H^0(Z_\Sigma, \Om_{Z_\Sigma}(D_h))$. So we get:
\[
\pi_* \Om(D_h) \cong \bigoplus_{u \in P \cap M} L(u).
\]
Since $E_2^{p,q} = 0$ for all $q > 0$,
the spectral sequence degenerates at the $E_2$ page and the cohomology of the total space is isomorphic to the cohomology of the pushforward:
\[
H^k(Y_\Sigma, \Om(D_h)\otimes \pi^* \mathcal{F}) \cong H^k(X, \pi_* (\Om(D_h) \otimes \pi^* \mathcal{F})).
\]
Moreover, we have
\[
\pi_* (\Om(D_h) \otimes \pi^* \mathcal{F}) \cong (\pi_* \Om(D_h)) \otimes \mathcal{F} \cong  \bigoplus_{u \in P_h \cap M}L(u)  \otimes \mathcal{F},
\]
and hence
\[
H^k(Y_\Sigma, \Om(D_h) \otimes \pi^* \mathcal{F}) \cong \bigoplus_{u \in P_h \cap M} H^k(X, L(u) \otimes \mathcal{F}). \qedhere
\]
\end{proof}

Let $Y_\Sigma \to X$ be as before and let $\beta_1,\ldots, \beta_r$ be generators of $K_\num(X)$. Let us further assume that $\beta_1,\ldots, \beta_r$ are classes of vector bundles on $X$. Denote by $\chi_{\beta}(\mathbf{t})$ the function computing the Euler characteristic of the tensor product of the $\beta_i$'s:
    \[
    \chi_{\beta}(t_1,\ldots,t_r) = \chi(X, \beta_1^{t_1}\otimes\cdots\otimes\beta_r^{t_r}).
    \]

For the computation of the numerical $K$-theory of $Y_\Sigma$, we need to compute the Euler characteristic of line bundles $\Om(D_h)$ associated to divisors $D_h$, whose restrictions to fibers of $\pi:Y_\Sigma\to X$ are not necessarily ample. In this general case, we associate a virtual polytope to the divisor $D_h$. Any virtual polytope has a geometric realization as a convex chain i.e.\ a linear combination of characteristic functions of convex polytopes. For details on virtual polytopes and their realization as convex chains we refer to \cite{PK92}. We denote by $\cP_\Sigma$ the vector space of virtual polytopes with normal fan $\Sigma$ and by $\cP_\Sigma^+$ the cone of convex polytopes in $\cP_\Sigma$. By the projectivity assumption, $\cP_\Sigma^+\subset \cP_\Sigma$ is a full-dimensional polyhedral cone. Furthermore, we denote by $\cP_{\Sigma,\Z}$ the lattice of integer virtual polytopes, and by $\cP_{\Sigma,\Z}^+$ the semigroup of convex polytopes with normal fan $\Sigma$.
\begin{corollary}\label{thm:bundleEuler}
We have that for all $h \in \Z^{\Sigma(1)}$, and $\mathbf{t} \in \Z^r$,
\[
\chi(Y_\Sigma, \Om(D_h)\otimes \pi^*\beta_1^{t_1}\otimes\ldots\otimes \pi^*\beta_r^{t_r})
=\sum_{u\in M} \mathbf{1}_{P_h}(u) \cdot \chi(X,\beta_1^{t_1}\otimes\ldots\otimes\beta_r^{t_r} \otimes L(u)),
\]
where $\mathbf{1}_{P_h}$ denotes the convex chain associated to the virtual polytope $P_h$.
\end{corollary}
 
\begin{proof}
    Identify $\Z^{\Sigma(1)} \cong \cP_{\Sigma, \Z}$ via $h \mapsto P_h$.
    For fixed $t_1,\ldots, t_r$, both sides of the equality are polynomials when restricted to $\cP_\Sigma$. Moreover, they agree on the full rank sub-semigroup of convex integral polytopes $\cP_{\Sigma,\Z}^+$ by \Cref{thm:coh} and hence they agree everywhere.
\end{proof}

As an immediate corollary of \Cref{thm:bundleBKK} and \Cref{thm:bundleEuler} we obtain the following polyhedral description for numerical $K$ and Chow rings of $Y_\Sigma$.
\begin{theorem}\label{thm:bundle K and Chow}
Let $\pi:Y_\Sigma\to X$ be a toric variety bundle as before and let $\Gamma$ be a graded semigroup generating numerical $K$-ring of $X$. Then we get:
\[
K_{\mathrm{num}}(Y_\Sigma) = K_{\Ehr_X},\quad A_{\mathrm{num}}(Y_\Sigma) = A_{\Vol_X},
\]
where functions $\Ehr_X$ and $\Vol_X$ are given by
\begin{align*}
    \Ehr_X\colon \cP_{\Sigma,\Z}\times \Gamma \to \Z, \quad &(P,\gamma)\mapsto \sum_{u\in M} \mathbf{1}_{P}(u) \cdot \chi(X,\gamma\otimes L(u)),\\
    \Vol_X \colon \cP_{\Sigma,\Z}\times \Gamma \to \Z, \quad &(P,\gamma)\mapsto \frac{(n+k -\rk(\gamma))!}{(k-\rk(\gamma))!}\int_{P} \deg_X(c_1(L(u))^{k-\rk(\gamma)}\cdot \gamma) du\,.
\end{align*}
\end{theorem}

A special case of toric variety are toroidal horospherical varieties which are $G$-equivariant compactifications of horospherical homogeneous spaces $G/H$. The  ring of conditions of horospherical homogeneous space $G/H$ was computed in \cite{khovanskii2021gorenstein}  using the Frobenius quotient description of the Chow rings of toric variety bundles. Using \Cref{thm:bundle K and Chow}, we can give a description for the $K$-theoretic version of the ring of conditions $KC(G/H)$ as in \Cref{rem:ring-conditions}. Moreover, by \Cref{thm:limit-lambda-chern}, we obtain the following statement.
\begin{corollary}\label{cor:conditions-ring-horospherical}
Let $G/H$ be a horospherical homogeneous space and let $KC(G/H)$ and $C(G/H)$ be its $K$-theoretic and classical rings of conditions. Then we have
 \begin{itemize}
      \item a Chern character $\ch: \Q KC(G/H) \xrightarrow{\cong} \Q C(G/H)$;
      \item a total Chern class map $c: \Q KC(G/H) \to \Q C(G/H)$;
      \item a $\lambda$-ring structure on $\Q KC(G/H)$.
  \end{itemize}
\end{corollary}

\subsection{Rational polyhedral fans}

Chow rings and $K$-rings of toric varieties can be described explicitly in terms of their fans. Although the associated toric varieties may be non-complete, these Chow rings often behave much like those of complete varieties. Consequently, Chow and $K$-rings defined for fans have attracted considerable attention, often independently of the study of toric varieties themselves. In this subsection, we study these rings from the viewpoint of shift and differential operators.

\subsubsection{Chow and $K$-rings of fans}
Let $\Sigma$ be a smooth, not necessarily complete fan in $N_\Q$, where $N$ is a lattice with dual $M=\Hom(N,\Z)$. Let $\PL(\Sigma)$ denote the integral piecewise linear functions on $\Sigma$, that is, those functions $\vert\Sigma\vert\to \R$ which are linear on each cone and have integral value on the primitive generator $u_\rho$ of every ray $\rho$. There is a natural basis for $\PL(\Sigma)$ consisting of the Courant functions. For an element $\rho$ of the set $\Sigma(1)$ of rays of $\Sigma$, the Courant function $x_\rho$ is the unique element of $\PL(\Sigma)$ with value $1$ on $u_\rho$ and constant value $0$ on all other rays. 

As all linear maps $N \to \Z$ restrict to integral piecewise linear functions on $\Sigma$, we obtain a map
\[
M\to \PL(\Sigma),\;\; m\mapsto \sum_{\rho\in\Sigma(1)} \langle m,u_\rho\rangle x_\rho.
\]

The (rational) Chow ring $\Q A(\Sigma)$ of $\Sigma$,
is the graded ring given by
\[
    \Q A(\Sigma)=\Sym(\PL_\Q(\Sigma))/(\mc I_\Sigma +\mc J_\Sigma),
\]
where $\mc I_\Sigma$ is the Stanley-Reisner ideal given by
\[
    \mc I_\Sigma =\left\langle \prod_{\rho\in I} x_\rho : I\subseteq \Sigma(1) \text{ does not span a cone of }\Sigma \right\rangle,
\]
and 
\[
    \mc J_\Sigma =\langle M\rangle   
\]
is generated by the linear functions.
This coincides with the Chow ring of the toric variety $X_\Sigma$, under the identification
$x_\rho \mapsto [D_\rho]$, where $D_\rho$ is the torus-invariant divisor corresponding to the ray $\rho$ on $X_\Sigma$.

We may describe the $K$-ring of $X_\Sigma$ using the ring of piecewise exponential functions on $\Sigma$, which is the subring $\PE(\Sigma)$
of $\Fun(|\Sigma|, \R)$ generated by the functions of the form $e^f$ for $f \in \PL(\Sigma)$.
We let
\[
    K(\Sigma) = \frac{\PE(\Sigma)}{\langle e^m - 1 \mid m \in M\rangle}.
\]
Then by \cite{Brion1997, Merkurjev1997} we have that $K(\Sigma) \cong K(X_\Sigma)$ via the map induced by $e^{x_\rho} \mapsto [\cO(D_\rho)]$. 

Under these identifications, the
Chern character  $\ch\colon \Q K(X_\Sigma)\xrightarrow{\cong} \Q A(X_\Sigma)$ may be written as
\[
    \ch\colon \Q K(\Sigma) \to \Q A(\Sigma), \;\; e^{x_\rho} \mapsto \exp(x_\rho).
\]

\subsubsection{Differential operators for fans}

Denote by $\partial_\rho$ the derivative on $\PL_\Q(\Sigma)$ in the direction of $x_\rho$. This is the element corresponding to $x_\rho$ under the natural isomorphism $\Sym(\PL_\Q(\Sigma))\cong \Diff_\Q(\PL(\Sigma))$. Using this isomorphism, we obtain a description of $\Q A(\Sigma)$ in terms of differential operators as
\[
\Q A^*(\Sigma)=\Diff_\Q(\PL(\Sigma))/(\mc I'_\Sigma+\mc J'_\Sigma), 
\]
where 
\[
    \mathcal I_\Sigma' =\left\langle \prod_{\rho\in I} \partial_\rho : I\subseteq \Sigma(1) \text{ does not span a cone of }\Sigma \right\rangle ,
\]
and 
\[
    \mathcal J_\Sigma' =\langle \partial_m : m\in M\rangle.
\]

Recall that \emph{$k$-dimensional Minkowski weights} are maps $\Sigma(k)\to \Q$ that satisfy the so-called balancing condition. If $\MW_k(\Sigma)$ denotes the group of $k$-dimensional Minkowski weights, then by \cite[Theorem 1]{FMSSspherical}, there is a canonical isomorphism
\[
    \Hom_\Q(\Q A^k(\Sigma),\Q)\cong \MW_k(\Sigma),\;\; 
    \phi\mapsto \left(\sigma\mapsto \phi\left(\prod_{\rho\in \sigma(\rho)}x_\rho\right)\right).
\]

As a corollary to Proposition \ref{prop:solution set to differential equation}, we obtain a description of Minkowski weights on a fan as polynomial function satisfying the differential equations given by $\mc I_\Sigma'$ and $\mc J_\Sigma'$.

\begin{corollary}
\label{cor:solution set to differential equations}
There is a canonical isomorphism of $\Q A(\Sigma)$-modules between $\MW(\Sigma)$ and the module consisting of 
polynomials $f$ on $\PL(\Sigma)$ satisfying the partial differential equations
\begin{align*}
\partial_m f&=0  \qquad \text{for all }m\in M, \\
\left(\prod_{\rho\in I}\partial_\rho \right) f &= 0  \qquad \text{for all }I\subseteq \Sigma(1)\text{ not spanning a cone.}  
\end{align*}
\end{corollary}

We say that $\Sigma$ is balanced, if it is pure-dimensional, say of dimension $d$, and assigning $1$ to all $d$-dimensional cones defines a Minkowski weight. The associated linear map
\[
    \deg\colon \Q A(\Sigma) \to \Q
\]
is the \emph{degree} map. The polynomial function 
\[
    \Vol_\Sigma \colon \PL_\Q(\Sigma) \to \Q,\;\; f\mapsto \deg(\exp(f))= \frac{1}{d!}\deg(f^n)
\]
is the \emph{volume polynomial.} 

Denote by $\Q A_\num(\Sigma)$ the \emph{numerical Chow ring of $\Sigma$}, that is, the quotient of $\Q A(\Sigma)$ by the ideal 
\[
    \{a\in \Q A(\Sigma) \mid \deg(a\cdot b)=0 \text{ for all }b\in \Q A(\Sigma)\}.
\]
By definition, it is a Poincaré duality ring generated in degree $1$ by $\PL_\Q(\Sigma)$ and whose degree map $\deg$ (which descends to $\Q A_\num(\Sigma)$) corresponds to the volume polynomial. By \Cref{thm:diffalg}, we obtain that $\Q A_\num(\Sigma)\cong \Q A_{\Vol_\Sigma}$.

\subsubsection{Ehrhart fans}

The holomorphic Euler characteristic defines morphisms of abelian groups from the $K$-rings of compact varieties to the integers. In particular, one obtains morphisms $K(\Sigma)\to \Q$ whenever $\Sigma$ is complete or supported on the Bergman fan of a representable matroid. This was generalized to non-representable matroids in \cite{LLPP24} and further to a larger class of fans, called \emph{Ehrhart fans}, which were introduced in  \cite{EhrhartFans}. 
More specifically, if $\Ehr_\Sigma: \PL(\Sigma) \to \Z$ denotes the Ehrhart function of $\Sigma$, then by \cite[Corollary 5.20]{EhrhartFans}
we obtain a morphism
\[
    \chi_\Sigma\colon K(\Sigma)\to \Z \ 
\]
generated by $e^f \mapsto \Ehr_\Sigma(f)$ for all $f \in \PL(\Sigma)$.

As before, we may consider the \emph{numerical $K$-ring $K_\num(\Sigma)$ of $\Sigma$}, defined by the quotient of $K(\Sigma)$ by the ideal 
\[
    \{\xi \in K(\Sigma) \mid \chi_\Sigma(\xi \cdot \eta) = 0 \text{ for all } \eta \in K(\Sigma) \}.
\]
Then $K_\num(\Sigma)$ is a Frobenius ring with non-degenerate pairing induced by $\chi_\Sigma$.
Note that we have a natural surjection $\Sh(\PL(\Sigma)) \to K(\Sigma)$ given by $T^f \mapsto e^f$ for all $f \in \PL(\Sigma)$.
In particular, by \Cref{thm:shiftalg}, we obtain that $K_\num(\Sigma) \cong K_{\Ehr_\Sigma}$.

We would like to know when does the Chern character $\Q K(\Sigma) \xrightarrow{\cong} \Q A(\Sigma)$ descend to an isomorphism on the numerical quotients $\Q K_\num(\Sigma)$ and $\Q A_\num(\Sigma)$. By \Cref{thm:ch-descends-iff-todd}, this is the case if and only if $\Ehr_\Sigma \sim \Vol_\Sigma$. This prompts the following question:

\begin{question}
\label{quest:Ehrhart function same annihilator as volume function?}
Let $\Sigma$ be a purely $d$-dimensional Ehrhart fan. Do we have $\Ehr_\Sigma \sim \Vol_\Sigma$?
\end{question}

We give a positive answer to this question when $\Sigma$ has Poincar\'e duality, meaning that $\Q A(\Sigma)$ is a Poincar\'e duality ring.

\begin{proposition}\label{prop:poincare-ehrhart-hrr}
If $\Sigma$ is an Ehrhart fan with Poincar\'e duality, then $\Ehr_\Sigma \sim \Vol_\Sigma$.
\end{proposition}

\begin{proof}
If $\Q A(\Sigma)$ is a Poincaré duality ring, we have 
\[
    \Q A(\Sigma)=\Q A_\num(\Sigma)\cong \Q A_{\Vol_\Sigma}.
\]
In particular, we obtain a surjective map $\Q A_{\Vol_\Sigma} \twoheadrightarrow \Q K_{\Ehr_\Sigma}$ induced by $\ch^{-1}$. In particular, the fact that this map is well-defined forces $\Q \Ann_\Diff(\Vol_\Sigma) \subseteq \Q \Ann_\Diff(\Ehr_\Sigma)$.
By \Cref{prop:polynomial-equivalence-inclusion}, this implies that
there exists some differential operator $\alpha \in \Diff_\Q(\PL(\Sigma))$, such that
$\Ehr_\Sigma = \alpha \cdot \Vol_\Sigma$.
By \cite[Proposition 3.13]{EhrhartFans} the degree-$d$ part of $\Ehr_\Sigma$ equals $\Vol_\Sigma$ (note that our volume polynomial differs by a factor of $1/d!$ from the one in \cite{EhrhartFans}), implying that $\alpha$ has constant coefficient $1$ and so  $\alpha \in \Q A_{\Vol_\Sigma}^\times$.
\end{proof}

\subsubsection{Exceptional isomorphisms for fans}

The Courant functions are a natural basis of $\PL(\Sigma)$.
We now study conditions for the existence of an exceptional isomorphism
$\Q K_{\Ehr_\Sigma}\to \Q A_{\Vol_\Sigma}$ with respect to this basis.
Define as in \Cref{sec:exceptional-isoms} the map $\varphi_{\frac{1}{1 - t}}: \Sh_\Q(\PL(\Sigma)) \to \Q A_{\Vol_\Sigma}$ determined by sending $D_{x_\rho}$ to $\partial_{\rho}$ for all $\rho$. As noted in the proof of \Cref{thm:exceptional-general}, this map is surjective.

To understand when this map descends to an isomorphism, it will be useful to rephrase Ehrharticity in the language of difference equations:

\begin{proposition}
\label{prop:IVP for Ehrharticity}
A unimodular fan $\Sigma$ is Ehrhart if and only if there exists a polynomial function $\Ehr_\Sigma$ on $\PL(\Sigma)$ satisfying the following sets of difference equations:

\begin{enumerate}[label=(\roman*)]
\item 
$D_m\Ehr_\Sigma=0$ \quad for all $m\in M$
\item
$\left(\prod_{\rho\in I} D_{x_\rho} \right)\Ehr_\Sigma =0$ \quad for all $I\subseteq\Sigma(1)$
not spanning a cone.
\item
$\left(\left(\prod_{\rho\in I} D_{x_\rho}\right) \Ehr_\Sigma\right)(0) =1$ \quad for all $I\subseteq\Sigma(1)$ spanning a cone.
\end{enumerate}
\end{proposition}

\begin{proof}
Suppose $\Sigma$ is Ehrhart with $\Ehr_\Sigma$ the Ehrhart polynomial on $\PL(\Sigma)$. Since $\Ehr_\Sigma$ descends to the quotient $\PL(\Sigma)/M$, it satisfies (i). Let $\Sigma^\rho$ be the star fan of $\rho$, then we may see $\Ehr_{\Sigma^\rho}$ as a polynomial on $\PL(\Sigma)$ using the homomorphism $\PL(\Sigma) \to \PL(\Sigma^\rho), [f] \mapsto [f]^\rho$. Then $D_{x_\rho} \Ehr_{\Sigma} = \Ehr_{\Sigma^\rho}$.
In particular, since for any $\tau, \rho \in \Sigma(1)$, 
$[x_\tau]^\rho = 0$ if and only if $\tau$ and $\rho$ do not span a cone, (ii) and (iii) follow by induction, and from the fact that $\Ehr_{\Sigma'}(0) = 1$ for all Ehrhart fans $\Sigma'$.

Now, suppose there is a polynomial $\Ehr_\Sigma$ such that conditions (i)-(iii) are satisfied. We will show that $\Sigma$ is Ehrhart with Ehrhart polynomial given by $\Ehr_\Sigma$ by induction on the dimension of $\Sigma$. By (i), $\Ehr_\Sigma$ descends to a well-defined function on $\PL(\Sigma)/M$. Furthermore, $D_{x_\rho} \Ehr_\Sigma$ is a polynomial function on $\PL(\Sigma^\rho)/M$ which also satisfies (i)-(iii), and hence by induction $\Sigma^\rho$ is Ehrhart with Ehrhart polynomial of $\Sigma^\rho$. Hence the conditions (1) and (2) in \cite[Definition 3.3]{EhrhartFans} are both satisfied and we are done.
\end{proof}

We use this characterization of Ehrhart fans to give an equivalent condition for the existence of the exceptional isomorphism.

\begin{proposition}
\label{prop:condition on exceptional iso for Ehrhart fans}
Let $\Sigma$ be an Ehrhart fan with Poincar\'e duality. $\varphi_{\frac{1}{1 - t}}$ descends to an isomorphism $\Q K_{\Ehr_\Sigma}\cong \Q A_{\Vol_\Sigma}$ if and only if
 \[
    \left(\sum_{\rho\in\Sigma(1)} \langle m,u_\rho\rangle D_{x_\rho}\right) \cdot \Ehr_\Sigma=0
 \]
 for all $m\in M$. 
\end{proposition}

\begin{proof}
Let $\zeta: \Diff_\Q(\PL(\Sigma)) \to \Sh_\Q(\PL(\Sigma))$ be the map given by $\partial_\rho \mapsto D_{x_\rho}$ for all $\rho \in \Sigma(1)$. Then clearly $\varphi_{\frac{1}{1-t}}$ descends to an isomorphism if and only if $\zeta$ does.
Note that 
$\sum_{\rho\in\Sigma(1)} \langle m,u_\rho\rangle D_{x_\rho} \cdot \Ehr_\Sigma = 0$ holds for all $m\in M$
if and only if $\zeta(\mc J'_\Sigma) \subseteq \Ann(\Ehr_\Sigma)$. Because $\Sigma$ has Poincar\'e duality, we have $\Q A(\Sigma) \cong \Q A_{\Vol_\Sigma}$, and hence $\Ann(\Vol_\Sigma) = \mc I'_\Sigma + \mc J'_\Sigma$.

If $\zeta$ descends to an isomorphism, then we necessarily have $\zeta(\mc J'_\Sigma) \subseteq \Ann(\Ehr_\Sigma)$.
In the other direction, suppose  $\zeta(\mc J'_\Sigma) \subseteq \Ann(\Ehr_\Sigma)$, then 
\Cref{prop:IVP for Ehrharticity}, implies that $\zeta(\mc I'_\Sigma) \subseteq \Ann(\Ehr_\Sigma)$.
Since $\Ann(\Vol_\Sigma) = \mc I'_\Sigma + \mc J'_\Sigma$, this shows that $\zeta$ descends to a well-defined map
$\Q A_{\Vol_\Sigma} \to \Q K_{\Ehr_\Sigma}$. Since the $[D_{x_\rho}]$ generate $\Q K_{\Ehr_\Sigma}$ as a $\Q$-module, the map is surjective. Finally, since $\Sigma$ has Poincaré duality ring, we have by \Cref{prop:poincare-ehrhart-hrr} that the Chern character induces an isomorphism between $\Q K_{\Ehr_\Sigma}$ and $\Q A_{\Vol_\Sigma}$, so both rings have the same dimension and so the map induced by $\zeta$ must be an isomorphism.
\end{proof}

When the exceptional isomorphism exists, we also fully understand the exceptional Todd class.

\begin{proposition}\label{prop:exceptional-implies-balanced}
Let $\Sigma$ be an Ehrhart fan. If 
$\varphi_{\frac{1}{1 - t}}$ descends to an isomorphism $\Q K_{\Ehr_\Sigma}\cong \Q A_{\Vol_\Sigma}$, then
the exceptional Todd class $\td_\ex$ is represented by the all-ones Minkowski weight 
\[
    \Sigma\ni \sigma\mapsto 1\in \Q .
\]
In particular, all skeleta of $\Sigma$ are balanced.
\end{proposition}

\begin{proof}
The exceptional Todd class $\td_\ex$ is the Poincar\'e dual of the composite
\[
    \Q A(\Sigma)\xrightarrow{\cong} \Q A_{\Vol_\Sigma}\xrightarrow{{\varphi_{\frac{1}{1 - t}}}^{-1}}\Q K_{\Ehr_\Sigma}\xrightarrow{\chi_\Sigma} \Q.
\]
On the square free monomial corresponding to a cone $\sigma\in\Sigma$, this composite acts as
\[
    \prod_{\rho\in\sigma(1)}x_\rho\mapsto\prod_{\rho\in\sigma(1)}\partial_\rho \mapsto \prod_{\rho\in\sigma(1)}D_{x_\rho}\mapsto 1.
\]
It follows that the Minkowski weight associated to the composite has weight $1$ on all cones.
\end{proof}

Finally, we show that in dimension 2 the condition of \Cref{prop:exceptional-implies-balanced} is also sufficient for existence of the exceptional isomorphism.

\begin{proposition}\label{prop:balanced-implies-exceptional}
Let $\Sigma$ be a two-dimensional Ehrhart fan with Poincaré duality. If the $1$-skeleton of $\Sigma$ is balanced, 
$\varphi_{\frac{1}{1 - t}}$ descends to an isomorphism $\Q K_{\Ehr_\Sigma}\cong \Q A_{\Vol_\Sigma}$.
\end{proposition}

\begin{proof}
Consider the element
\[
z=\sum_{\rho\in \Sigma(1)}\ch(D_{x_\rho})\otimes u_\rho \in \Q A_{\Vol_\Sigma} \otimes_\Q N_\Q \ .
\]
Note that $z=0$ if and only if $\sum_{\rho\in\Sigma(1)}\langle m,u_\rho\rangle \ch(D_{x_\rho})=0$ in $A_{\Vol_\Sigma}$ for all $m 
\in M$, which is equivalent to the existence of the exceptional isomorphism by Proposition \ref{prop:condition on exceptional iso for Ehrhart fans}. In other words, it suffices to show that $z=0$. The degree-$0$ graded part of $z$ is $0$, because the same is true for all the terms $\ch(D_{x_\rho})$ appearing in the expression. The degree-$1$ part of $z$ equals
\[
    \sum_{\rho\in\Sigma(1)}\partial_{\rho}\otimes u_\rho=0,
\]
which also vanishes because $\partial_m\Vol_\Sigma=0$ for all $m\in M$. Therefore, $z$ is concentrated in degree $2$. As the degree map is an isomorphism in degree $2$, it suffices to show that $\deg(z)=0$. But because $z$ is concentrated in degree $2$, we have $z=\td\cdot z$ and hence
\[
    \deg(z)= \deg(\td\cdot z)= \sum_{\rho\in \Sigma(1)} (D_{x_\rho}\cdot \Ehr_\Sigma(0)) \otimes u_\rho = \sum_{\rho\in\Sigma(1)}u_\rho =0 \ ,
\]
where the second equality is by Hirzebruch-Riemann-Roch, the third one is by Proposition \ref{prop:IVP for Ehrharticity}, and the final one is precisely the balancing condition for the $1$-skeleton.
\end{proof}

\begin{question}
Does the equivalence between the existence of an exceptional isomorphism and the balancing of the skeleta also hold in dimension greater than $2$?
\end{question}

\printbibliography

@preamble{
   "\def\cprime{$'$} "
}

@article{BEST2023,
  title={Tautological classes of matroids},
  author={Berget, Andrew and Eur, Christopher and Spink, Hunter and Tseng, Dennis},
  journal={Invent. math.},
  volume={233},
  pages={951--1039},
  year={2023},
  publisher={Springer},
  doi={10.1007/s00222-023-01194-5}
}

@article{BerZel,
	Author = {Berenstein, A. and Zelevinsky, A.},
	Doi = {10.1007/s002220000102},
	Journal = {Invent. math.},
	Mrclass = {17B10 (12K10 17B20 17B37 20G05)},
	Mrnumber = {1802793},
	Mrreviewer = {Ivan Arzhantsev},
	Pages = {77--128},
	Title = {Tensor product multiplicities, canonical bases and totally positive varieties},
	Volume = {143},
	Year = {2001},
}

@online{EhrhartFans,
  title = {Piecewise-Exponential Functions and {{Ehrhart}} Fans},
  author = {Chan, Melody and Clader, Emily and Klivans, Caroline and Ross, Dustin},
  year = 2025,
  month = jul,
  number = {arXiv:2503.22636},
  eprint = {2503.22636},
  primaryclass = {math.CO},
  publisher = {arXiv},
  archiveprefix = {arXiv}
}

@online{Cheng2025,
  title = {On the Tangent Bundle and the Divisor Theory of a General Matroid},
  author = {Cheng, Ronnie},
  year = 2026,
  month = mar,
  eprint = {2510.06609},
  primaryclass = {math.AG},
  publisher = {arXiv},
  archiveprefix = {arXiv},
}

@online{cheng2026tangent,
  title={Tangent classes for matroid building sets},
  author={Cheng, Ronnie},
  year={2026},
  eprint = {2606.22650},
  primaryclass = {math.AG},
  publisher = {arXiv},
  archiveprefix = {arXiv}
}

@article{carocci2026chow,
  title={Chow theory of toric variety bundles},
  author={Carocci, Francesca and Monin, Leonid and Nabijou, Navid},
  journal={International Mathematics Research Notices},
  volume={2026},
  number={1},
  pages={rnaf369},
  year={2026},
  publisher={Oxford University Press},
doi = {10.1093/imrn/rnaf369},
}

@article{CarocciNabijou2,
       author = {{Carocci}, F. and {Nabijou}, N.},
        title = "{Tropical expansions and toric variety bundles}",
        doi = {10.1007/s00229-025-01684-1},
        journal = {manuscripta math.},
        volume={177},
        number={3},
        year={2026}
}

@book {Eis,
    AUTHOR = {Eisenbud, David},
     TITLE = {Commutative algebra, with a view toward algebraic geometry},
    SERIES = {Graduate Texts in Mathematics},
    VOLUME = {150},
      NOTE = {},
  publisher = {Springer New York},
      YEAR = {1995},
       DOI = {10.1007/978-1-4612-5350-1},
}

@book{FultonLang85,
  title = {Riemann-{{Roch Algebra}}},
  author = {Fulton, William and Lang, Serge},
  year = 1985,
  series = {Grundlehren Der Mathematischen {{Wissenschaften}}},
  volume = {277},
  publisher = {Springer New York},
  %address = {New York, NY},
  doi = {10.1007/978-1-4757-1858-4},
  copyright = {http://www.springer.com/tdm},
  langid = {english},
  doi = {10.1007/978-1-4757-1858-4}
}

@article{FS97,
  title = {Intersection Theory on Toric Varieties},
  author = {Fulton, William and Sturmfels, Bernd},
  year = 1997,
  month = mar,
  journal = {Topology},
  volume = {36},
  number = {2},
  pages = {335--353},
  %issn = {0040-9383},
  doi = {10.1016/0040-9383(96)00016-X},
}

@book{fulton93,
  title = {Introduction to Toric Varieties},
  author = {Fulton, William},
  year = 1993,
  series = {Annals of Mathematics Studies},
  number = {no. 131},
  publisher = {Princeton University Press},
  %address = {Princeton, NJ},
  isbn = {978-0-691-03332-7},
  lccn = {QA571 .F85 1993},
  keywords = {Toric varieties},
}

@article{FY04,
  title = {Chow Rings of Toric Varieties Defined by Atomic Lattices},
  author = {Feichtner, Eva Maria and Yuzvinsky, Sergey},
  year = 2004,
  month = mar,
  journal = {Inventiones Mathematicae},
  volume = {155},
  number = {3},
  pages = {515--536},
  %issn = {0020-9910, 1432-1297},
  doi = {10.1007/s00222-003-0327-2},
}

@article{hof2020,
  title={Cohomology rings of toric bundles and the ring of conditions},
  author={Hofscheier, Johannes and Khovanskii, Askold and Monin, Leonid},
  journal={Arnold Mathematical Journal},
  volume={10},
  number={2},
  pages={171--221},
  year={2024},
  publisher={Springer},
  doi={10.1007/s40598-023-00233-6}
}

@book{iarrobino1999,
  title = {Power {{Sums}}, {{Gorenstein Algebras}}, and {{Determinantal Loci}}},
  author = {Iarrobino, Anthony and Kanev, Vassil},
  year = 1999,
  series = {Lecture {{Notes}} in {{Mathematics}}},
  volume = {1721},
  publisher = {Springer Berlin Heidelberg},
  %address = {Berlin, Heidelberg},
  doi = {10.1007/BFb0093426},
}

@article{Kaveh11,
	Author = {Kaveh, K.},
	Fjournal = {Journal of Lie Theory},
	Issn = {0949-5932},
	Journal = {J. Lie Theory},
	Number = {2},
	Pages = {263--283},
	Title = {Note on cohomology rings of spherical varieties and volume polynomial},
	Volume = {21},
	Year = {2011}}

@online{khovanskii2021gorenstein,
      title={Gorenstein algebras and toric bundles}, 
      author={Askold Khovanskii and Leonid Monin},
      year={2021},
      eprint={2106.15562},
      archivePrefix={arXiv},
      primaryClass={math.AC},
}

@book{Knu73,
  title = {{$\lambda$}-{{Rings}} and the {{Representation Theory}} of the {{Symmetric Group}}},
  author = {Knutson, Donald},
  year = 1973,
  series = {Lecture {{Notes}} in {{Mathematics}}},
  volume = {308},
  publisher = {Springer Berlin Heidelberg},
  %address = {Berlin, Heidelberg},
  doi = {10.1007/BFb0069217},
}

@online{kavvil,
      title={On a notion of anticanonical class for families of convex polytopes}, 
      author={Kiumars Kaveh and Elise Villella},
      year={2018},
      month={2},
      eprint={1802.06674},
      archivePrefix={arXiv},
      primaryClass={math.AG},
}

@article{littelmanncones,
	Author = {Littelmann, P.},
	Journal = {Transformation groups},
	Number = {2},
	Pages = {145--179},
	Publisher = {Springer},
	Title = {Cones, crystals, and patterns},
	Volume = {3},
	Year = {1998},
    doi = {10.1007/BF01236431}}

@article{LLPP24,
title = {K-rings of wonderful varieties and matroids},
journal = {Advances in Mathematics},
volume = {441},
year = {2024},
%issn = {0001-8708},
doi = {10.1016/j.aim.2024.109554},
author = {Matt Larson and Shiyue Li and Sam Payne and Nicholas Proudfoot},
}

@misc{MoninSmirnov23,
 author = {Leonid Monin and Evgeny Smirnov},
 title = {Polyhedral models for {K}-theory of toric and flag varieties},
 year = {2026},
eprint={2606.07142},
archivePrefix={arXiv},
primaryClass={math.AG},
}

@article{PK92,
 author = {Pukhlikov, A. V. and Khovanskij, A. G.},
 title = {Finitely additive measures of virtual polytopes},
 fjournal = {St. Petersburg Mathematical Journal},
 journal = {St. Petersbg. Math. J.},
 issn = {1061-0022},
 volume = {4},
 number = {2},
 pages = {337--356},
 year = {1993},
 zbMATH = {146391},
 Zbl = {0791.52010}
}

@article{totaro14, 
    title={Chow groups, {C}how cohomology, and linear varieties},
    volume={2}, 
    DOI={10.1017/fms.2014.15}, 
    journal={Forum of Mathematics, Sigma}, 
    author={Totaro, Burt}, 
    year={2014}, 
    pages={e17}
}

@article{VanDerWaerden,
 author = {Van der Waerden, B. L.},
 title = {On {Hilbert}'s function, series of composition of ideals and a generalisation of the theorem of {B{\'e}zout}.},
 fjournal = {Proceedings. Akadamie van Wetenschappen Amsterdam},
 journal = {Proc. Akad. Wet. Amsterdam},
 issn = {0370-0348},
 volume = {31},
 pages = {749--770},
 year = {1928},
 zbMATH = {2576183},
 JFM = {54.0190.02}
}

@article{KleimanSnapper,
 author = {Kleiman, S. L.},
 title = {Toward a numerical theory of ampleness},
 fjournal = {Annals of Mathematics. Second Series},
 journal = {Ann. Math. (2)},
 %issn = {0003-486X},
 volume = {84},
 pages = {293--344},
 year = {1966},
 doi = {10.2307/1970447},
 zbMATH = {3235756},
 Zbl = {0146.17001}
}

@article{CombinatorialHodge,
 author = {Adiprasito, Karim and Huh, June and Katz, Eric},
 title = {Hodge theory for combinatorial geometries},
 fjournal = {Annals of Mathematics. Second Series},
 journal = {Ann. Math. (2)},
 %issn = {0003-486X},
 volume = {188},
 number = {2},
 pages = {381--452},
 year = {2018},
 doi = {10.4007/annals.2018.188.2.1},
 keywords = {14T15,05A99,05E16,14F99},
 zbMATH = {6921184},
 Zbl = {1442.14194}
}

@article{NormalComplexes,
 author = {Nathanson, Anastasia and Ross, Dustin},
 title = {Tropical fans and normal complexes: putting the ``volume'' back in ``volume polynomials''},
 fjournal = {Advances in Mathematics},
 journal = {Adv. Math.},
 %issn = {0001-8708},
 volume = {420},
 note = {108981},
 year = {2023},
 doi = {10.1016/j.aim.2023.108981},
 keywords = {14T15,52B40,05B35},
 zbMATH = {7679037},
 Zbl = {1536.14059}
}

@article{Snapper,
 %ISSN = {00959057, 19435274},
 URL = {http://www.jstor.org/stable/24900666},
 author = {Ernst Snapper},
 journal = {Journal of Mathematics and Mechanics},
 number = {6},
 pages = {967--992},
 publisher = {Indiana University Mathematics Department},
 title = {Multiples of Divisors},
 volume = {8},
 year = {1959}
}

@misc{InverseSystem,
 author = {Macaulay, F. S.},
 title = {The algebraic theory of modular systems.},
 year = {1916},
 howpublished = {Cambridge: {University} press, {XIV} u. 112 {S}. {{\(8^{\circ}\)}}.},
 zbMATH = {2610079},
 JFM = {46.0167.01},
 doi = {10.3792/chmm/1263317740}
}

@book{hirzebruch1956neue,
  title={Neue topologische {M}ethoden in der algebraischen {G}eometrie},
  author={Hirzebruch, Friedrich},
  year={1956},
  publisher={Springer Berlin, Heidelberg},
  doi={10.1007/978-3-662-41083-7},
}

@incollection{de1985complete,
  title={Complete symmetric varieties II Intersection theory},
  author={De Concini, Corrado and Procesi, Claudio},
  booktitle={Algebraic groups and related topics},
  volume={6},
  pages={481--514},
  year={1985},
  publisher={Mathematical Society of Japan},
  doi={10.2969/aspm/00610481}
}

@article{kazarnovskii2021newton,
  title={Newton polytopes and tropical geometry},
  author={Kazarnovskii, B. Ya. and Khovanskii, A. G. and Esterov, A. I.},
  journal={Russian Mathematical Surveys},
  volume={76},
  number={1},
  pages={91--175},
  year={2021},
  publisher={London Mathematical Society, Turpion Ltd and the Russian Academy of Sciences},
  doi={10.1070/RM9937}
}

@article{strickland2011equivariant,
  title={The equivariant ring of conditions of conics},
  author={Strickland, E},
  journal={Journal of Algebra},
  volume={329},
  number={1},
  pages={274--285},
  year={2011},
  publisher={Elsevier},
  doi={10.1016/j.jalgebra.2009.12.002}
}

@article{strickland2008ring,
  title={The ring of conditions of a semisimple group},
  author={Strickland, Elisabetta},
  journal={Journal of Algebra},
  volume={320},
  number={7},
  pages={3069--3078},
  year={2008},
  publisher={Elsevier},
  doi = {10.1016/j.jalgebra.2008.07.003},
}

@article{esterov2017characteristic,
  title={Characteristic classes of affine varieties and Pl{\"u}cker formulas for affine morphisms},
  author={Esterov, Alexander},
  journal={Journal of the European Mathematical Society},
  volume={20},
  number={1},
  pages={15--59},
  year={2017},
  doi={10.4171/jems/758}
}

@misc{gibson2018rings,
  title={Rings of Conditions of Rank 1 Spherical Varieties},
  author={Gibson, Julia},
  year={2018},
  note={Master's Thesis, McMaster University},
  url={http://hdl.handle.net/11375/22737}
}

@article{FMSSspherical,
 author = {Fulton, William and MacPherson, Robert and Sottile, F. and Sturmfels, Bernd},
 title = {Intersection theory on spherical varieties},
 fjournal = {Journal of Algebraic Geometry},
 journal = {J. Algebr. Geom.},
 issn = {1056-3911},
 volume = {4},
 number = {1},
 pages = {181--193},
 year = {1995},
 language = {English},
 zbMATH = {729200},
 Zbl = {0819.14019}
}

@article{Brion1997,
author = {Brion, Michel and Vergne, Michèle},
journal = {Journal für die reine und angewandte Mathematik},
pages = {67-92},
title = {An equivariant Riemann-Roch theorem for complete, simplicial toric varieties.},
url = {http://eudml.org/doc/153880},
volume = {482},
year = {1997},
}

@article{Merkurjev1997,
author = {Merkurjev, A. S.},
journal = {Algebra i Analiz},
pages = {175-214},
volume = {9},
issue = {4},
title = {Comparison of the equivariant and the standard $K$-theory of algebraic varieties},
year = {1997},
}

\Addresses

\end{document}